\documentclass[a4paper,reqno,11pt]{amsart}
\usepackage[leqno]{amsmath}
\newfont{\cyr}{wncyr10 scaled 1100}

\usepackage[left=2.7cm,right=2.7cm,top=3.5cm,bottom=3cm]{geometry}
\usepackage{tikz}
\usetikzlibrary{arrows.meta,decorations.markings}
\usepackage{amsthm,amssymb,amsmath,amsfonts,mathrsfs,amscd}
\usepackage{mathtools}
\usepackage{extarrows}
\usepackage[latin1]{inputenc}
\usepackage[all]{xy}
\usepackage{latexsym}
\usepackage{longtable}
\usepackage{xcolor}
\usepackage{tikz-cd}
\usepackage{extarrows}
\usepackage{comment}
\usepackage[shortlabels]{enumitem}
\usepackage{setspace}
\setdisplayskipstretch{}
\usepackage{multirow}
\usepackage{xcolor,colortbl}
\usepackage{changepage}
\usepackage{float}
\usepackage{hyperref}
\numberwithin{equation}{section}
\theoremstyle{plain}
\newtheorem{theorem}{Theorem}[section]
\newtheorem{corollary}[theorem]{Corollary}

\newtheorem{lemma}[theorem]{Lemma}
\newtheorem{proposition}[theorem]{Proposition}
\newtheorem{conjecture}[theorem]{Conjecture}

\theoremstyle{definition}
\newtheorem{definition}[theorem]{Definition}

\newtheorem{examplewr}[theorem]{Example}

\newtheorem{definition/corollary}[theorem]{Definition/Corollary}
\newtheorem{definition/lemma}[theorem]{Definition/Lemma}
\theoremstyle{remark}
\newtheorem{obswr}[theorem]{Observation}
\newtheorem{remarkwr}[theorem]{Remark}

\definecolor{Gray}{gray}{0.85}
\definecolor{LightCyan}{rgb}{0.88,1,1}

\newcolumntype{g}{>{\columncolor{Gray}}c}
\newcolumntype{y}{>{\columncolor{LightCyan}}c}
\newcolumntype{o}{>{\columncolor{pink}}c}

\newenvironment{remark}{\begin{remarkwr}\begin{upshape}}{\end{upshape}\end{remarkwr}}
\newenvironment{example}{\begin{examplewr}\begin{upshape}}{\end{upshape}\end{examplewr}}

\newcommand{\Z}{\mathbb{Z}}
\newcommand{\Q}{\mathbb{Q}}
\newcommand{\R}{\mathbb{R}}
\newcommand{\C}{\mathbb{C}}

\newcommand{\PP}{\mathbb{P}}

\newcommand{\GL}{\mathrm{GL}}

\newcommand{\SL}{\mathrm{SL}}

\newcommand{\Hom}{\mathrm{Hom}}
\newcommand{\res}{\mathrm{res}}

\newcommand{\Div}{\mathrm{Div}}

\newcommand{\red}{\mathrm{red}}
\DeclareMathOperator{\inte}{\mathrm{int}_{(-,\omega)}}
\DeclareMathOperator{\eva}{\mathrm{ev}_{(-,\omega)}}

\newcommand{\too}{\longrightarrow}

\begin{document}
 \include{thebibliography}

\title[Antisymmetry of real quadratic singular moduli]
{Antisymmetry of real quadratic singular moduli}

\author{S\"oren Sprehe}

\address{S.~S.: Bielefeld University, Bielefeld, Germany}
\email{ssprehe@math.uni-bielefeld.de}

\begin{abstract}
We confirm a conjecture of Darmon--Vonk on the antisymmetry of real quadratic singular moduli. The proof relies on a careful analysis of rigid meromorphic cocycles \`a la Darmon--Gehrmann--Lipnowski for the orthogonal group of a quaternion algebra. Moreover, we prove the modularity of a generating series of Kudla--Millson divisors in the spirit of Gross--Kohnen--Zagier.
\end{abstract}

\maketitle

\tableofcontents
\section{Introduction}
Let $\mathcal{H}\subset \mathbb{C}$ be the complex upper half-plane. Recall that $\mathcal{H}$ is endowed with an action of the modular group $\SL_2(\mathbb{Z})$ by M\"obius transformations. By definition, a holomorphic function $f\colon \mathcal{H}\to \mathbb{C}$ is a modular function if and only if $f$ is invariant under the $\SL_2(\mathbb{Z})$-action on $\mathcal{H}$. An example of a modular function is the modular $j$-function $j\colon \mathcal{H}\to \mathbb{C}$ whose Laurent series expansion in the variable $q=\exp(2\pi i z)$ for $z\in \mathcal{H}$ is of the form 
\[
j(z)=q^{-1}+744+196884q+21493760q^2+\cdots.
\]
A singular modulus is then the value of $j$ at a CM point in $\mathcal{H}$, that is, an imaginary quadratic number contained in $\mathcal{H}$. It is well-known that singular moduli are algebraic integers. In fact, one of the fundamental results in the theory of complex multiplication implies that singular moduli are contained in specific abelian extensions of imaginary quadratic fields. Since singular moduli are algebraic, it makes sense to study their norms and traces. In the influential work \cite{GrossZagiersingular}, Gross and Zagier consider the difference of two singular moduli
\[
J_\infty(\tau_1,\tau_2)\coloneq j(\tau_1)-j(\tau_2),
\]
and determine the prime factorization of the absolute norm of this quantity. Simply by its definition, the assignment $J_\infty$ is symmetric, up to sign, in the sense that we have 
\[
J_\infty(\tau_1,\tau_2)=-J_\infty(\tau_2,\tau_1).
\] 
When attempting to naively apply the methods of complex multiplication to construct abelian extensions of real quadratic fields, one quickly encounters a series of seemingly insurmountable challenges. For example, real quadratic irrationalities do not lie in the upper half-plane $\mathcal{H}$, where the $j$-function is defined. However, in their seminal work \cite{DarmonVonksingularmoduli}, Darmon and Vonk use $p$-adic methods to propose a framework in which so-called rigid meromorphic cocycles play the role of the modular $j$-function, and singular moduli are replaced by the values of these objects at special points. \\
We fix a prime number $p\in \mathbb{Z}$ and denote by $\Gamma\coloneq \mathrm{SL}_2(\mathbb{Z}[1/p])$ the Ihara group. A rigid meromorphic cocycle is an element of the first group cohomology of $\Gamma$ with values in the multiplicative group of non-zero rigid meromorphic functions on Drinfeld's upper half-plane $\mathcal{H}_p\coloneq \mathbb{P}^1(\mathbb{C}_p)\setminus \mathbb{P}^1(\mathbb{Q}_p)$ with prescribed divisor. Let $\omega\in \mathcal{H}_p$ be a real quadratic irrationality contained in $\mathcal{H}_p$. In analogy to the notion of CM points, we refer to such points as RM points. Then there exists a distinguished class $\mathcal{D}_\omega$ in the first group cohomology of $\Gamma$ with values in the group of Weil divisors on $\mathcal{H}_p$. The class $\mathcal{D}_\omega$ is supported on the $\Gamma$-orbit of $\omega$ and, thus, should be interpreted as the analogue of the divisor of the function $J_\infty(-,\tau_2)\colon \mathcal{H}\to \mathbb{C}$ for a CM point $\tau_2\in \mathcal{H}$. The divisor of a rigid meromorphic cocycle is then a finite linear combination of classes $\mathcal{D}_\omega$. It is a theorem of Darmon--Vonk (see \cite[Theorem 2.25]{DarmonVonksingularmoduli}) that the classes
$\mathcal{D}_\omega$ never lift to a rigid meromorphic cocycle. However, there is a non-trivial ideal $\mathcal{I}_p$ of the Hecke algebra, related to the annihilator of the space of modular forms of weight $2$ and level $\Gamma_0(p)$, such that for every $\mathrm{T}\in\mathcal{I}_p$ the class
$\mathrm{T}\cdot\mathcal{D}_\omega$ lifts to a rigid meromorphic cocycle $\mathcal{J}_\omega^{\mathrm{T}}$.
All results of this article are stated and proved for these lifts (see Section \ref{section5}).
For the sake of simplicity, we omit the Hecke operator from the notation in this introduction. The reader may replace $\mathcal{J}_\omega$ by $\mathcal{J}_\omega^{\mathrm{T}}$
throughout. The class $\mathcal{J}_\omega$ is unique up to multiplication with an analytic cocycle, that is, a class with values in the group of invertible rigid analytic functions on $\mathcal{H}_p$.
While cohomology classes are not quite functions, one can still make sense of the value of a rigid meromorphic cocycle at an RM point in $\mathcal{H}_p$. Indeed, if $\tau\in \mathcal{H}_p$ is an RM point not contained in the $\Gamma$-orbit of $\omega$, then the stabilizer $\Gamma_\tau$ of $\tau$ in $\Gamma$ is, up to sign, infinite cyclic. There is a distinguished generator $\gamma_\tau\in \Gamma_\tau$ called the automorph of $\tau$, and we consider the $p$-adic number 
\begin{equation}\label{intro1}
\tag{1} J_p(\tau,\omega)\coloneq \mathcal{J}_\omega(\gamma_\tau)(\tau)\in \mathbb{C}_p^\times,
\end{equation}
where we identify $\mathcal{J}_\omega$ with a $1$-cocycle representing it. The group of analytic cocycles contains the subgroup generated by the trivial cocycle, whose value at $\tau$ is an explicit unit $\varepsilon_\tau^{-1}$ in the real quadratic field $\Q(\tau)$ (see Lemma \ref{corvalueatanalyticisunit}). Since this is a finite index subgroup, the quantity $J_p(\tau,\omega)$ is well-defined up to multiplication by an element of
\[
E_\tau=\{x\in\mathbb{C}_p^\times : x^m\in\varepsilon_\tau^{\mathbb{Z}} \text{ for some } m\geq 1\}.
\]
Note that this group contains all roots of unity in $\mathbb{C}_p$.
The function $J_p$ enjoys striking parallels with the difference of two classical singular moduli. Indeed, the $p$-adic number $J_p(\tau,\omega)$ is conjectured to be algebraic and expected to belong to the compositum of specific abelian extensions of the real quadratic fields $\mathbb{Q}(\tau)$ and $\mathbb{Q}(\omega)$. Moreover, the assignment does not depend on the choice of RM point in its $\Gamma$-orbit. Crucially, numerical experiments provide evidence that the $p$-adic number $J_p(\tau,\omega)$ admits an analogous factorization as the difference of two classical singular moduli (see \cite[Conjecture 3.27]{DarmonVonksingularmoduli}). In analogy to the symmetry of the function $J_\infty$, Darmon and Vonk conjecture the following (see \cite[Remark 3.19]{DarmonVonksingularmoduli}).

\medskip\medskip
\noindent
{\bf Conjecture A [Conjecture \ref{mainconjecturetoprove}]}. {\em 
Let $(\tau,\omega)$ be a pair of RM points such that $\tau\notin \Gamma\cdot \omega$. Then
\[
J_p(\tau,\omega)=J_p(\omega,\tau)^{-1}
\]
up to multiplication by an element in $E_\tau\cdot E_\omega$.
}

\medskip

Numerically, the antisymmetry was verified in numerous examples. However, the roles of the RM points $\tau$ and $\omega$ in the construction of $J_p(\tau,\omega)$ are substantially different. Thus, the natural question arises if there is a more \emph{symmetric} approach to define $J_p$. In this article, we construct the function $J_p$ as a symmetric two-variable function defined on the product $\mathcal{H}_p^2$. Using this reinterpretation, Conjecture A follows from the graded commutativity of the cup product. Our approach is motivated by the observation that if $J_p$ exists as a two-variable function, then its divisor should be supported on twisted diagonals of the form
\[
\Delta_{v,p}\coloneq \lbrace (vz,z): z\in \mathcal{H}_p\rbrace\subset \mathcal{H}_p^2
\] 
for $v\in \Gamma$. These cycles arise in the work of Darmon, Gehrmann and Lipnowski \cite{DGL} on rigid meromorphic cocycles for orthogonal groups in the specific case of the $4$-dimensional rational quadratic space $(\mathrm{M}_2(\mathbb{Q}),\mathrm{det})$. Their work produces cohomology classes $\mathcal{D}_n$ for integers $n\geq 1$ that belong to the second group cohomology of $\Gamma^2$ with values in the group $\mathrm{Div}^\dagger_{\mathrm{rq}}(\mathcal{H}_p^2)$ of rational quadratic divisors on $\mathcal{H}_p^2$. This is the group of Weil divisors on $\mathcal{H}_p^2$ which are supported on the divisors $\Delta_{v,p}$ for matrices $v\in \GL_2(\mathbb{Q})$ of positive determinant. We call the classes $\mathcal{D}_n$ \emph{Kudla--Millson divisors} for $\Gamma^2$. In the special case $n=1$, the class $\mathcal{D}_1$ is induced by a cocycle which is supported on the divisors $\Delta_{\gamma,p}$ for $\gamma\in \Gamma$ and, therefore, should be viewed as the $p$-adic analogue of the divisor of the two-variable function $J_\infty$. 
As for the class $\mathcal{D}_\omega$, the class $\mathcal{D}_1$ lifts after applying a Hecke operator
$\mathrm{T}\in\mathcal{I}_p$. We suppress $\mathrm{T}$ from the notation and denote by $\mathcal{J}_1$ a class of $\Gamma^2$ with values in the multiplicative group $\mathcal{M}_{\mathrm{rq}}^\times$ of non-zero rigid meromorphic functions on $\mathcal{H}_p^2$ with rational quadratic divisor, whose associated divisor is $\mathcal{D}_1$.
The class $\mathcal{J}_1$ is uniquely determined by this condition up to multiplication with a class with values in the group of invertible rigid analytic functions on $\mathcal{H}_p^2$. We can meaningfully evaluate $\mathcal{J}_1$ at a pair of RM points $(\tau,\omega)\in \mathcal{H}_p^2$ not contained in the same $\Gamma$-orbit, and we consider the $p$-adic number 
\begin{equation}\label{intro2}
\tag{2} \widehat{J}_p(\tau,\omega)\coloneq \mathcal{J}_1[\tau,\omega]\in \mathbb{C}_p^\times,
\end{equation}
where $\mathcal{J}_1[\tau,\omega]$ is the value of $\mathcal{J}_1$ at $(\tau,\omega)$. This quantity is well-defined up to multiplication with a root of unity. Our central result shows that the functions defined in \eqref{intro1} and \eqref{intro2} coincide:

\medskip\medskip
\noindent
{\bf Theorem B [Theorem \ref{proptoproveconjecture}]}. {\em 
Let $(\tau,\omega)$ be a pair of RM points not belonging to the same $\Gamma$-orbit. Then
\[
J_p(\tau,\omega)^{-2}=\widehat{J}_p(\tau,\omega)
\]
up to multiplication by an element of $E_\tau$.
}

\medskip

The advantage of the function $\widehat{J}_p$ is its antisymmetry: if we denote by $\varsigma_p\colon \mathcal{H}_p^2\to \mathcal{H}_p^2$ the morphism of rigid analytic spaces which switches the coordinates, then $\varsigma_p$ induces an involution on the level of divisors and rigid meromorphic functions on $\mathcal{H}_p^2$, respectively. The identity $\varsigma_p(\Delta_{v,p})=\Delta_{v^{-1},p}$ extends to the observation that the class $\mathcal{D}_1$ remains fixed under this involution (Lemma \ref{symmetryofD1}). Therefore, the image of the class $\mathcal{J}_1$ under the involution on functions equals $\mathcal{J}_1$ up to multiplication with a class with values in the group of invertible rigid analytic functions. This implies the antisymmetry of $\widehat{J}_p$:

\medskip\medskip
\noindent
{\bf Proposition C [Proposition \ref{propforantisymmetry}]}. {\em 
The function $\widehat{J}_p$ satisfies
\[
\widehat{J}_p(\tau,\omega)=\widehat{J}_p(\omega,\tau)^{-1}
\]
up to a root of unity.
}

\medskip
In conclusion, this proposition and Theorem B imply the desired Conjecture A. \\
As another application of the framework of \cite{DGL} for the quadratic space $(\mathrm{M}_2(\mathbb{Q}),\mathrm{det})$, we show that the Kudla--Millson divisors $\mathcal{D}_n$ arise as Fourier coefficients of a modular generating series with values in an appropriate substitute for the Chow group. This result is in the spirit of the emerging $p$-adic counterpart of the Kudla program  \cite{kudlaspecialcyclesand}. We recall the classical situation in a few words. Let $(V,q)$ be a rational quadratic space of real signature $(r,2)$. The real symmetric space associated to $(V,q)$ is endowed with a complex structure, and quotients of the special orthogonal group $\mathrm{SO}_V$ give rise to a family of Shimura varieties. A substantial feature of this family is that they have many algebraic cycles. In fact, there are sub-Shimura varieties of the same type of all codimensions. For example, for every rational vector $v\in V$ of positive length there is a rational quadratic divisor associated to $v$, which induces a sub-Shimura variety $Z(v)$. These objects allow us to define cycles $Z(n)$ for integers $n\geq 0$. It is a theorem of Borcherds \cite{borcherdsgkz} that the generating series made up of the cycles $Z(n)$ is the Fourier expansion of a modular form valued in the first Chow group of an orthogonal Shimura variety. In analogy with those classical modularity results, Darmon, Gehrmann and Lipnowski prove several modularity results for Kudla--Millson divisors for a rational quadratic space of signature $(r,s)$ (see \cite[Section 3.3]{DGL}).
More precisely, they consider the definite case, that is, the case $s=0$, and certain hyperbolic cases where $s=1$.
The proof in the definite case is based on the construction of explicit infinite $p$-adic Borcherds products (see \cite[Proposition 3.26]{DGL}), which are limits of finite products indexed by certain vectors of prescribed length. When $s \neq 0$, the set of such vectors may be infinite, and, in particular, we cannot apply the construction of the Borcherds products for a signature $(2,2)$ space.
In the hyperbolic case, the proof is based on modular symbols paired with intersection numbers. For a signature $(2,2)$ space, the symmetric space $X_\infty$ is $4$-dimensional and not hyperbolic. Thus, the notion of a geodesic connecting two isotropic elements does not make sense in this setting. In our theorem the role of the Chow group is played by the cokernel of the divisor map, that is, we define 
\[
\mathrm{CH}(\Gamma^2)\coloneq \mathrm{H}^2(\Gamma^2,\mathrm{Div}^\dagger_{\mathrm{rq}}(\mathcal{H}_p^2))/\mathrm{div}\left(\mathrm{H}^2(\Gamma^2,\mathcal{M}_{\mathrm{rq}}^\times,)\right)
\]
and denote by $\mathrm{CH}(\Gamma^2)_\mathbb{Q}$ the $\mathbb{Q}$-vector space $\mathrm{CH}(\Gamma^2)\otimes \mathbb{Q}$. We establish the following theorem:

\medskip\medskip
\noindent
{\bf Theorem D [Theorem \ref{modulatitytheorem}]}. {\em 
There is a unique constant term $\mathcal{D}_0\in \mathrm{CH}(\Gamma^2)_\mathbb{Q}$ such that the generating series
\[
\mathcal{D}_0+\sum_{n\geq 1} \mathcal{D}_n\cdot q^n\in \mathrm{CH}(\Gamma^2)_\mathbb{Q}\left[[q]\right]
\]
is modular of weight $2$ and level $\Gamma_0(p)$.
}

\medskip
The proof of this theorem heavily relies on the structure of the space of lifting obstructions of the Kudla--Millson divisors. In the end, this space is connected to the space of modular forms of weight $2$ and level $\Gamma_0(p)$ via Eichler--Shimura theory, and the theorem follows from the Hecke module structure. \\
In the main body of the article, we work with the more general setup of an indefinite quaternion algebra $B/\mathbb{Q}$ which is split at $p$ and a coordinate-free version of the $p$-adic upper half-plane $\mathcal{H}_p$. In the split case $B=\mathrm{M}_2(\mathbb{Q})$, we recover the statements made in the introduction. \\
To summarize, the article is organized as follows. We start by recalling the construction of divisor-valued cocycles of Darmon--Vonk and of Darmon--Gehrmann--Lipnowski, respectively, in Section \ref{section1}. Moreover, we study the specific cases of the $3$- and $4$-dimensional quadratic spaces arising from the quaternion algebra $B/\mathbb{Q}$. In Section \ref{section2} we discuss the notion of Hecke operators in our setup and determine the non-trivial ideal $\mathcal{I}_p$ of the Hecke algebra. Afterwards, we prove the modularity theorem (Theorem \ref{modulatitytheorem}) in Section \ref{modularitxtheoremchapter}. Finally, Sections \ref{section4} and \ref{section5} are devoted to explicitly constructing the functions $J_p$ and $\widehat{J}_p$ and proving the main theorem (Theorem \ref{proptoproveconjecture}).
\subsection*{Acknowledgements}
This work is a generalization of my PhD thesis. I would like to thank my advisor Lennart Gehrmann for his guidance throughout my PhD. This research was supported through the program ``Oberwolfach Research Fellows'' by the Mathematisches Forschungsinstitut Oberwolfach in 2025, and I thank Mike Daas, H\aa vard Damm-Johnsen and Lennart Gehrmann for several inspiring discussions during our stay. I thank Claudia Alfes and H\aa vard Damm-Johnsen for helpful comments on an earlier version of this manuscript. I am grateful to the anonymous referees for their careful reading and for many
valuable comments and suggestions, which significantly improved the exposition of this article. \\ 
The research of the author is funded by the Deutsche Forschungsgemeinschaft (DFG, German Research Foundation) -- SFB-TRR 358/1 2023 -- 491392403.
\section*{Preliminaries and notations}
We will use the following notation and facts throughout the whole article.
\subsubsection*{General notations}
We fix a rational prime $p\in \mathbb{Z}$ and an algebraic closure $\overline{\mathbb{Q}}_p$ of $\mathbb{Q}_p$ once and for all. We normalize the $p$-adic absolute value on $\overline{\mathbb{Q}}_p$ such that $|p|_p=p^{-1}$. Let $\mathbb{C}_p$ be the completion of $\overline{\mathbb{Q}}_p$ with respect to the $p$-adic absolute value. We fix an algebraic closure $\overline{\mathbb{Q}}$ of the rational numbers and fix embeddings $\mathbb{C}\hookleftarrow \overline{\mathbb{Q}}\hookrightarrow \mathbb{C}_p$. Let $\mathcal{H}\subset \C$ be the complex upper half-plane. The space $\mathcal{H}$ is a Riemannian manifold via the hyperbolic metric $ds^2=(dx^2+dy^2)\cdot y^{-2}$. \\
If $V$ is a $\mathbb{Q}$-vector space and $L/\mathbb{Q}$ is a field extension, then $V_L=V\otimes _\mathbb{Q} L$. \\
If $X$ is a topological space and $n\geq 0$ an integer, then we denote by $\mathrm{H}_n(X)$ (respectively by $\widetilde{\mathrm{H}}_n(X)$) the $n$-th (reduced) singular homology group of $X$ with $\Z$ coefficients.
\subsubsection*{Quaternion algebras}
Let $K$ be a field of $\mathrm{char}(K)\neq 2$. A $K$-algebra $B$ is a \emph{quaternion algebra} if $B$ admits a $K$-basis $(1,i,j,ij)$ such that $i^2=a$, $j^2=b$ and  $ij=-ji$ for some $a,b\in K^\times$. A quaternion algebra $B/K$ is equipped with an involution $\overline{(-)}\colon B\to B$ defined by 
\begin{equation*}
\overline{t+b_0}\coloneq t-b_0 \quad \mathrm{for} \quad t\in K, \  b_0\in \mathrm{Span}_K(i,j,ij).
\end{equation*}
The \emph{reduced trace} $\mathrm{Trd}\colon B\to K$ and the \emph{reduced norm} $\mathrm{Nrd}\colon B\to K$ on $B$ are defined by 
\begin{equation*}
\mathrm{Trd}(\alpha)\coloneq \alpha+\overline{\alpha} \quad \mathrm{and} \quad \mathrm{Nrd}(\alpha)\coloneq \alpha\cdot \overline{\alpha} \quad \mathrm{for} \quad \alpha\in B,
\end{equation*}
respectively.
If $\mathcal{O}\subset B$ is a subring containing $1$, then we denote by 
\[
\mathcal{O}_1^\times\coloneq \lbrace b\in \mathcal{O}^\times : \mathrm{Nrd}(b)=1\rbrace\subset B^\times
\]
the group of norm $1$ units in $\mathcal{O}$. \\
If $U\subset B$ is a $K$-linear subspace of $B$, then we denote by $U_0\coloneq\mathrm{ker}\left(\mathrm{Trd}\colon U\to K\right)$ the subspace of reduced trace zero elements in $U$. \\
We endow $B$ with an action of the group $B^\times \times B^\times$ via two-sided multiplication:
\begin{equation*}
\gamma\centerdot v\coloneq \gamma_1\cdot v\cdot \gamma_2^{-1} \quad \mathrm{for} \quad \gamma=(\gamma_1,\gamma_2)\in B^\times \times B^\times, \  v\in B.
\end{equation*}
The group $B^\times$ acts on the space $B_0$ by conjugation:
\begin{equation*}
\gamma \centerdot v\coloneq (\gamma,\gamma)\centerdot v= \gamma\cdot v\cdot \gamma^{-1} \quad \mathrm{for} \quad \gamma\in B^\times, \ v\in B_0.
\end{equation*}
We denote by $(B^\times\times B^\times)_{\mathrm{Nrd}}\subset B^\times\times B^\times$ the group of pairs $\gamma=(\gamma_1,\gamma_2)\in B^\times\times B^\times$ with $\mathrm{Nrd}(\gamma)\coloneq\mathrm{Nrd}(\gamma_1)=\mathrm{Nrd}(\gamma_2)$, and by 
\[
G\coloneq \{\gamma\in B^\times:\mathrm{Nrd}(\gamma)>0\}
\] 
the group of units of positive reduced norm.

\section{Divisor-valued cocycles}\label{section1}
This section recalls the construction of (quaternionic) divisor-valued cocycles introduced by Darmon--Vonk in \cite{DarmonVonksingularmoduli} respectively by Darmon--Gehrmann--Lipnowski in \cite{DGL}. 

\subsection{Symmetric spaces}
Let $(V,q)$ be a non-degenerate rational quadratic space of real signature $(r,s)$, and denote by $(-,-)\colon V\times V\to \mathbb{Q}$ the corresponding symmetric bilinear form. That is, $(v,v)=2\cdot q(v)$ for all $v\in V$. We consider the space
\begin{equation}
X_\infty\coloneq \lbrace Z\subset V_\mathbb{R}  : \mathrm{dim}_\mathbb{R}(Z)=s, \ q|_Z<0\rbrace,
\end{equation}
viewed as a subspace of the Grassmannian of $s$-dimensional subspaces of $V_\mathbb{R}$. 
The dimension of the real manifold $X_\infty$ is $r\cdot s$.
It is a standard fact that the space $X_\infty$ is hermitian if and only if $r=2$ or $s=2$. In this case, the \emph{projective model} associated to $(V,q)$ encodes the complex structure. Denote by $\mathcal{Q}_V/\mathbb{Q}$ the quadric of isotropic lines in $V$. That is:
\[
\mathcal{Q}_V(L)=\lbrace v\in V_L \setminus \lbrace 0\rbrace : q(v)=0\rbrace/L^\times \subset \mathbb{P}_V(L)
\] 
for a field extension $L/\mathbb{Q}$. The variety $\mathcal{Q}_V$ is a closed algebraic subvariety of the projective space $\mathbb{P}_V$, and we denote by $[v]\in \mathcal{Q}_V$ the class of an isotropic vector $v\in V$. Assume that $s=2$. The complex space
\[
\mathcal{K}\coloneq \lbrace [v]\in \mathcal{Q}_V(\mathbb{C}) : (v,v^\sigma)<0\rbrace\subset \mathcal{Q}_V(\mathbb{C}),
\]
where $v^\sigma\in V_\mathbb{C}$ is the vector obtained by complex conjugation, is a complex manifold consisting of two connected components $\mathcal{K}^+$ and $\mathcal{K}^-$. These are interchanged by complex conjugation. 
If we write a vector $v\in V_\mathbb{C}$ representing an element of $\mathcal{K}$ as $v=v_1+i\cdot v_2$ with $v_1,v_2\in V_\mathbb{R}$, then $v_1$ and $v_2$ are of negative length and orthogonal to each other. Thus, we can define a map
\[
\mathcal{K} \too X_\infty, \quad [v_1+i\cdot v_2]\mapsto \mathrm{Span}_\mathbb{R}(v_1,v_2),
\]
which is a well-defined two-to-one map. If we fix once and for all a connected component $\mathcal{K}^+\subset \mathcal{K}$, then we obtain a real analytic isomorphism $\mathcal{K}^+\simeq X_\infty$.
While the orthogonal group $\mathrm{O}_V(\mathbb{R})$ still acts on the whole space $\mathcal{K}$, its action does not preserve the fixed component $\mathcal{K}^+$. In particular, an element $\gamma\in \mathrm{O}_V(\mathbb{R})$ preserves $\mathcal{K}^+$ if and only if its determinant (as an endomorphism on $V_\mathbb{R}$) coincides with its image under the real spinor norm $\mathrm{sn}_\mathbb{R}\colon \mathrm{O}_V(\mathbb{R})\to \mathbb{R}^\times/(\mathbb{R}^\times)^2$. We denote the group of such isometries by $\mathrm{O}_V^+(\mathbb{R})$.  \\
If $[v]\in \mathcal{K}$, then $(w,v)\neq 0$ for all isotropic vectors $w\in V_\mathbb{R}$. Thus, we have an inclusion 
\[
\mathcal{K}\subset \lbrace [v]\in \mathcal{Q}_V(\mathbb{C}) : (v,w)\neq 0 \ \mathrm{for \ all} \ [w]\in \mathcal{Q}_V(\mathbb{R})\rbrace,
\]
which is an equality except in signature $(2,2)$, where the right-hand side consists of four connected components. We note that the right-hand side depends neither on the signature of $(V,q)$ nor on the field extension $\mathbb{C}/\mathbb{R}$. Motivated by this, Darmon--Gehrmann--Lipnowski associate to a non-degenerate rational quadratic space of \emph{arbitrary signature} the $p$-adic space $X_p/\mathbb{C}_p$ defined by
\begin{equation}\label{padicspacedefinition}
X_p\coloneq \lbrace [v]\in \mathcal{Q}_V(\mathbb{C}_p) : (v,w)\neq 0 \ \mathrm{for \ all} \ [w]\in \mathcal{Q}_V(\mathbb{Q}_p)\rbrace.
\end{equation} 
\begin{remark}
We note:
If $V$ is $3$-dimensional, then $X_p=\mathcal{Q}_V(\mathbb{C}_p)\setminus \mathcal{Q}_V(\mathbb{Q}_p)$. Moreover, if $(V,q)$ is of real signature $(1,2)$, then $\mathcal{K}=\mathcal{Q}_V(\mathbb{C})\setminus \mathcal{Q}_V(\mathbb{R})$. Indeed, two isotropic lines in a $3$-dimensional quadratic space are orthogonal to each other if and only if the lines coincide.
\end{remark}
In this article, we are particularly interested in the cases where $(V,q)$ is given by
\begin{enumerate}
\item[\textbullet]an indefinite quaternion algebra $B/\mathbb{Q}$ which is split at $p$ endowed with the reduced norm map $\mathrm{Nrd}\colon B\to \mathbb{Q}$ and
\item[\textbullet] the space $B_0$ endowed with the reduced norm map $\mathrm{Nrd}\colon B_0\to \mathbb{Q}$.
\end{enumerate}
From now on assume that $B/\Q$ is indefinite and split at $p$. We write $B$ and $B_0$ for the quadratic space $(B,\mathrm{Nrd})$ and $(B_0,\mathrm{Nrd})$, respectively. In both situations, the quadratic space is non-degenerate. More precisely, we have 
\[
\mathrm{disc}(B,\mathrm{Nrd})=1 \text{ mod }(\mathbb{Q}^{\times})^2 \quad  \mathrm{and}  \quad \mathrm{disc}(B_0,\mathrm{Nrd})=1 \text{ mod }(\mathbb{Q}^{\times})^2.
\]
Here we are referring to the discriminant of rational quadratic spaces and not the discriminant of a quaternion algebra. 
In particular, this notion is only well defined up to multiplication with non-zero squares in $\Q$. 
Since $B$ is indefinite, the space $B$ is of real signature $(2,2)$, and the space $B_0$ is of real signature $(1,2)$, respectively.
The action of the group $B^\times\times B^\times$ on $B$ by two-sided multiplication gives rise to a sequence 
\[
1\too  \mathbb{Q}^\times \too (B^\times\times B^\times)_{\mathrm{Nrd}}\too \mathrm{SO}_B(\mathbb{Q})\too 1
\]
and, analogously, the action of $B^\times$ on $B_0$ by conjugation gives rise to a sequence
\[
1\too \mathbb{Q}^\times \too B^\times \too \mathrm{SO}_{B_0}(\mathbb{Q})\too 1.
\]
Recall that we denote by $G\subset B^\times$ the subgroup of elements of positive reduced norm.
Note that for $\gamma\in B^\times$, its image under $\mathrm{sn}_\mathbb{R}$ with respect to $B_0$ equals its reduced norm. Thus, there is a natural map $G\to \mathrm{O}_{B_0}^+(\mathbb{R})$, which induces an inclusion 
\[
G/\mathbb{Q}^\times\hookrightarrow \mathrm{O}_{B_0}^+(\mathbb{R})\cap \mathrm{SO}_{B_0}(\mathbb{R}).
\]
Furthermore, the map $G\times G\to  B_\mathbb{R}^\times \times B_\mathbb{R}^\times$ which maps a pair $(\gamma_1,\gamma_2)$ to the pair $\left(\gamma_1,r\cdot \gamma_2\right)$, where $r=(\mathrm{Nrd}(\gamma_1)/\mathrm{Nrd}(\gamma_2))^{1/2}$, induces a natural map
\[
F\colon(G\times G)/\mathbb{Q}^\times\too \mathrm{O}_B^+(\mathbb{R})\cap \mathrm{SO}_B(\mathbb{R}).
\]
If $\gamma\in G\times G$ is an element, then the action of $\gamma$ on the quadric $\mathcal{Q}_B/\mathbb{R}$ coincides with the action of $F(\gamma)$ on it.  \\
By definition, the real symmetric space $X_{\infty,B_0}$ associated to $B_0$ is a subspace of the real symmetric space $X_{\infty,B}$ associated to $B$.
To describe the relation between both spaces and, simultaneously, the relation between the corresponding $p$-adic spaces, we need the following proposition:
\begin{proposition}\label{prosplit}
There is an isomorphism of $\mathbb{Q}$-varieties
\[
\mathcal{Q}_B\too \mathcal{Q}_{B_0}\times \mathcal{Q}_{B_0},
\]
which commutes with the $B^\times \times B^\times$-action on both sides.
\end{proposition}
Basically, the desired splitting is a consequence of the linear algebra lemma below:
\begin{lemma}\label{langeslemmaamanfang}
Let $u\in B$ be a non-zero isotropic vector and  let
\[
\ell_u\colon B\too B, \quad x\mapsto u\cdot x \qquad \mathrm{and} \qquad r_u\colon B\too B, \quad x\mapsto x\cdot u.
\]
be the left- ~and right-multiplication maps. Recall that if $U\subset B$ is a subspace of $B$, then we denote by $U_0$ the subspace of reduced trace zero elements in $U$. 
We have the following:
\begin{enumerate}
\item[(1)]The spaces $\mathrm{ker}(\ell_u)_0$ and $\mathrm{ker}(r_u)_0$ are isotropic lines in $B_0$.
\item[(2)]The vector $u\in B$ has reduced trace zero if and only if $\mathrm{ker}(\ell_u)_0=\mathrm{ker}(r_u)_0=\mathrm{Span}(u)$.
\item[(3)]We have
\begin{equation*}
\mathrm{ker}(\ell_u)=
   \begin{cases}
     \mathrm{ker}(\ell_u)_0\oplus \mathrm{Span}(\overline{u}), & \mathrm{if} \quad \mathrm{Trd}(u)\neq 0 \quad \mathrm{and}  \\
     u\cdot B, & \mathrm{if} \quad \mathrm{Trd}(u)= 0,
     \end{cases}
\end{equation*}
and 
\begin{equation*}
u\cdot B= \mathrm{ker}(r_u)_0 \oplus \mathrm{Span}(u) \quad \mathrm{if} \quad \mathrm{Trd}(u)\neq 0.
\end{equation*}
\item[(4)]If $\mathrm{Trd}(u)\neq 0$, then we have
\[
\mathrm{Span}(1,u)^\perp=\mathrm{ker}(\ell_u)_0\oplus \mathrm{ker}(r_u)_0.
\]
\item[(5)]Let $\gamma=(\gamma_1,\gamma_2)\in B^\times \times B^\times$. Then
\[
\mathrm{ker}(r_{\gamma\centerdot u})_0=\gamma_1\centerdot \mathrm{ker}(r_u)_0 \qquad \mathrm{and} \qquad  \mathrm{ker}(\ell_{\gamma\centerdot u})_0=\gamma_2\centerdot \mathrm{ker}(\ell_u)_0.
\]
\end{enumerate}
\end{lemma}
\begin{proof}
First of all, observe that $\mathrm{ker}(r_u)=\overline{\mathrm{ker}(\ell_{\overline{u}})}$.
This implies:
\[
\mathrm{ker}(r_u)_0=\mathrm{ker}(\ell_{\overline{u}})_0=\lbrace b\in B_0 :  \ell_u(b)=\mathrm{Trd}(u)\cdot b\rbrace.
\]
The second equality follows from the obvious identity $\ell_{\overline{u}}=\mathrm{Trd}(u)\cdot \mathrm{id}-\ell_u$.\\ 
Clearly, $\mathrm{ker}(\ell_u)_0$ must be isotropic and, hence, $\mathrm{ker}(\ell_u)_0$ is either trivial or a line. But it cannot be trivial, because $\ell_u(B_0)$ is (as a totally isotropic space) at most $2$-dimensional. Statement (1) follows. Statement (2) is an immediate consequence of statement (1) and the identity in the second line of this proof.     \\
If $u$ has non-vanishing reduced trace, then $\overline{u}\notin \mathrm{ker}(\ell_u)_0$. Hence, $\mathrm{ker}(\ell_u)_0\oplus \mathrm{Span}(\overline{u})$ is a $2$-dimensional subspace of $\mathrm{ker}(\ell_u)$. However, $\mathrm{ker}(\ell_u)$ is at most $2$-dimensional and we derive the equality of these spaces. Analogously, $u\notin \mathrm{ker}(r_u)_0$, and, therefore, $\mathrm{ker}(r_u)_0\oplus \mathrm{Span}(u)$ is a $2$-dimensional subspace of $\ell_u(B)$, which is, again, at most $2$-dimensional. The first part of statement (3) follows. The second part is obvious from the identity $u^2=0$, which holds if $u$ has vanishing reduced trace and reduced norm.\\
If $u$ has non-vanishing reduced trace, then
\[
\mathrm{Span}(1,u)^\perp=\lbrace b\in B_0 :  \mathrm{Trd}(u\cdot b)=0\rbrace
\]
is $2$-dimensional and contains both $\mathrm{ker}(\ell_u)_0$ and $\mathrm{ker}(r_u)_0$. Statement (2) implies that these lines are different and we derive statement (4). Statement (5) is a straightforward computation and left to the reader.
\end{proof}
We can now prove Proposition \ref{prosplit}:
\begin{proof}[Proof of Proposition \ref{prosplit}:]
We claim that the map 
\begin{equation}
U=[u]\mapsto \left(\mathrm{ker}(r_u)_0,\mathrm{ker}(\ell_u)_0\right)
\end{equation}
is an isomorphism. \\
The injectivity is easily deduced from statement (2) and statement (4) in Lemma \ref{langeslemmaamanfang}.\\
For the surjectivity, consider two isotropic lines $U_1,U_2\subset B_0$. In case $U\coloneq U_1=U_2$, the preimage is given by $U\subset B$. 
We are reduced to consider the case $U_1\neq U_2$. Being different isotropic lines in $B_0$, the lines $U_1$ and $U_2$ cannot be orthogonal to each other. In particular, the orthogonal complement of $U_1\oplus U_2$ in $B_0$ is spanned by a vector $\alpha\in B_0$ satisfying $-\mathrm{Nrd}(\alpha)=\mathrm{disc}(B_0,\mathrm{Nrd})=1 \ \mathrm{mod} \ (\mathbb{Q}^{\times})^2$. Thus, $\alpha^2=-\mathrm{Nrd}(\alpha)=r^2$, for a certain $r\in \mathbb{Q}^\times$. Now, the orthogonal complement $(U_1\oplus U_2)^\perp$ of $U_1\oplus U_2$ in $B$ is spanned by $1$ and $\alpha$ and, hence, contains the isotropic vector $u\coloneq r+\alpha$. This vector is of reduced trace $2\cdot r\neq 0$. But then 
\[
(U_1\oplus U_2)^\perp=\mathrm{Span}(1,u)=\left(\mathrm{ker}(\ell_u)_0\oplus \mathrm{ker}(r_u)_0\right)^\perp
\]
by statement (4). Finally, this implies that either  $U_1=\mathrm{ker}(r_u)_0$ and $U_2=\mathrm{ker}(\ell_u)_0$, or $U_1=\mathrm{ker}(r_{\overline{u}})_0$ and $U_2=\mathrm{ker}(\ell_{\overline{u}})_0$, and the surjectivity follows. The compatibility with the $B^\times \times B^\times$-action on both sides is statement (5).
\end{proof}
Proposition \ref{prosplit} implies immediately:
\begin{corollary}\label{propositionA}
The following hold:
\begin{enumerate}
\item[(1)]There is a real analytic isomorphism 
\[
X_{\infty,B} \ \simeq \  X_{\infty, B_0} \times X_{\infty,B_0},
\]
which commutes with the $G\times G$-action on both sides.
\item[(2)]If $B$ is split at $p$, then there is an isomorphism of rigid analytic varieties
\[
X_{p,B} \ \simeq \ X_{p,B_0}\times X_{p,B_0},
\]
which commutes with the $B^\times \times B^\times$-action on both sides.
\end{enumerate}
\end{corollary}
We remind ourselves that if $K$ is any base field with $\mathrm{char}(K)\neq 2$, then a quaternion algebra $B/K$ is either a division algebra or is isomorphic to the quaternion algebra of $2\times 2$-matrices $\mathrm{M}_2(K)$. In the latter case, the quadric $\mathcal{Q}_{B_0}/K$ of isotropic lines can be identified with the projective line $\mathbb{P}^1$ over $K$. Under this identification, the action of $\GL_2$ on $\mathcal{Q}_{B_0}$ by conjugation corresponds to M\"obius transformations on $\mathbb{P}^1$. Observe that an isomorphism $B \simeq \mathrm{M}_2(K)$ is far from being unique. However, the \emph{Skolem--Noether Theorem} implies that two different isomorphisms differ by conjugation with a matrix in $\GL_2(K)$. Thus, the two identifications of $\mathcal{Q}_{B_0}$ with $\mathbb{P}^1$ differ by a M\"obius transformation with a matrix in $\GL_2(K)$. \\
In conclusion, for the indefinite quaternion algebra $B/\mathbb{Q}$ which is split at $p$ there are non-canonical isomorphisms 
\[
X_{\infty,B_0} \simeq  \mathcal{H} \qquad \mathrm{and} \qquad X_{p,B_0}  \simeq  \mathbb{P}^1(\mathbb{C}_p)\setminus \mathbb{P}^1(\mathbb{Q}_p),
\]
where $\mathbb{P}^1(\mathbb{C}_p)\setminus \mathbb{P}^1(\mathbb{Q}_p)$ is the \emph{Drinfeld upper half-plane}. Motivated by these isomorphisms, we put 
\begin{equation}
\mathcal{H}_\infty\coloneq X_{\infty,B_0} \qquad \mathrm{and} \qquad \mathcal{H}_p\coloneq X_{p,B_0}.
\end{equation}

\subsection{Special cycles and divisor-valued cocycles}
The central objects in the definition of the desired divisor-valued cocycles are certain special cycles on $X_\infty$ and $X_p$ that are defined as follows: ~Let $(V,q)$ be a non-degenerate rational quadratic space of real signature $(r,s)$. If $v\in V_\mathbb{R}$ is a vector of positive length, then we define the special cycle
\begin{equation}\label{deltavinfty}
\Delta_{v,\infty}\coloneq \lbrace Z\in X_\infty : Z\subset \mathrm{Span}_\mathbb{R}(v)^\perp \rbrace.
\end{equation}
In other words, $\Delta_{v,\infty}$ parametrises the maximal negative definite subspaces of the real quadratic space $\mathrm{Span}_\mathbb{R}(v)^\perp$ of signature $(r-1,s)$.
The dimension of $\Delta_{v,\infty}$ is $(r-1)\cdot s$, and it defines a cycle of real codimension $s$ in $X_\infty$.
The orthogonal group $\mathrm{O}_V(\mathbb{R})$ acts on such cycles, and for any $\gamma\in \mathrm{O}_V(\mathbb{R})$ we have
\[
\gamma\cdot \Delta_{v,\infty}=\Delta_{\gamma\cdot v,\infty}.
\]
Later (see the discussion before equation \eqref{definitiondivcocyclesum}), we will define an \emph{intersection number} of certain $s$-chains on the real symmetric space $X_\infty$ with the special cycles $\Delta_{v,\infty}$. 
Therefore, we require certain results on singular homology groups.
Without going into detail, one can endow both spaces $X_{\infty}$ and $\Delta_{v,\infty}$ with an \emph{orientation}, where the latter choice depends on the vector $v$, which determines an identification 
\begin{equation}\label{idenfifictionwithzgeneralcase}
\widetilde{\mathrm{H}}_{s-1}\left(X_\infty\setminus \Delta_{v,\infty}\right)=\mathbb{Z}.
\end{equation}
See \cite[Proposition 2.4]{DGL} and the discussion afterwards. The identification in equation \eqref{idenfifictionwithzgeneralcase} is compatible with the action of the spin group $\mathrm{Spin}_V(\mathbb{R})\coloneq \mathrm{SO}_V(\mathbb{R})\cap \mathrm{O}_V^+(\mathbb{R})$. This means, if $v\in V_\mathbb{R}$ is a vector of positive length and $\gamma\in \mathrm{Spin}_V(\mathbb{R})$ is an isometry of real spinor norm $1$, then the induced map 
\[
\gamma_\star\colon \widetilde{\mathrm{H}}_{s-1}\left(X_\infty\setminus \Delta_{v,\infty}\right)\too \widetilde{\mathrm{H}}_{s-1}\left(X_\infty\setminus \Delta_{\gamma \cdot v,\infty}\right)
\]
is the identity after identifying both groups with $\mathbb{Z}$. See \cite[Lemma 2.5]{DGL}.\\
Recall that if $s=2$, then we identified the symmetric space $X_\infty$ with the complex space $\mathcal{K}^+$. In terms of this model, the special cycle $\Delta_{v,\infty}$ is identified with the set $\lbrace [u]\in \mathcal{K}^+ : (u,v)=0 \rbrace$. By analogy, Darmon--Gehrmann--Lipnowski define for an anisotropic vector $v\in V_{\mathbb{Q}_p}$ the $p$-adic cycle
\begin{equation}
\Delta_{v,p}\coloneq\lbrace [u]\in X_p :  (u,v)=0 \rbrace\subset X_p.
\end{equation}
By definition, $\Delta_{v,p}$ has codimension $1$ in $X_p$, and, therefore, is a divisor on $X_p$. 
The orthogonal group $\mathrm{O}_V(\mathbb{Q}_p)$ acts on such divisors, and for any $\gamma\in \mathrm{O}_V(\mathbb{Q}_p)$ we find that
\[
\gamma\cdot \Delta_{v,p}=\Delta_{\gamma\cdot v,p}.
\]
Recall that a \emph{Weil divisor} on the rigid analytic space $X_p$ is a (possibly infinite) formal sum $D=\sum_{i\in I} n_i\cdot D_i$ of prime divisors $D_i$, that is, irreducible closed analytic subvarieties of codimension $1$, with $n_i\in \Z$, which is \emph{locally finite}.
This means, if $U\subset X_p$ is an affinoid subvariety, then the set $\{ D_i : U\cap D_i\neq \emptyset  \text{  and  }  n_i\neq 0\}$ is finite. 
We denote the abelian group of Weil divisors on $X_p$ by $\Div^\dagger(X_p)$. 
Let $V_+\subset V$ be the set of vectors of positive length.
The elements of the group
\begin{equation}
\mathrm{Div}^\dagger_{\mathrm{rq}}(X_p)\coloneq \bigg\{ \mathcal{D} \in  \mathrm{Div}^\dagger(X_p) : \mathcal{D}=\sum_{v\in V_+} n_v\cdot \Delta_{v,p}\bigg\}
\end{equation}
are called \emph{rational quadratic divisors}.
The group $\mathrm{Div}^\dagger_{\mathrm{rq}}(X_p)$ is an $\mathrm{O}_V(\mathbb{Q})$-module, and, therefore, we can consider its cohomology groups.
Fix a $\mathbb{Z}[1/p]$-lattice $L\subset V$ on which the quadratic form $q$ is $\mathbb{Z}[1/p]$-valued, and consider a $p$-arithmetic subgroup $\Gamma\subset \mathrm{SO}(L)$ that is contained in the spin group $\mathrm{Spin}_V(\mathbb{R})$. 
The desired divisor-valued cocycles are defined as formal sums of rational quadratic divisors $\Delta_{v,p}$.
To ensure that these sums are locally finite, one must weight them by intersection numbers between $s$-chains on the real symmetric space $X_\infty$ and the special cycles $\Delta_{v,\infty}$.
However, not all $s$-chains on $X_\infty$ intersect the cycles $\Delta_{v,\infty}$ \emph{transversally}. Since the boundary of a chain may lie on $\Delta_{v,\infty}$, we cannot define their intersection number. 
To overcome this, Darmon--Gehrmann--Lipnowski construct the complex $\mathscr{C}_\bullet(X_\infty)$. By definition, it is the maximal subcomplex of the singular chain complex $C_\bullet(X_\infty)$ whose $s$-chains intersect all special cycles $\Delta_{v,\infty}$ transversally:
\begin{equation}\label{complexgeneraldef}
\mathscr{C}_n(X_\infty)\coloneq
   \begin{cases}
     C_n(X_\infty\setminus\bigcup_{w\in V_+} \Delta_{w,\infty}) , & \mathrm{if} \ n<s, \\
     \lbrace c\in C_s(X_\infty) : \partial_s(c)\in \mathscr{C}_{s-1}(X_\infty)\rbrace, & \mathrm{if} \ n=s \ \mathrm{and} \\
     C_n(X_\infty), & \mathrm{if} \ n\geq s+1.
     \end{cases}
\end{equation}
The complex $\mathscr{C}_\bullet(X_\infty)$ is a complex of $\mathbb{Z}[\Gamma]$-modules. 
Moreover, the inclusion $\mathscr{C}_\bullet(X_\infty)\hookrightarrow C_\bullet(X_\infty)$ is a quasi-isomorphism. Thus, $\mathscr{C}_\bullet(X_\infty)$ is a resolution of $\mathbb{Z}$ over $\mathbb{Z}[\Gamma]$ (see \cite[Proposition 2.17]{DGL}). \\
If $c\in \mathscr{C}_s(X_\infty)$ is an $s$-chain, the boundary $\partial_s(c)$ is an $(s-1)$-cycle on the complement $X_\infty\setminus \Delta_{w,\infty}$ for all $w\in V_+$.
We define the signed intersection number of $c$ with $\Delta_{w,\infty}$ as the image of $\partial_s(c)$ in the reduced homology group $\widetilde{\mathrm{H}}_{s-1}(X_\infty\setminus \Delta_{w,\infty})$, which we identified with $\mathbb{Z}$:
\[
c\cap \Delta_{w,\infty}\coloneq [\partial_s(c)]\in \widetilde{\mathrm{H}}_{s-1}(X_\infty\setminus \Delta_{w,\infty}).
\]
Finally, if $v\in V_+$, then we can consider the map
\begin{equation}\label{definitiondivcocyclesum}
\mathcal{D}_v\colon \mathscr{C}_s(X_\infty)\too \prod_{w\in V_+} \mathbb{Z}\cdot \Delta_{w,p}, \qquad c\mapsto \sum_{w\in \Gamma v} (c\cap \Delta_{w,\infty})\cdot \Delta_{w,p}.
\end{equation} 
The map $\mathcal{D}_v$ takes values in the group of Weil divisors on $X_p$ (see \cite[Lemma 2.14]{DGL}) and is clearly a homomorphism of groups. 
Moreover, using the compatibility relation
\[
(\gamma c)\cap \Delta_{\gamma w,\infty}=c \cap \Delta_{w,\infty}
\]
for all $c\in \mathscr{C}_s(X_\infty)$, $w\in V_+$, and $\gamma \in \Gamma$ (see \cite[Lemma 2.5]{DGL}), we see that $\mathcal{D}_v$ is $\Gamma$-equivariant. Finally, since $\partial_{s}\circ\partial_{s+1}=0$, the map $\mathcal{D}_v$ is an $s$-cocycle, thus defines a class in $\mathrm{H}^s\left(\Hom_\Gamma(\mathscr{C}_\bullet(X_\infty),\mathrm{Div}^\dagger_{\mathrm{rq}}(X_p))\right)$.
By its definition, $\mathcal{D}_v$ is independent of the choice of representative of the $\Gamma$-orbit of $v$. \\
The resolution $\mathscr{C}_\bullet(X_\infty)$ is not a projective resolution of $\mathbb{Z}$ over $\mathbb{Z}[\Gamma]$ in general.
However, the \emph{fundamental lemma of homological algebra} (\cite[Lemma 7.4]{brown}) implies that for any projective resolution $P_\bullet$ of $\mathbb{Z}$ over $\mathbb{Z}[\Gamma]$, there exists a quasi-isomorphism $p_\bullet\colon P_\bullet\to \mathscr{C}_\bullet(X_\infty)$ extending the augmentation map.
This quasi-isomorphism is unique up to homotopy, but not explicit.
Therefore, we obtain a canonical group homomorphism
\[
\mathrm{H}^s\left(\Hom_\Gamma(\mathscr{C}_\bullet(X_\infty),\mathrm{Div}^\dagger_{\mathrm{rq}}(X_p))\right)\too \mathrm{H}^s\left(\Gamma,\mathrm{Div}^\dagger_{\mathrm{rq}}(X_p)\right),
\]
and we denote the image of $[\mathcal{D}_v]$ by the same symbol (with a slight abuse of notation):
\[
[\mathcal{D}_v]\in \mathrm{H}^s\left(\Gamma,\mathrm{Div}^\dagger_{\mathrm{rq}}(X_p)\right).
\]

\subsubsection{The $3$-dimensional case}\label{sec:3-dim}
Recall that we let $B$ be an indefinite quaternion algebra which is split at $p$.
The $p$-adic cycles $\Delta_{v,p}\subset \mathcal{H}_p$ can be described as follows. Let $v\in B_0$ be anisotropic and let $K_v/\mathbb{Q}$ denote the splitting field of the polynomial $X^2+\mathrm{Nrd}(v)\in \mathbb{Q}[X]$. Then $\Delta_{v,p}$ is empty if and only if $K_v=\mathbb{Q}$ or $p$ is split in $K_v$. Otherwise, if $p$ is inert or ramified in the quadratic field $K_v$, then $\Delta_{v,p}=\omega+\omega^\prime$ where $\omega,\omega^\prime\in \mathcal{Q}_{B_0}(K_v)\setminus \mathcal{Q}_{B_0}(\mathbb{Q})$ are certain points interchanged by the non-trivial element in the Galois group $\mathrm{Gal}(K_v/\mathbb{Q})$. We call the points that belong to a special cycle $\Delta_{v,p}$ for an anisotropic vector $v\in B_0$ \emph{quadratic points}. Moreover, we say that $\omega$ is an \emph{RM point} if the associated quadratic field $K_v$ is real quadratic. It is not hard to show that a point $\omega\in \mathcal{H}_{p}$ is a quadratic point if and only if $\omega$ is defined over a quadratic field $K/\mathbb{Q}$ in which $p$ is non-split.  \\
The action of the special orthogonal group $\mathrm{SO}_{B_0}(\mathbb{Q})\simeq B^\times/\mathbb{Q}^\times$ on $\mathcal{H}_{p}$ preserves quadratic points and, thus, it makes sense to study the stabilizer of such points under this action: ~Let $K/\mathbb{Q}$ be a quadratic field which splits $B$, and let $\omega\in \mathcal{Q}_{B_0}(K)$ be a point not defined over $\mathbb{Q}$. A unit $\gamma\in B^\times$ acts on the line $\omega$ in two possible ways:
\begin{enumerate}
\item[\textbullet]we can consider the line $\gamma\centerdot\omega$ where $\gamma$ acts by conjugation and
\item[\textbullet]the line $\gamma\cdot \omega$ where $\gamma$ acts via the multiplication in $B_K$.
\end{enumerate}
The next lemma shows that the stabilizer of $\omega$ is independent of the actions:
\begin{lemma}\label{lem:stabilizer}
Let $\gamma\in B^\times$ be a unit. Then $\gamma\centerdot\omega=\omega$ if and only if $\gamma\cdot \omega=\omega$.
\end{lemma}
\begin{proof}
Since $\omega$ is of reduced trace zero, we have $\overline{\omega}=\omega$. Hence, if $\gamma\cdot \omega=\omega$, then
\[
\gamma \centerdot \omega=\gamma\cdot\omega\cdot \gamma^{-1}=\gamma\cdot \omega\cdot \overline{\gamma}=\gamma\cdot \overline{\gamma\cdot \omega}=\omega.
\]
Conversely, if $\gamma\centerdot\omega=\omega$, then $\gamma\cdot \omega=\omega\cdot \gamma$. Let $u\in (B_0)_K$ be an isotropic vector that spans the line $\omega$, and consider the left- ~and right-multiplication map $\ell_{\gamma\cdot u}\colon B\to B$ and $ r_{\gamma\cdot u}\colon B\to B$, respectively. Since $\gamma$ is a unit in $B$ and $\omega$ is of reduced trace zero, we have
\[
\mathrm{ker}(\ell_{\gamma\cdot u})_0=\mathrm{ker}(\ell_u)_0=\omega \qquad   \mathrm{and} \qquad \mathrm{ker}(r_{\gamma\cdot u})_0=\mathrm{ker}(r_{u\cdot \gamma})_0=\mathrm{ker}(r_u)_0=\omega,
\]
where the first equality in the second equation follows from the assumption $\gamma\cdot \omega=\omega\cdot \gamma$. Thus, $\gamma\cdot \omega$ is of reduced trace zero and $\gamma\cdot \omega=\mathrm{ker}(\ell_{\gamma\cdot u})_0=\omega$ as desired.
\end{proof}
We relate the stabilizer of $\omega$ in $B^\times$ to the quadratic field $K$:
Let $u\in (B_0)_K$ be an isotropic vector which spans the line $\omega$.
The ring
\begin{equation}
K_\omega\coloneq \lbrace \gamma\in B  :  \gamma\cdot u=\lambda_\gamma\cdot u \ \mathrm{for} \ \lambda_\gamma\in K \rbrace
\end{equation}
does not depend on the representative $u$.
By the lemma above, we see that $K_\omega^\times$ is precisely the stabilizer of $\omega$ in $B^\times$.
Define the ring homomorphism
\[
\Psi_\omega \colon K_\omega\too K, \quad  \gamma \mapsto \lambda_\gamma.
\]
Clearly, $\Q\subset K_\omega$, and for any $x\in \Q$ we have $\Psi_\omega(x)=x$.
The map $\Psi_\omega$ is injective: ~suppose there is a non-zero element $\gamma\in K_\omega$ with $\gamma\cdot u=0$. Then $\gamma$ must be isotropic and $u\in \mathrm{ker}(\ell_\gamma)_0$. However, this line is defined over $\mathbb{Q}$ because $\gamma$ is defined over $\mathbb{Q}$, contradicting the assumption that $\omega$ is a quadratic point.
The above yields the following commutative diagram
\begin{equation*}
\begin{tikzcd}
K_\omega \arrow[rr, hook, "\Psi_\omega"] &  &K\\
& \Q \arrow[ur, hook] \arrow[ul, hook].
\end{tikzcd}
\end{equation*}
It follows that $K_\omega$ equals $\Q$ or the field $K$. However, in the proof of \cite[Lemma 5.4.7]{Voight}, Voight constructs an element $\gamma\in K_\omega$ which is not an element of $\mathbb{Q}$.
Thus, $\Psi_\omega$ is an isomorphism between $K_\omega$ and $K$.
We claim that it satisfies the compatibility relation \[\mathrm{Nm}_{K/\Q}(\Psi_\omega(\gamma))=\mathrm{Nrd}(\gamma) \quad \text{for all }  \gamma \in K_\omega.
\]
If $\gamma\in \Q$, then this is obvious as both terms equal $\gamma^2$.
Assume that $\gamma\in K_\omega\setminus \Q$.
Since $\gamma$ is a root of the reduced characteristic polynomial $F_\gamma(X)\coloneq X^2-\mathrm{Trd}(\gamma)X+\mathrm{Nrd}(\gamma)\in \Q[X]$,
the eigenvalues of the left-multiplication map $\ell_\gamma\colon B_K\to B_K$ are the zeros of $F_\gamma$. 
But, $\Psi_\omega(\gamma)\in K$ is an eigenvalue of $\ell_\gamma$, and the assumption $\gamma\notin \Q$ implies that $\Psi_\omega(\gamma)\notin \Q$.
Therefore, $\Psi_\omega(\gamma)$ and the Galois conjugate of $\Psi_\omega(\gamma)$ are the roots of $F_\gamma$. Thus, $\mathrm{Nm}_{K/\Q}(\Psi_\omega(\gamma))=F_\gamma(0)=\mathrm{Nrd}(\gamma)$ as claimed. \\
Denote by $R\subset B$ the (up to conjugation) unique maximal order in $B$, and by $\Gamma\coloneq R[1/p]^\times_1$ the group of norm $1$ units in $R[1/p]$.
The map $\Psi_\omega$ induces an isomorphism between the stabilizer $\Gamma_\omega$ of $\omega$ in $\Gamma$ and the group of norm $1$ units in a $\mathbb{Z}[1/p]$-order in the real quadratic field $K$.
Such a group is, up to $2$-torsion, infinite cyclic because $p$ is non-split in $K$, and we deduce that 
\begin{equation}
\Gamma_\omega\simeq \mathbb{Z}  \oplus \mathbb{Z}/2\mathbb{Z}.
\end{equation}
There is precisely one generator $\gamma_\omega\in \Gamma_\omega$ whose associated unit $\varepsilon_\omega\coloneq \Psi_\omega(\gamma_\omega)\in K^\times$ belongs to the interval $(1,\infty)$ for a fixed embedding $K\hookrightarrow \R$. We call $\gamma_\omega$ the \emph{automorph of $\omega$}.\\
If $\omega\in \mathcal{H}_{p}$ is an RM point, then both $\omega$ and its Galois conjugate $\omega^\prime$ also define elements of $\mathcal{Q}_{B_0}(\mathbb{R})$ and, thus, can be interpreted as boundary points of $\mathcal{H}_{\infty}\subset \mathcal{Q}_{B_0}(\mathbb{C})\setminus \mathcal{Q}_{B_0}(\mathbb{R})$. We denote by $\Delta_{\omega,\infty}\subset \mathcal{H}_\infty$ the unique geodesic on $\mathcal{H}_\infty$ which connects the boundary points $\omega^\prime$ and $\omega$. Furthermore, we endow $\Delta_{\omega,\infty}$ with an orientation by saying that it starts in $\omega^\prime$ and ends in $\omega$. If we consider the rational quadratic space $(B_0,-\mathrm{Nrd})$ of real signature $(2,1)$, then its associated real symmetric space $X_{\infty}$ is canonically identified with the real symmetric space $\mathcal{H}_\infty$. Moreover, its $p$-adic space $X_p$ equals the space $\mathcal{H}_{p}$. In this setup, any RM point can be identified with a special cycle on $X_\infty$, and under the identification of $X_\infty$ with $\mathcal{H}_\infty$ this cycle equals $\Delta_{\omega,\infty}$. In other words, $\Delta_{\omega,\infty}$ \emph{is} a special cycle on $\mathcal{H}_\infty$, but for the corresponding signature $(2,1)$ space. We remind ourselves that we denote by $G\subset B^\times$ the subgroup of elements of positive reduced norm. The action of $G$ on $\mathcal{H}_\infty$ is orientation preserving. Therefore, for any $\gamma\in G$ we have the equality of \emph{oriented} cycles
\[
\gamma\centerdot \Delta_{\omega,\infty}=\Delta_{\gamma\centerdot \omega,\infty}.
\]
The oriented geodesic $\Delta_{\omega,\infty}$ splits $\mathcal{H}_\infty$ into two connected components, i.e. ~$\mathcal{H}_\infty\setminus \Delta_{\omega,\infty}=X^+\sqcup X^-$ where the orientation of $\Delta_{\omega,\infty}$ determines the choice of $X^+$ and $X^-$. Any two different points in $X^-$ (resp. ~in $X^+$) can be connected by a geodesic segment in $X^-$ (resp. ~$X^+$). Thus, the homology groups $\mathrm{H}_0(X^-)$ and $\mathrm{H}_0(X^+)$ are infinite cyclic, each generated by the homology class induced from a point.
Hence, they can be canonically identified with $\mathbb{Z}$. We conclude that the reduced homology group $\widetilde{\mathrm{H}}_0(\mathcal{H}_\infty\setminus \Delta_{\omega,\infty})$, which is the kernel of the augmentation map
\[
\mathrm{H}_0(X^-)\oplus \mathrm{H}_0(X^+)\too \mathbb{Z}, \qquad \left(n_0, n_1\right)\mapsto n_0+n_1,
\]
is infinite cyclic. Once and for all we fix an orientation on $\mathcal{H}_\infty$, and we deduce an identification 
\begin{equation}\label{identificarionreducedhomology}
\widetilde{\mathrm{H}}_0(\mathcal{H}_\infty\setminus \Delta_{\omega,\infty})=\mathbb{Z}.
\end{equation}
From now on, we will denote by $\mathscr{C}_\bullet(\mathcal{H}_\infty)$ the subcomplex of the singular chain complex $C_\bullet(\mathcal{H}_\infty)$ defined in equation \eqref{complexgeneraldef} for the rational quadratic space $(B_0,-\mathrm{Nrd})$. That is, the special cycles appearing in this definition are the geodesics $\Delta_{\omega,\infty}\subset \mathcal{H}_\infty$ for RM points $\omega\in \mathcal{H}_{p}$. Let $\omega\in \mathcal{H}_{p}$ be an RM point defined over the real quadratic field $K$. Then we can associate to $\omega$ the divisor-valued cocycle
\begin{equation}
\mathcal{D}_\omega\colon \mathscr{C}_1(\mathcal{H}_\infty)\too \mathrm{Div}^\dagger\left(\mathcal{H}_{p}\right), \qquad c\mapsto \sum_{\tau\in \Gamma\centerdot \omega} (c\cap \Delta_{\tau,\infty})\cdot \tau.
\end{equation}
While the group $\Gamma$ is not a $p$-arithmetic subgroup of $\mathrm{SO}(R[1/p])$, its image under the natural map $B^\times \to \mathrm{SO}_{B_0}(\mathbb{Q})$ is such a subgroup. Thus, we can still apply the construction of \cite{DGL}. In the split case $B=\mathrm{M}_2(\mathbb{Q})$, we can take $\Gamma$ to be the Ihara group $\SL_2(\mathbb{Z}[1/p])$, and recover the divisor-valued cocycle attached to $\omega$ introduced in \cite[Equation 27]{DarmonVonksingularmoduli}.
\begin{definition}
We denote by $\mathcal{RM}(\Gamma)\subset \mathrm{H}^1(\Gamma,\mathrm{Div}^\dagger(\mathcal{H}_p))$ the subgroup generated by the classes $[\mathcal{D}_\omega]$ for RM points $\omega\in \mathcal{H}_p$.
\end{definition} 
\begin{proposition}\label{prop:class}
Let $\mathcal{H}_p^\mathrm{RM}$ denote the subset of RM points.
The classes $\{ [\mathcal{D}_\omega]\}_{\omega\in \Gamma\backslash \mathcal{H}_p^{\mathrm{RM}}}$ are non-zero and linearly independent. 
In particular, 
\[
\mathcal{RM}(\Gamma)=\bigoplus_{\omega\in \Gamma\backslash \mathcal{H}_p^{\mathrm{RM}}} \mathbb{Z}\cdot [\mathcal{D}_\omega].
\]
\end{proposition}
\begin{proof}
If $B$ is split, the proposition follows from the fact that the classes $[\mathcal{D}_\omega]$
lift to divisor-valued, $\Gamma$-invariant modular symbols. Non-triviality then follows from the arguments below \cite[Equality 18]{DarmonVonksingularmoduli}, and linear independence is ensured by the disjoint support of classes associated to different  $\Gamma$-orbits.\\
If $B$ is non-split, \cite[Proposition 8]{GehrmannquaternionicRC} implies that the canonical map 
\[
\bigoplus_{\omega\in \Gamma\backslash \mathcal{H}_p} \mathrm{H}^1(\Gamma,\mathrm{Div}^\dagger(\Gamma\centerdot \omega))\too \mathrm{H}^1(\Gamma,\mathrm{Div}^\dagger(\mathcal{H}_p))
\]
is an isomorphism.
Thus, it suffices to show that the classes $[\mathcal{D}_\omega]$ are non-zero in $\mathrm{H}^1(\Gamma,\mathrm{Div}^\dagger(\Gamma\centerdot \omega))$.
By Lemma 1 of \emph{loc.cit.}, there exists an arithmetic subgroup $\Gamma_\circ\subset \Gamma$ (depending on $\omega$) and a $\Gamma$-module isomorphism
\[
\mathrm{Div}^\dagger(\Gamma\centerdot \omega)\too \mathrm{Coind}^\Gamma_{\Gamma_{\circ}}(\mathbb{Z}[\Gamma_\circ \centerdot \omega]).
\] 
Shapiro's Lemma then yields an isomorphism \[
\mathrm{H}^1(\Gamma,\mathrm{Div}^\dagger(\Gamma\centerdot \omega))\too \mathrm{H}^1(\Gamma_\circ,\Z[\Gamma_\circ\centerdot \omega]),
\] 
and the class $[\mathcal{D}_\omega]$ maps to the class induced by the $1$-cocycle
\[
\mathcal{D}_{\Gamma_\circ \centerdot \omega}\colon \mathscr{C}_1(\mathcal{H}_\infty) \too \Z[\Gamma_\circ\centerdot \omega], \qquad c\mapsto \sum_{\tau\in \Gamma_\circ \centerdot \omega} (c\cap \Delta_{\tau,\infty})\cdot \tau.
\]
Finally, the proof of \cite[Lemma 2.6]{GehrmannSprehe} can be applied to show that $\mathrm{H}^1(\Gamma_\circ,\Z[\Gamma_\circ\centerdot \omega])$ is infinite cyclic, and Lemma 2.12 of \emph{loc.cit.} implies that the class induced by $\mathcal{D}_{\Gamma_\circ \centerdot \omega}$ is a generator.
\end{proof}

\subsubsection{The $4$-dimensional case}\label{sectionfourdimensionalcase}
Recall that we established isomorphisms
\begin{equation*}
X_{\infty,B} \simeq  \mathcal{H}_\infty\times \mathcal{H}_\infty \qquad  \mathrm{and} \qquad  X_{p,B}  \simeq  \mathcal{H}_{p}\times \mathcal{H}_{p}.
\end{equation*} 
We compute the special cycles $\Delta_{v,\infty}$ and $\Delta_{v,p}$ under these identifications:
\begin{lemma}\label{computationdeltavp}
Let $v\in B$ be a vector of positive length. Then
\[
\Delta_{v,\infty}=\lbrace (v\centerdot z,z) :  z\in \mathcal{H}_{\infty}\rbrace\subset \mathcal{H}_{\infty}\times \mathcal{H}_{\infty},
\]
and
\[
\Delta_{v,p}=\lbrace (v\centerdot z,z) : z\in \mathcal{H}_{p}\rbrace\subset \mathcal{H}_{p}\times \mathcal{H}_{p}.
\]
\end{lemma}
\begin{proof}
The compatibility of the splitting of $\mathcal{Q}_B$ into $\mathcal{Q}_{B_0}\times \mathcal{Q}_{B_0}$ with the $B^\times \times B^\times$-action on both sides together with the identities
\[
(v,1)\centerdot \Delta_{1,\infty}=\Delta_{v,\infty} \qquad \mathrm{and} \qquad (v,1)\centerdot \Delta_{1,p}=\Delta_{v,p}
\]
imply that it suffices to consider the case $v=1$. Recall that
\[
\Delta_{1,\infty}=\lbrace [u]\in X_{\infty,B} : (u,1)=0\rbrace.
\]
But, $(u,1)=\mathrm{Trd}(u)$, which implies that $\Delta_{1,\infty}$ consists of the reduced trace zero lines in $X_\infty$. Such a line is mapped to the pair $([u],[u])$ and the desired identification follows. The statement for the $p$-adic special cycles follows completely analogously. 
\end{proof}
\begin{lemma}\label{lem:idenfificationwithunit}
Let $v\in B$ be a vector of positive length. Then 
\begin{equation*}
\widetilde{\mathrm{H}}_n\left(\mathcal{H}_\infty^2\setminus \Delta_{v,\infty}\right) \simeq 
   \begin{cases}
     \mathbb{Z}, & \mathrm{if} \ n=1 \ \mathrm{and}  \\
     0, & \mathrm{otherwise.} 
     \end{cases}
\end{equation*}
\end{lemma}
\begin{proof}
First of all, the image of the special cycle $\Delta_{1,\infty}$ under the map $(v,1)\colon \mathcal{H}_\infty^2\to \mathcal{H}_\infty^2$ equals $\Delta_{v,\infty}$, and we deduce isomorphisms
\[
(v,1)_\star \colon \widetilde{\mathrm{H}}_n\left(\mathcal{H}_\infty^2\setminus \Delta_{1,\infty}\right)\too \widetilde{\mathrm{H}}_n\left(\mathcal{H}_\infty^2\setminus \Delta_{v,\infty}\right)
\]
for all integers $n\geq 0$. Fix a point $P=(\tau,\tau)\in \Delta_{1,\infty}$ and consider the exponential map $\exp_\tau\colon \mathrm{T}_\tau(\mathcal{H}_\infty)\to \mathcal{H}_\infty$ at $\tau$. Here, the tangent space $W\coloneq \mathrm{T}_\tau(\mathcal{H}_\infty)$ is a $2$-dimensional real vector space. The exponential map $\exp_P=\exp_\tau\times \exp_\tau$ at $P$ induces a diffeomorphism 
\[
\exp_P\colon W^2\setminus \Delta\left(W\right)\too \mathcal{H}_\infty^2\setminus \Delta_{1,\infty},
\] 
and we derive isomorphisms
\[
\widetilde{\mathrm{H}}_n\left(W^2\setminus \Delta\left(W\right)\right)\xrightarrow{\simeq} \widetilde{\mathrm{H}}_n\left(\mathcal{H}_\infty^2\setminus \Delta_{1,\infty}\right)
\]
for all integers $n\geq 0$. The map
\begin{equation}\label{isostep27}
W^2\setminus \Delta(W)\too W\setminus \lbrace 0\rbrace \times W, \qquad (w_1,w_2)\mapsto (w_1-w_2,w_2)
\end{equation}
is a homeomorphism and, as a real vector space, $W$ is homotopy equivalent to a point. Thus,
\[
\widetilde{\mathrm{H}}_n\left(W^2\setminus \Delta(W)\right)\simeq \widetilde{\mathrm{H}}_n\left(W\setminus \lbrace 0\rbrace\right)
\]
for all integers $n\geq 0$. Finally, the punctured space $W\setminus \lbrace 0\rbrace$ is homotopy equivalent to the $1$-sphere $\mathbb{S}^1$ whose homology groups are the ones in the lemma.
\end{proof}
The fixed orientation on $\mathcal{H}_\infty$ determines an orientation on the tangent space $W$. This means, it determines whether a basis $(v_1,v_2)$ of $W$ is positively or negatively oriented. 
If $(v_1,v_2)$ is a positively oriented basis of $W$, then the isomorphism of real vector spaces $W \to \R^2$ which maps $(v_1,v_2)$ to the standard basis $(e_1,e_2)$ induces an isomorphism in homology $\mathrm{H}_1(W\setminus \{0\}) \simeq \mathrm{H}_1(\R^2\setminus \{0\})$, which does not depend on the basis $(v_1,v_2)$.
We orient the unit circle clockwise, giving an identification 
\begin{equation}\label{eq:idenfitication2}
\mathrm{H}_1(W\setminus \{0\}) = \mathrm{H}_1(\R^2\setminus \{0\})=\Z.
\end{equation}
For later reference, we prove the following symmetry observation:
\begin{lemma}\label{phicomputation}
Consider the continuous map $\varsigma_\infty\colon \mathcal{H}_\infty^2\too \mathcal{H}_\infty^2$ defined by $(z_1,z_2)\mapsto (z_2,z_1)$.
The map $\varsigma_\infty$ induces a chain map
\[
\varsigma_\infty\colon\mathscr{C}_\bullet(\mathcal{H}_\infty^2)\too \mathscr{C}_\bullet(\mathcal{H}_\infty^2).
\]
Moreover, $\varsigma_\infty$ preserves the special cycle $\Delta_{1,\infty}$ and the induced map
\[
(\varsigma_\infty)_\star\colon \mathrm{H}_1\left(\mathcal{H}_\infty^2\setminus \Delta_{1,\infty}\right)\too \mathrm{H}_1\left(\mathcal{H}_\infty^2\setminus \Delta_{1,\infty}\right)
\]
is the identity.
\end{lemma}
\begin{proof}
The first statement follows immediately from the equality 
\[
\varsigma_\infty(\Delta_{u,\infty})=\Delta_{u^{-1},\infty}
\]
for $u\in G$. For the second statement, fix a point $P=(\tau,\tau)\in \Delta_{1,\infty}$ and denote by $W\coloneq \mathrm{T}_\tau(\mathcal{H}_\infty)$ the tangent space at $\tau$ and by $d\colon W^2\setminus \Delta(W)\to W\setminus \lbrace 0\rbrace$ the composition of the map in equation \eqref{isostep27} with the projection to $W\setminus \lbrace 0\rbrace$. The isomorphisms of real vector spaces
\[
\phi\colon W^2\too W^2, \quad  (v,w)\mapsto(w,v) \qquad \mathrm{and} \qquad \phi^{\prime}\colon W \too  W, \quad v\mapsto -v 
\]
fit into the commutative diagram
\begin{center}
\begin{minipage}{\linewidth}
\centering
\begin{tikzcd}
\mathcal{H}_\infty^2\setminus \Delta_{1,\infty} \arrow[rr]{rr}{\varsigma_\infty} & & \mathcal{H}_\infty^2\setminus \Delta_{1,\infty}  \\
W^2\setminus \Delta(W)\arrow[u]{u}{\exp_P} \arrow[d]{d}{d}\arrow[rr]{rr}{\phi} & & W^2\setminus \Delta(W)\arrow[u]{u}{\exp_P}\arrow[d]{d}{d} \\
W\setminus \lbrace 0 \rbrace \arrow[rr]{rr} {\phi^{\prime}} & & W\setminus \lbrace 0 \rbrace.
 \end{tikzcd}
\end{minipage}
\end{center}
As a homomorphism of real vector spaces, the map $\phi^{\prime}$ has determinant $1$, hence the induced map $(\phi^{\prime})_\star\colon\mathrm{H}_1(W\setminus \lbrace 0\rbrace)\to \mathrm{H}_1(W\setminus \lbrace 0\rbrace)$ is the identity. The lemma follows.
\end{proof}
The construction of \cite{DGL} for the group $\Gamma^2$ produces cocycles
\[
\mathcal{D}_v\colon \mathscr{C}_2(\mathcal{H}_\infty^2)\too \mathrm{Div}^\dagger_{\mathrm{rq}}(\mathcal{H}_p^2), \qquad c\mapsto \sum_{w\in \Gamma^2\centerdot v} (c\cap \Delta_{w,\infty}) \cdot \Delta_{w,p}
\]
for every vector $v\in B$ of positive length. Recall that $\Gamma^2$ acts on $B$ via two-sided multiplication. Therefore, $\Gamma^2\centerdot v=\Gamma\cdot v\cdot \Gamma$. In particular, for $v=1$ we obtain 
\[
\mathcal{D}_1(c)=\sum_{\gamma\in \Gamma}(c\cap\Delta_{\gamma,\infty})\cdot \Delta_{\gamma,p}
\]
for $c\in \mathscr{C}_2(\mathcal{H}_\infty^2)$.
One readily verifies that the class induced by $\mathcal{D}_1$ is \emph{symmetric} in the following sense. Denote by $\varsigma_p\colon \mathcal{H}_p^2\to \mathcal{H}_p^2$ the morphism of rigid analytic spaces and by $t\colon G^2\to G^2$ the homomorphism of groups which switches the coordinates, respectively. Clearly, $\varsigma_p$ is compatible with $t$ in the sense that $\varsigma_p(\gamma\centerdot z)=t(\gamma)\centerdot \varsigma_p(z)$ for all $\gamma\in G^2$ and $z\in \mathcal{H}_p^2$. The morphism $\varsigma_p$ induces a homomorphism of groups 
\[
\varsigma_p\colon \mathrm{Div}^\dagger_{\mathrm{rq}}(\mathcal{H}_p^2)\too \mathrm{Div}^\dagger_{\mathrm{rq}}(\mathcal{H}_p^2), \qquad \Delta_{u,p}\mapsto \Delta_{u^{-1},p},
\]
which is compatible with $t$. Hence, the pair $(t,\varsigma_p)$ induces a homomorphism in cohomology (compare \cite[Chapter III.8]{brown}) 
\begin{equation}\label{definitionsigma}
\sigma\colon \mathrm{H}^2(\Gamma^2,\mathrm{Div}^\dagger_{\mathrm{rq}}(\mathcal{H}_p^2))\too \mathrm{H}^2(\Gamma^2,\mathrm{Div}^\dagger_{\mathrm{rq}}(\mathcal{H}_p^2)).
\end{equation}
Lemma \ref{phicomputation} enables us to prove:
\begin{lemma}\label{symmetryofD1}
The class $[\mathcal{D}_1]$ remains fixed under the homomorphism $\sigma$.
\end{lemma}
\begin{proof}
The chain map 
\[
\Hom_{\Gamma^2}(\mathscr{C}_\bullet(\mathcal{H}_\infty^2),\mathrm{Div}^\dagger_{\mathrm{rq}}(\mathcal{H}_p^2))\too \Hom_{\Gamma^2}(\mathscr{C}_\bullet(\mathcal{H}_\infty^2),\mathrm{Div}^\dagger_{\mathrm{rq}}(\mathcal{H}_p^2)), \qquad f\mapsto \varsigma_p\circ f\circ \varsigma_\infty
\]
induces the homomorphism $\sigma$. Thus, the lemma follows if we can prove that this chain map fixes the cocycle $\mathcal{D}_1$. Let $c\in \mathscr{C}_2(\mathcal{H}_\infty^2)$ be a chain. We compute
\[
\varsigma_p\circ \mathcal{D}_{1} \circ \varsigma_\infty \  (c)=\sum_{\gamma\in \Gamma} (\varsigma_\infty(c)\cap \Delta_{\gamma,\infty}) \cdot \Delta_{\gamma^{-1},p}.
\]
We make the following claim:
\[
\varsigma_\infty(c)\cap \Delta_{\gamma,\infty}=c\cap \Delta_{\gamma^{-1},\infty}.
\]
Using the compatibility of $\varsigma_\infty$ with $t$, we can rewrite the left-most term to
\[
\varsigma_\infty(c)\cap \Delta_{\gamma,\infty}=\varsigma_\infty\left((1,\gamma^{-1})\centerdot c\right)\cap \Delta_{1,\infty},
\]
and the right-most term to
\[
c\cap \Delta_{\gamma^{-1},\infty}=\left((1,\gamma^{-1})\centerdot c\right)\cap \Delta_{1,\infty}.
\]
The claim follows from Lemma \ref{phicomputation}.
\end{proof}
We end this section by defining the notion of Kudla--Millson divisors in this setup. For any integer $n\geq 1$ consider the set 
\begin{equation}
\mathcal{O}_n\coloneq \lbrace b\in R[1/p] : \mathrm{Nrd}(b)=n\rbrace
\end{equation}
of elements of reduced norm $n$. Following the proof of Eichler's theorem on norms (see \cite[Theorem 28.6.1]{Voight}), one can show that the set $\mathcal{O}_n$ is non-empty for all integers $n\geq 1$. Strong approximation for $B$ implies that there is a finite system of representatives $\lbrace \beta_j\rbrace_{j=1}^r$ for the double coset space $\Gamma\backslash\mathcal{O}_n/\Gamma$, and we define the cocycle 
\[
\mathcal{D}_n\coloneq \sum_{j=1}^r \mathcal{D}_{\beta_j}\in  \Hom_{\Gamma^2}\left(\mathscr{C}_2(\mathcal{H}_\infty^2), \mathrm{Div}^\dagger_{\mathrm{rq}}(\mathcal{H}_p^2)\right),
\]
whose definition is independent of the chosen system of representatives.
\begin{definition}
A \emph{Kudla--Millson divisor} for $\Gamma^2$ is a finite linear combination of classes $[\mathcal{D}_n]$ for $n\geq 1$. We denote by $\mathcal{KM}(\Gamma^2)\subset \mathrm{H}^2\left(\Gamma^2,\mathrm{Div}^\dagger_{\mathrm{rq}}(\mathcal{H}_p^2)\right)$ the group of Kudla--Millson divisors.
\end{definition}
Later (Corollary \ref{symmetryofdn}), we will prove that the statement in Lemma \ref{symmetryofD1} extends to all Kudla--Millson divisors. The reason for this is the compatibility of the map $\sigma$ with the action of the Hecke operators which we introduce in the next section.
\section{Lifting obstructions}\label{section2}
In this section we study lifting obstructions for the divisor-valued cocycles \`a la Darmon--Vonk and for the Kudla--Millson divisors for $\Gamma^2$. Let $d_B\in \mathbb{Z}$ be the \emph{discriminant} of $B$. That is, $d_B$ is the product of all primes $l\in \mathbb{Z}$ for which $B_l$ is a division ring. In particular, $d_B=1$ if $B$ is split.
\subsection{Hecke operators}\label{sec:heckeop}
We start by briefly recalling the notion of Hecke operators in our setup. A good reference for the following facts is \cite[Chapter 5.3]{miyake}. Let $v\in B$ be a vector of positive reduced norm. Strong approximation for $B$ implies that the double coset $\Gamma\cdot v\cdot \Gamma$ is a finite disjoint union of right cosets
\[
\Gamma\cdot v\cdot \Gamma=\bigsqcup_{j=1}^{r_v}  \Gamma\cdot\alpha_j(v).
\]
See \cite[Lemma 5.3.4]{miyake}. We denote by $\mathbb{T}$ the free $\mathbb{Z}$-module generated by the symbols $\mathrm{T}_v\coloneq \Gamma\cdot v\cdot \Gamma$. Using systems of representatives for the coset spaces $\Gamma\backslash \Gamma\cdot v\cdot \Gamma$ for $v\in G$, one can define a multiplication in $\mathbb{T}$ endowing it with a $\mathbb{Z}$-algebra structure (see \cite[Lemma 2.7.3]{miyake}).
Recall that we denote by $\mathcal{O}_n\subset R[1/p]$ the subset of reduced norm $n$ elements in $R[1/p]$ and that $\mathcal{O}_n$ decomposes into a finite disjoint union of double cosets, say with system of representatives $\lbrace \beta_j\rbrace_{j=1}^r$. We define the $n^{\mathrm{th}}$ Hecke operator $\mathrm{T}_n$ by
\begin{equation}
\mathrm{T}_n\coloneq \sum_{j=1}^r \mathrm{T}_{\beta_j}\in \mathbb{T}.
\end{equation}
It is a standard fact that the Hecke algebra $\mathbb{T}$ is commutative and generated by the Hecke operators $\mathrm{T}_l$ and the Hecke operators $\mathrm{T}(l,l)\coloneq \Gamma\cdot l\cdot \Gamma$ for prime numbers $l\in \mathbb{Z}$. Moreover, one has the following:
\begin{enumerate}
\item[(a)]$\mathrm{T}_m\cdot \mathrm{T}_n=\mathrm{T}_{n\cdot m}$ for coprime integers $n$ and $m$, 
\item[(b)]
\begin{equation*}
\mathrm{T}_l\cdot \mathrm{T}_{l^e}=
   \begin{cases} 
    \mathrm{T}_{l^{e+1}}+l\cdot \mathrm{T}(l,l)\cdot \mathrm{T}_{l^{e-1}} , & \mathrm{if} \ l\nmid d_B\cdot p \ \mathrm{and}\\ 
     \mathrm{T}_{l^{e+1}} , & \mathrm{if} \  l\mid d_B\\
     \end{cases}
\end{equation*}
for integers $e\geq 1$ and 
\item[(c)]$\mathrm{T}_{p^{2\cdot m+\varepsilon}}=\mathrm{T}(p,p)^m\cdot \mathrm{T}_{p^\varepsilon}$ for $m\in \mathbb{Z}$, $m \geq 0$, and $\varepsilon\in \lbrace 0,1\rbrace$.
\end{enumerate}
We refer to \cite[Corollary 5.3.7]{miyake} for a proof of this fact. \\
Let $\Omega$ be a $G$-module and consider a projective resolution $C_\bullet$ of $\mathbb{Z}$ over $\mathbb{Z}[G]$. Since projective $\mathbb{Z}[G]$-modules are projective as $\mathbb{Z}[\Gamma]$-modules, we can compute the group cohomology of $\Gamma$ with values in $\Omega$ in terms of the resolution $C_\bullet$. Thus, the chain map
\[
\mathrm{T}_v\colon  \Hom_{\Gamma}(C_\bullet,\Omega)\too \Hom_{\Gamma}(C_\bullet,\Omega), \quad u\mapsto \mathrm{T}_v(u),
\]
where
\[
\mathrm{T}_v(u):  c\mapsto \sum_{j=1}^{r_v}\alpha_j(v)^{-1}\centerdot u(\alpha_j(v)\centerdot c),
\]
induces an endomorphism $\mathrm{T}_v \colon \mathrm{H}^i(\Gamma,\Omega)\to \mathrm{H}^i(\Gamma,\Omega)$ for all $i\geq 0$. We deduce maps $\mathbb{T}\to \mathrm{End}(\mathrm{H}^i(\Gamma,\Omega))$ for $i\geq 0$ and one can show that these maps are ring homomorphisms. Compare \cite[Lemma 2.7.4]{miyake}. Analogously, if $\mathrm{M}$ is a $G^2$-module and $P_\bullet$ is a projective resolution of $\mathbb{Z}$ over $\mathbb{Z}[G^2]$, then for all pairs $(v,w)\in G^2$ we obtain Hecke operators 
\[
\mathrm{T}_v\times \mathrm{T}_w\colon \mathrm{H}^i(\Gamma^2,\mathrm{M})\too \mathrm{H}^i(\Gamma^2,\mathrm{M}) \quad \mathrm{for} \quad i\geq 0,
\]
induced by the chain map
\[
\mathrm{T}_{v}\times \mathrm{T}_w \ (u):  x\mapsto \sum_{j=1}^{r_v} \sum_{k=1}^{r_w}\left(\alpha_j(v)^{-1}, \alpha_k(w)^{-1}\right)\centerdot u\left( \left(\alpha_j(v), \alpha_k(w)\right) \centerdot x\right)
\]
for $u\in \Hom_{\Gamma^2}(P_\bullet,\mathrm{M})$. We obtain ring homomorphisms $\mathbb{T}\times \mathbb{T}\to \mathrm{End}(\mathrm{H}^i(\Gamma^2,\mathrm{M}))$ for $i\geq 0$. 
\begin{remark}\label{remark2.4}
If $\Omega$ is a $G$-module whose module structure factors through $G/\mathbb{Q}^\times$, then the Hecke operators $\mathrm{T}(l,l)$ act trivially on $\mathrm{H}^i(\Gamma,\Omega)$ for all $i\geq 0$. Therefore, the commutativity of $\mathbb{T}$ together with the identities (a), (b) and (c) imply that the image of $\mathbb{T}$ in the endomorphism ring of $\mathrm{H}^i(\Gamma,\Omega)$ equals the subgroup generated by the Hecke operators $\mathrm{T}_n$ with $n\geq 1$. Analogously, if $\mathrm{M}$ is a $G^2$-module whose module structure factors through $(G/\mathbb{Q}^\times)^2$, then the image of $\mathbb{T}\times \mathbb{T}$ in the endomorphism ring equals the subgroup generated by the Hecke operators $\mathrm{T}_n\times \mathrm{T}_m$, $\mathrm{T}_n\times 1$ and $1\times \mathrm{T}_m$ for integers $n,m\geq 1$ .
\end{remark}
The group of Weil divisors $\mathrm{Div}^\dagger(\mathcal{H}_p)$ is a $G/\mathbb{Q}^\times$-module. Thus, the cohomology groups are endowed with a $\mathbb{T}$-module structure. In particular, the group $\mathcal{RM}(\Gamma)$ is a $\mathbb{T}$-module. If $B$ is split, then this follows because $\mathcal{RM}(\Gamma)$ is isomorphic to the group of $\Gamma$-invariant modular symbols, and if $B$ is non-split, then this is \cite[Proposition 11]{GehrmannquaternionicRC}. Similarly, the group of rational quadratic divisors is a $(G/\mathbb{Q}^\times)^2$-module, and the cohomology groups are endowed with a $\mathbb{T}\times \mathbb{T}$-module structure. To show that the subgroup $\mathcal{KM}(\Gamma^2)$ of Kudla--Millson divisors is a $\mathbb{T}\times \mathbb{T}$-module, we need to compute the image of the class $[\mathcal{D}_1]$ under the Hecke operators $\mathrm{T}_n\times \mathrm{T}_m$. Since the resolution $\mathscr{C}_\bullet(\mathcal{H}_\infty^2)$ is a resolution of $\mathbb{Z}$ over $\mathbb{Z}[G^2]$, we can do this in terms of the cocycle $\mathcal{D}_1$.
\begin{proposition}\label{lemmacomputation}
We have the following:
\begin{enumerate}
\item[(1)]Let $n\geq 1$ be an integer. Then, $(\mathrm{T}_n\times 1)(\mathcal{D}_1)\overset{(a)}{=}\mathcal{D}_n\overset{(b)}{=}(1\times \mathrm{T}_n)(\mathcal{D}_1)$.
\item[(2)]Let $\mathrm{T}\in \mathbb{T}$ be any Hecke operator. Then, $(\mathrm{T}\times 1)(\mathcal{D}_1)=(1\times \mathrm{T})(\mathcal{D}_1)$.
\end{enumerate}
In particular,
\[
\mathcal{KM}(\Gamma^2)=(\mathbb{T}\times 1)\cdot [\mathcal{D}_1].
\]
\end{proposition}
\begin{proof}
Let $v\in G$. We claim the following:
\[
(\mathrm{T}_{v}\times 1) (\mathcal{D}_{1})\overset{(a^\prime)}{=} \mathcal{D}_{v^{-1}}\overset{(b^\prime)}{=} (1\times \mathrm{T}_{v^{-1}})   (\mathcal{D}_{1}).
\]
To see this, let $c\in \mathscr{C}_2(\mathcal{H}_\infty^2)$ be a $2$-chain. We compute
\begin{align*}
(\mathrm{T}_{v}\times 1) (\mathcal{D}_{1}) (c)&=\sum_{i=1}^{r_v} (\alpha_i(v)^{-1},1) \centerdot \mathcal{D}_{1}\left((\alpha_i(v),1) \centerdot c\right)  \\
&=\sum_{i=1}^{r_v} \sum_{\gamma\in \Gamma} \left(\left((\alpha_i(v),1)\centerdot c\right)\cap \Delta_{\gamma,\infty} \right)\cdot \  \Delta_{\alpha_i(v)^{-1} \cdot\gamma,p} \\
&=\sum_{i=1}^{r_v} \sum_{\gamma\in \Gamma}  (c\cap \Delta_{\alpha_i(v)^{-1}\cdot \gamma,\infty}) \cdot \  \Delta_{\alpha_i(v)^{-1} \cdot \gamma,p}\\
&=\sum_{\gamma\in \Gamma \cdot  v^{-1}\cdot \Gamma} (c\cap \Delta_{\gamma,\infty}) \cdot \  \Delta_{\gamma,p}=\mathcal{D}_{v^{-1}}(c).
\end{align*}
The last equality holds because 
\[
\Gamma \cdot v^{-1}\cdot  \Gamma=(\Gamma \cdot v \cdot \Gamma)^{-1}=\bigsqcup_{i=1}^{r_v} \alpha_i(v)^{-1} \cdot \Gamma. 
\]
This shows equality $(a^\prime)$. Equality $(b^\prime)$ follows by a  similar computation and is left to the reader. \\
The identities in statement (1) follow now from $(a^\prime)$ and $(b^\prime)$: Let $\lbrace\beta_j\rbrace_{j=1}^r$ be a system of representatives for $\Gamma\backslash \mathcal{O}_n/\Gamma$. 
Then 
\[
\mathcal{D}_n=\sum_{j=1}^r\mathcal{D}_{\beta_j}\overset{(b^\prime)}{=}( 1\times \sum_{j=1}^r\mathrm{T}_{\beta_j} )(\mathcal{D}_1)=(1\times \mathrm{T}_n)(\mathcal{D}_1),
\]
and $(b)$ is proven.
To prove $(a)$, observe that
\[
\mathcal{O}_n=\overline{\mathcal{O}_n}=\bigsqcup_{j=1}^r \Gamma\cdot \overline{\beta_j}\cdot \Gamma.
\]
Thus, $\big\{\overline{\beta_j}\big\}_{j=1}^r$ is a system of representatives for $\Gamma\backslash \mathcal{O}_n/\Gamma$.
Moreover, since $\mathrm{Nrd}(\beta_j)=n$, we have $\overline{\beta_j}=n\cdot \beta_j^{-1}$. 
This gives the identity $\mathcal{D}_{\overline{\beta_j}}=\mathcal{D}_{\beta_j^{-1}}$, and we derive
\[
\mathcal{D}_n=\sum_{j=1}^r\mathcal{D}_{\overline{\beta_j}}=\sum_{j=1}^r\mathcal{D}_{\beta_j^{-1}}\overset{(a^\prime)}{=}(\sum_{j=1}^r\mathrm{T}_{\beta_j}\times 1)(\mathcal{D}_1)=(\mathrm{T}_n\times 1)(\mathcal{D}_1).
\]
Finally, the remaining claims follow directly from (1) and Remark \ref{remark2.4} 
\end{proof}
\begin{corollary}\label{symmetryofdn}
The homomorphism $\sigma$ defined in equation \eqref{definitionsigma} is the identity after restricting to the subgroup $\mathcal{KM}(\Gamma^2)$.
\end{corollary}
\begin{proof}
Let $n\geq 1$ be an integer. We compute
\[
\sigma(\mathcal{D}_n)=\sigma\left((\mathrm{T}_n\times 1) \cdot \mathcal{D}_1\right)=(1\times \mathrm{T}_n)\cdot \sigma(\mathcal{D}_1)=(1\times \mathrm{T}_n)\cdot \mathcal{D}_1=\mathcal{D}_n,
\]
where the third equality is Lemma \ref{symmetryofD1}.
\end{proof}
\subsection{Lifting obstructions}
For the moment, let $X_p$ be the $p$-adic symmetric space associated to an arbitrary rational quadratic space $(V,q)$. We denote by $\mathcal{A}$ the ring of rigid analytic functions and by $\mathcal{M}$ the field of rigid meromorphic functions on $X_p$, respectively. Furthermore, we denote by $\mathrm{Div}^\dagger(X_p)$ the group of Weil divisors on $X_p$. The multiplicative groups $\mathcal{A}^\times$ of invertible rigid analytic and $\mathcal{M}^\times$ of non-zero rigid meromorphic functions on $X_p$ are (left) $\mathrm{O}_V(\mathbb{Q}_p)$-modules via
\[
(\gamma\centerdot f)(z)\coloneq f(\gamma^{-1}\centerdot z) \quad \mathrm{for} \quad f\in \mathcal{M}^\times \qquad \mathrm{and} \qquad \gamma\in \mathrm{O}_V(\mathbb{Q}_p). 
\]
The divisor map, which assigns to every rigid meromorphic function its divisor, gives rise to an exact sequence of $\mathrm{O}_V(\mathbb{Q}_p)$-modules
\[
0\too \mathcal{A}^\times\too \mathcal{M}^\times\xrightarrow{\mathrm{div}} \mathrm{Div}^\dagger(X_p).
\]
We have:
\begin{proposition}\label{propdivissurjective}
For the rational quadratic spaces $V=B_0$ and $V=B$, the divisor map is surjective. 
In other words, we have a short exact sequence of $\mathrm{O}_V(\mathbb{Q}_p)$-modules
\[
0\too \mathcal{A}^\times\too \mathcal{M}^\times\too \mathrm{Div}^\dagger(X_p)\too 0.
\]
\end{proposition}
\begin{proof}
For $V=B_0$ this result is well-known, see for example \cite[Theorem 2.7.6]{vanderputfresnel}; for $V=B$ this follows because the Picard group of $\mathcal{H}_p^2$ vanishes by \cite[Theorem 1.4]{GehrmannSprehe}.
\end{proof}
From now on, we denote by $\mathcal{A}_n$ (respectively $\mathcal{M}_n$) the ring of rigid analytic (respectively rigid meromorphic) functions on $\mathcal{H}_p^n$ for $n\in \lbrace 1,2\rbrace$. The short exact sequence in Proposition \ref{propdivissurjective} induces exact sequences in cohomology
\[
\mathrm{H}^1(\Gamma,\mathcal{M}_{1}^\times)\too  \mathrm{H}^1(\Gamma,\mathrm{Div}^\dagger(\mathcal{H}_p))\too \mathrm{H}^{2}(\Gamma,\mathcal{A}_{1}^\times)
\]
and
\[
\mathrm{H}^2(\Gamma^2,\mathcal{M}_{2}^\times)\too  \mathrm{H}^2(\Gamma^2,\mathrm{Div}^\dagger(\mathcal{H}_p^2))\too \mathrm{H}^{3}(\Gamma^2,\mathcal{A}_{2}^\times),
\]
which are compatible with the $\mathbb{T}$-module (respectively $\mathbb{T}\times \mathbb{T}$-module) structure. The rest of this section is devoted to analysing the images of the connecting homomorphisms and, in particular, to determining an ideal of the Hecke algebra that annihilates them. 
\subsubsection{The Darmon--Vonk case}\label{liftingobstructiondv}
The short exact sequence  
\[
0\too \mathbb{C}_p^\times \too \mathcal{A}_{1}^\times \too \mathcal{A}_{1}^\times/\mathbb{C}_p^\times\too 0
\]
induces an exact sequence in cohomology 
\[
\mathrm{H}^2(\Gamma,\mathbb{C}_p^\times)\too  \mathrm{H}^2(\Gamma,\mathcal{A}_{1}^\times) \too  \mathrm{H}^2(\Gamma,\mathcal{A}_{1}^\times/\mathbb{C}_p^\times).
\]
Thus, we can determine the annihilator of $\mathrm{H}^2(\Gamma,\mathcal{A}_{1}^\times)$ by studying the left- ~and right-most Hecke modules. \\
As a $p$-arithmetic group, $\Gamma$ is of type $\mathrm{FP}_\infty$. That is, there exists a resolution of the trivial $\Gamma$-module $\mathbb{Z}$ that consists of finitely generated free $\mathbb{Z}[\Gamma]$-modules. See \cite{borelserre1976} for a proof of this fact. This finiteness condition enables us to apply the following Universal Coefficient Theorem.
Let $H$ be a group of type $\mathrm{FP}_\infty$, and let $F$ be a trivial $H$-module. Then for every integer $n\geq 0$ there is a natural split exact sequence
 \[
0\too \mathrm{H}^n(H,\mathbb{Z})\otimes F \too  \mathrm{H}^n(H,F)\too \mathrm{Tor}_1^\mathbb{Z}\left(\mathrm{H}^{n+1}(H,\mathbb{Z}), F\right)\too 0.
\]
See \cite[Theorem 3.5.5]{SpreheThesis} for a proof. This gives a split exact sequence
\[
0\too \mathrm{H}^2(\Gamma,\mathbb{Z})\otimes \mathbb{C}_p^\times \too \mathrm{H}^2(\Gamma,\mathbb{C}_p^\times)\too \mathrm{Tor}^\mathbb{Z}_1(\mathrm{H}^{3}(\Gamma,\mathbb{Z}),\mathbb{C}_p^\times)\too 0.
\]
The first arrow is compatible with the $\mathbb{T}$-action. That is, an element $(\mathrm{T}\cdot \mathcal{J})\otimes \lambda$ is mapped to $\mathrm{T}\cdot (\mathcal{J}\otimes \lambda)$. This property is easily verified on the level of cocycles (see \cite[Lemma 3.5.6]{SpreheThesis}). 
Furthermore, $\mathrm{H}^j(\Gamma,\Z)$ is a finitely generated abelian group for all $j\geq 0$, because $\Gamma$ is of type $\mathrm{FP}_\infty$. Thus, the right-most term of the sequence is finite, and the ideal $\mathrm{Ann}_\mathbb{T}(\mathrm{H}^2(\Gamma,\mathbb{Z}))$ is non-zero.
If we denote by $f\geq 1$ the exponent of the torsion part of $\mathrm{H}^{3}(\Gamma,\mathbb{Z})$, then 
\[
0\neq  f\cdot \mathrm{Ann}_\mathbb{T}\left(\mathrm{H}^2(\Gamma,\mathbb{Z})\right)\subset \mathrm{Ann}_\mathbb{T}\left(\mathrm{H}^2(\Gamma,\mathbb{C}_p^\times)\right).
\]
The central combinatorial object for studying the Hecke module $\mathrm{H}^2(\Gamma,\mathcal{A}_1^\times/\mathbb{C}_p^\times)$ is the \emph{Bruhat--Tits tree} $\mathcal{T}$ for $B_p\simeq \GL_2(\mathbb{Q}_p)$. We denote by $\mathcal{V}$ the set of vertices and by $\mathcal{E}^{\mathrm{or}}$ the set of oriented edges of $\mathcal{T}$. For an edge $e\in \mathcal{E}^{\mathrm{or}}$ we denote by $s(e)\in \mathcal{V}$ the source and by $t(e)\in \mathcal{V}$ the target of $e$. Furthermore, we denote by $\overline{e}$ the edge obtained by reversing the orientation of $e$. 
Denote by $\mathcal{F}(\mathcal{V},\mathbb{Z})$ and $\mathcal{F}(\mathcal{E}^{\mathrm{or}},\mathbb{Z})$ the groups of maps $\mathcal{V}\to \mathbb{Z}$ and $\mathcal{E}^{\mathrm{or}}\to \mathbb{Z}$, respectively. We endow both groups with a $\GL_2(\mathbb{Q}_p)$-module structure via
\[
(\gamma \varphi)(x)=\varphi(\gamma^{-1} x).
\]
Moreover, let 
\[
\mathcal{F}_0(\mathcal{E}^{\mathrm{or}},\mathbb{Z})\coloneq\lbrace \varphi\in \mathcal{F}(\mathcal{E}^{\mathrm{or}},\mathbb{Z}) : \varphi(e)+\varphi(\overline{e})=0 \ \mathrm{for \ all} \ e\in \mathcal{E}^{\mathrm{or}}\rbrace.
\]
The map 
\[
\partial\colon \mathcal{F}_0(\mathcal{E}^{\mathrm{or}},\mathbb{Z})\too \mathcal{F}(\mathcal{V},\mathbb{Z}), \qquad \varphi\mapsto \left(v\mapsto \sum_{s(e)=v} \varphi(e)\right)
\]
induces a short exact sequence of $\GL_2(\mathbb{Q}_p)$-modules
\begin{equation}
0\too \mathrm{ker}(\partial) \too \mathcal{F}_0(\mathcal{E}^{\mathrm{or}},\mathbb{Z})\too \mathcal{F}(\mathcal{V},\mathbb{Z}) \too 0.
\end{equation}
The elements in $\mathrm{ker}(\partial)$ are called \emph{harmonic cochains on $\mathcal{T}$}.
In \cite{vfples}, van der Put constructs a $\GL_2(\mathbb{Q}_p)$-equivariant map $\mathcal{A}_{1}^\times\to \mathrm{ker}(\partial)$, whose kernel equals the group of constant functions. 
Let us briefly sketch this construction: 
We fix an isomorphism $\iota_p\colon B_p\to \mathrm{M}_2(\mathbb{Q}_p)$ that identifies $R\otimes_\mathbb{Z} \mathbb{Z}_p$ with $\mathrm{M}_2(\mathbb{Z}_p)$. 
This gives an identification $\mathcal{H}_p=\PP^1(\mathbb{C}_p)\setminus \mathbb{P}^1(\mathbb{Q}_p)$.
The group $\Gamma$ acts on the set of vertices with two, and on the set of unoriented edges of $\mathcal{T}$ with one orbit. Namely, a system of representatives is given by a vertex $v_0$ whose stabilizer in $\mathrm{SL}_2(\mathbb{Q}_p)$ is $\mathrm{SL}_2(\mathbb{Z}_p)$, a vertex $v_1$ whose stabilizer $\SL_2(\Z_p)^\prime$ in $\SL_2(\mathbb{Q}_p)$ is a conjugate of $\SL_2(\Z_p)$, and the oriented edge $e_0$ which connects $v_0$ and $v_1$. 
See \cite[Chapter II.1.4]{SerreTrees} for further reference.
We denote by $\red \colon \mathcal{H}_p \to \mathcal{T}$ the $\GL_2(\Q_p)$-equivariant reduction map (see \cite[Section 1.3]{Dasguptap-adicupper}).
Given an oriented edge $e\in \mathcal{E}^{\mathrm{or}}$, define $U_e\coloneq \mathrm{red}^{-1}(e)$ to be the preimage of $e$ under the reduction map.
Note that $U_e$ is open.
If $f\in \mathcal{A}_1^\times$, then the function $z\mapsto \mathrm{ord}_p(f(z))\in \Q$ is linear, and its slope is integral. The van der Put map is then defined by 
\[
\mathcal{A}_1^\times \too \mathrm{ker}(\partial), \qquad f\mapsto \left( e\mapsto -\text{slope of} \   \mathrm{ord}_p(f|_{U_e})\right).
\] 
It is a theorem of van der Put (see \cite[Proposition 1.1]{vfples}) that this map is surjective, and that its kernel is $\C_p^\times$. In particular, the van der Put map induces the short exact sequence of $\GL_2(\mathbb{Q}_p)$-modules
\begin{equation}\label{sequencevanderput}
0\too \mathcal{A}_{1}^\times/\mathbb{C}_p^\times\too \mathcal{F}_0(\mathcal{E}^{\mathrm{or}},\mathbb{Z})\too \mathcal{F}(\mathcal{V},\mathbb{Z})\too 0.
\end{equation}
The stabilizer of the vertex $v_0$ in $\Gamma$ is the arithmetic group $\Gamma_0\coloneq R_1^\times$ of reduced norm $1$ units in $R$, the stabilizer $\Gamma_0^\prime$ of $v_1$ is conjugated to $\Gamma_0$, and the stabilizer $\Gamma_0^p$ of $e_0$ is the group of reduced norm $1$ units in an Eichler order of level $p$ contained in the maximal order $R$. 
This implies that there are isomorphisms of $\Gamma$-modules
\[
\mathcal{F}_0(\mathcal{E}^{\mathrm{or}},\mathbb{Z})\simeq \mathrm{Maps}(\Gamma e_0,\mathbb{Z})\simeq \mathrm{Coind}^\Gamma_{\Gamma_0^p}(\mathbb{Z})
\]
and
\[
\mathcal{F}(\mathcal{V},\mathbb{Z}) \simeq \mathrm{Coind}^\Gamma_{\Gamma_0}(\mathbb{Z}) \ \oplus \mathrm{Coind}^\Gamma_{\Gamma_0^\prime}(\mathbb{Z}).
\]
Consequently, the short exact sequence \eqref{sequencevanderput} induces an exact sequence in cohomology 
\[
\mathrm{H}^1(\Gamma_0,\mathbb{Z})\oplus \mathrm{H}^1(\Gamma_0^\prime,\mathbb{Z}) \too \mathrm{H}^{2}(\Gamma,\mathcal{A}^\times_{1}/\mathbb{C}_p^\times)\too \mathrm{H}^{2}(\Gamma_0^p,\mathbb{Z}),
\]
which is compatible with the action of the Hecke operators $\mathrm{T}_n$ for $p\nmid n$. Finally, the arithmetic groups $\Gamma_0$, $\Gamma_0^\prime$ and $\Gamma_0^p$ are of type $\mathrm{FP}_\infty$ which implies that the integer cohomology groups are finitely generated abelian groups. Thus, the annihilator $\mathrm{Ann}_\mathbb{T}\left(\mathrm{H}^{2}(\Gamma,\mathcal{A}^\times_{1}/\mathbb{C}_p^\times)\right)$ must be non-trivial by the sequence above. We have proven:
\begin{lemma}
The ideal $\mathrm{Ann}_\mathbb{T}(\mathrm{H}^2(\Gamma,\mathcal{A}_1^\times))$ is non-trivial.
\end{lemma}

\subsubsection{The $4$-dimensional case}
Let 
\begin{equation}
\mathcal{M}_{\mathrm{rq}}^\times\coloneq \mathrm{div}^{-1}\left(\mathrm{Div}^\dagger_{\mathrm{rq}}(\mathcal{H}_p^2)\right)
\end{equation}
be the $G^2$-module of non-zero rigid meromorphic functions whose associated divisor is rational quadratic. Proposition \ref{propdivissurjective} implies the exactness of the short exact sequence of $G^2$-modules
\begin{equation}\label{sesforrq}
0\too \mathcal{A}_{2}^\times\too \mathcal{M}_{\mathrm{rq}}^\times\xrightarrow{\mathrm{div}} \mathrm{Div}^\dagger_{\mathrm{rq}}(\mathcal{H}_p^2) \too 0,
\end{equation}
and we obtain an exact sequence of Hecke modules
\[
\mathrm{H}^2(\Gamma^2,\mathcal{M}_{\mathrm{rq}}^\times)\too  \mathrm{H}^2(\Gamma^2,\mathrm{Div}^\dagger_{\mathrm{rq}}(\mathcal{H}_p^2))\xrightarrow{\kappa} \mathrm{H}^{3}(\Gamma^2,\mathcal{A}_{2}^\times).
\]
The rest of this subsection is devoted to determining an ideal of the Hecke algebra $\mathbb{T}\times \mathbb{T}$ that annihilates the lifting obstruction $\kappa([\mathcal{D}_1])$ of $[\mathcal{D}_1]$.
To do so, we study the Hecke module $\mathrm{H}^{3}(\Gamma^2,\mathcal{A}_{2}^\times)$:
The short exact sequence 
\[
0\too \mathbb{C}_p^\times\too \mathcal{A}_{2}^\times \too \mathcal{A}_{2}^\times/\mathbb{C}_p^\times \too  0
\]
induces an exact sequence in cohomology
\[
\mathrm{H}^3(\Gamma^2,\mathbb{C}_p^\times)\too \mathrm{H}^3(\Gamma^2,\mathcal{A}_{2}^\times) \xrightarrow{\overline{(-)}} \mathrm{H}^3(\Gamma^2,\mathcal{A}_{2}^\times/\mathbb{C}_p^\times).
\]
Since $\Gamma$ is of type $\mathrm{FP}_\infty$, the same is true for the product $\Gamma^2$.
Thus, the Universal Coefficient Theorem yields a split exact sequence
\[
0\too \mathrm{H}^3(\Gamma^2,\mathbb{Z})\otimes \mathbb{C}_p^\times \too \mathrm{H}^3(\Gamma^2,\mathbb{C}_p^\times)\too \mathrm{Tor}_1^\mathbb{Z}\left(\mathrm{H}^4(\Gamma^2,\Z),\C_p^\times\right)\too 0.
\]
The right-most term is finite because $\mathrm{H}^4(\Gamma^2,\Z)$ is a finitely generated abelian group as the integer cohomology of a group of type $\mathrm{FP}_\infty$.
For the computation of the left-most term and our computations in the remaining sections, we need the K\"unneth formula. Since this plays an important role in our arguments in the next sections, we state it precisely:
\begin{proposition}\label{prop:kunneth}
Let $G_1$ and $G_2$ be two groups, and assume that $G_2$ is of type $\mathrm{FP}_\infty$.
Let $\Omega_1$ be a $G_1$-module, and let $\Omega_2$ be a $G_2$-module that is free and of finite rank as a $\mathbb{Z}$-module. Then there is a natural split exact sequence
\begin{align*}
0\too &\bigoplus_{r+s=n} \mathrm{H}^r(G_1,\Omega_1)\otimes \mathrm{H}^s(G_2,\Omega_2)\too \mathrm{H}^n (G_1\times G_2,\Omega_1\otimes \Omega_2)\\
\too &\bigoplus_{r+s=n+1} \mathrm{Tor}_1^\mathbb{Z}\left(\mathrm{H}^r(G_1,\Omega_1),\mathrm{H}^s(G_2,\Omega_2)\right)\too 0
\end{align*}
for all $n\geq 0$.
\end{proposition}
\begin{proof}
See \cite[Theorem 2.3.7]{SpreheThesis} for a proof, and also \cite[Exercise 6.1.8]{Weibel} for a statement of the proposition.
\end{proof}
The left-most map in the K\"unneth Formula is induced by the cochain cross-product: ~let $P_\bullet$ be a resolution of $\Z$ over $\Z[G_1]$, and $Q_\bullet$ be a resolution of $\Z$ over $\Z[G_2]$. The total complex $P_\bullet\otimes Q_\bullet$ is the complex whose objects are
\[
\left(P_\bullet\otimes Q_\bullet\right)_n=\bigoplus_{r+s=n} P_r\otimes Q_s,
\]
and whose differentials are given by
\[
\partial_{r+s}(p_r\otimes q_s)=\partial_r(p_r)\otimes q_s + (-1)^{r} p_r\otimes \partial_s(q_s) 
\]
for $p_r\in P_r$ and $q_s\in Q_s$. 
If $P_\bullet$ and $Q_\bullet$ are projective resolutions, then $P_\bullet\otimes Q_\bullet$ is a projective resolution of $\Z$ over $\Z[G_1\times G_2]$.
For cochains $u_1\in \Hom_{G_1}(P_r,\Omega_1)$ and $u_2\in \Hom_{G_2}(Q_s,\Omega_2)$ we define the cochain cross-product of $u_1$ and $u_2$ to be the cochain
\[
u_1\times u_2\colon \left(P_\bullet\otimes Q_\bullet\right)_{r+s}\too P_r\otimes Q_s \too \Omega_1\otimes \Omega_2,
\]
where the second arrow is induced by 
\[
p_r\otimes q_s\mapsto (-1)^{r\cdot s} \cdot u_1(p_r)\otimes u_2(q_s).
\]
We are now able to prove:
\begin{lemma}\label{computationofh3gamma2}
The group $\mathrm{H}^3(\Gamma^2,\mathbb{C}_p^\times)$ is finite.
If $e$ is the exponent of $\mathrm{H}^3(\Gamma^2,\mathbb{C}_p^\times)$, then 
\[
e\cdot \mathrm{Ann}_{\mathbb{T}\times \mathbb{T}}\left(\mathrm{H}^3(\Gamma^2,\mathcal{A}_{2}^\times/\mathbb{C}_p^\times)\right)\subset \mathrm{Ann}_{\mathbb{T}\times \mathbb{T}}\left(\mathrm{H}^3(\Gamma^2,\mathcal{A}_{2}^\times)\right).
\] 
\end{lemma}
\begin{proof}
The lemma follows if we can prove that the group $\mathrm{H}^3(\Gamma^2,\mathbb{Z})$ is finite.
To see this, we apply the K\"unneth Formula to obtain a split exact sequence 
\[
0\too \bigoplus_{i+j=3} \mathrm{H}^i(\Gamma,\mathbb{Z})\otimes \mathrm{H}^j(\Gamma,\mathbb{Z}) \too \mathrm{H}^3(\Gamma^2,\mathbb{Z}) \too \left(\mathrm{finite}\right)\too 0.
\]
Note that the right-most term is indeed finite because it is a direct sum of groups of the form $\mathrm{Tor}_1^\mathbb{Z}\left(\mathrm{H}^i(\Gamma,\Z),\mathrm{H}^{j}(\Gamma,\mathbb{Z})\right)$ for $i+j=4$.
These are finite because $\mathrm{H}^{i}(\Gamma,\mathbb{Z})$ is a finitely generated abelian group for all $i\geq 0$.
Thus, it suffices to prove that each group in the direct sum on the left is finite. Margulis' normal subgroup theorem (\cite[Chapter VIII, Theorem 2.6]{margulis}) implies that $\mathrm{H}^1(\Gamma,\mathbb{Z})$ vanishes.
Hence, the only non-vanishing terms on the left might be $\mathrm{H}^3(\Gamma,\Z)$. This group can be computed via a Mayer--Vietoris exact sequence:
Theorem 3 in Chapter II.1.4 of \cite{SerreTrees} yields that the group $\Gamma$ is the amalgamation of $\Gamma_0$ and $\Gamma_0^\prime$ along $\Gamma_0^p$. Hence, we obtain a Mayer--Vietoris exact sequence 
\begin{equation}\label{eq:mayer}
\mathrm{H}^2(\Gamma_0,\mathbb{Z})\oplus \mathrm{H}^2(\Gamma_0^\prime,\mathbb{Z})  \xrightarrow{d} \mathrm{H}^2(\Gamma_0^p,\mathbb{Z}) \too \mathrm{H}^3(\Gamma,\mathbb{Z})\too \mathrm{H}^3(\Gamma_0,\mathbb{Z})\oplus \mathrm{H}^3(\Gamma_0^\prime,\mathbb{Z}).
\end{equation}
The map $d$ is given by 
\[
d\colon \mathrm{H}^2(\Gamma_0,\mathbb{Z})\oplus \mathrm{H}^2(\Gamma_0^\prime,\mathbb{Z}) \xrightarrow{\res} \mathrm{H}^2(\Gamma_0^p,\mathbb{Z})\oplus \mathrm{H}^2(\Gamma_0^p,\mathbb{Z}) \too \mathrm{H}^2(\Gamma_0^p,\mathbb{Z}),
\]
where the second map is the difference of two classes.\\
We start by considering the case that $B$ is split.
Then both $\Gamma_0$ and $\Gamma_0^\prime$ are isomorphic to the modular group $\mathrm{SL}_2(\Z)$, and $\Gamma_0^p$ is isomorphic to the congruence subgroup $\Gamma_0(p)$.
Therefore, $\mathrm{H}^i(\Gamma_0,\Z)\simeq \mathrm{H}^i(\Gamma_0^\prime,\Z)$, and $\mathrm{H}^i(\Gamma_0^p,\Z)$ are finite for all $i\geq 2$.
The finiteness of $\mathrm{H}^3(\Gamma,\Z)$ follows from the sequence \eqref{eq:mayer}. \\
Next, assume that $B$ is non-split. Then the right-most groups of \eqref{eq:mayer} are finite, and the groups $\mathrm{H}^2(\Gamma_0,\mathbb{Z})$, $\mathrm{H}^2(\Gamma_0^\prime,\mathbb{Z}) $ and $\mathrm{H}^2(\Gamma_0^p,\mathbb{Z})$ are, up to torsion, infinite cyclic (see \cite[Corollary 4.1]{hughes}). 
Thus, the cokernel of $d$ is finite if and only if the torsion-free part of the image of $d$ is not $0$.
This always holds: ~the composition
\[
\mathrm{H}^2(\Gamma_0,\Z)\xrightarrow{\res} \mathrm{H}^2(\Gamma_0^p,\Z) \xrightarrow{\mathrm{cor}} \mathrm{H}^2(\Gamma_0,\Z),
\]
where the second map is the co-restriction, equals multiplication with the index $[\Gamma_0:\Gamma_0^p]$.
In particular, if we take a generator $J\in\mathrm{H}^2(\Gamma_0,\Z)$ of the free part, then $\res(J)$ cannot be torsion because $[\Gamma_0:\Gamma_0^p]\cdot J$ is not torsion.
We derive that $d(J,0)$ is a non-torsion class in the image of $d$, and the finiteness of $\mathrm{H}^3(\Gamma,\mathbb{Z})$ follows in this case.
\end{proof}
The non-triviality of the annihilator of $\mathrm{H}^2(\Gamma,\mathcal{A}_2^\times)$ thus follows, if  $\mathrm{Ann}_{\mathbb{T}\times \mathbb{T}}\left(\mathrm{H}^3(\Gamma^2,\mathcal{A}_{2}^\times/\mathbb{C}_p^\times)\right)\neq 0$.
For this purpose, consider the $G^2$-module homomorphism
\begin{equation}\label{propropropbloed}
 \mathcal{A}_{1}^\times/\mathbb{C}_p^\times  \otimes   \Z \  \oplus  \  \Z \otimes \mathcal{A}_{1}^\times/\mathbb{C}_p^\times \too \mathcal{A}_{2}^\times/\mathbb{C}_p^\times, \qquad (f_1,f_2)\mapsto f_1(z_1)\cdot f_2(z_2).
\end{equation}
It is a theorem of Gehrmann--Sprehe (\cite[Theorem 1.4]{GehrmannSprehe}) that this homomorphism is an isomorphism. See also \cite[Theorem B.4.1]{SpreheThesis}.
Hence, we deduce a Hecke module isomorphism 
\[
\mathrm{H}^3(\Gamma^2,\mathcal{A}_{2}^\times/\mathbb{C}_p^\times) \simeq  \mathrm{H}^3(\Gamma^2,\mathcal{A}_{1}^\times/\mathbb{C}_p^\times\otimes \mathbb{Z})\ \oplus \ \mathrm{H}^3(\Gamma^2, \mathbb{Z}\otimes\mathcal{A}_{1}^\times/\mathbb{C}_p^\times).
\]
The right-most groups can be computed with the K\"unneth formula: we derive a split exact sequence
\begin{equation*}
0\too \bigoplus_{i+j=3} \mathrm{H}^i(\Gamma,\mathcal{A}_{1}^\times/\mathbb{C}_p^\times)\otimes \mathrm{H}^j(\Gamma,\mathbb{Z})\too \mathrm{H}^3(\Gamma^2,\mathcal{A}_{1}^\times/\mathbb{C}_p^\times\otimes \mathbb{Z})  \too \left(\text{ finite } \right)\too 0,
\end{equation*}
and an analogous sequence for $\mathrm{H}^3(\Gamma^2,\mathbb{Z}\otimes \mathcal{A}_{1}^\times/\mathbb{C}_p^\times)$.
The right-most term is, indeed, finite because it is a direct sum of $\mathrm{Tor}$-terms with one arguments equals $\mathrm{H}^{j}(\Gamma,\mathbb{Z})$, and these are finitely generated abelian groups for all $j\geq 0$. 
Moreover, we can invoke the van der Put sequence \eqref{sequencevanderput} to show that $\mathrm{H}^0(\Gamma,\mathcal{A}_1^\times/\mathbb{C}_p^\times)=0$.
Furthermore, $\mathrm{H}^1(\Gamma,\mathbb{Z})=0$ by Margulis' theorem.
Consequently, the only possibly non-zero terms on the left-hand side are $\mathrm{H}^1(\Gamma,\mathcal{A}_{1}^\times/\mathbb{C}_p^\times)\otimes \mathrm{H}^2(\Gamma,\mathbb{Z})$ and $\mathrm{H}^3(\Gamma,\mathcal{A}_{1}^\times/\mathbb{C}_p^\times)\otimes \mathrm{H}^0(\Gamma,\mathbb{Z})$.
The cohomology cross-product commutes with the $\mathbb{T}\times \mathbb{T}$-action by a direct computation on cochains: $(\mathrm{T}_1\cdot J_1)\times (\mathrm{T}_2\cdot J_2)=(\mathrm{T}_1\times \mathrm{T}_2)\cdot (J_1\times J_2)$ (see \cite[Proposition 3.4.1]{SpreheThesis}).
Thus, the above exact sequence implies the following:
\begin{lemma}\label{lem:b}
There exists an integer $k\geq 1$ such that 
\[
 1\times \left(k\cdot\mathrm{Ann}_\mathbb{T}\left(\mathrm{H}^0(\Gamma,\mathbb{Z})\right)\cdot \mathrm{Ann}_\mathbb{T}\left(\mathrm{H}^2(\Gamma,\mathbb{Z})\right)\right)\subset \mathrm{Ann}_{\mathbb{T}\times \mathbb{T}}\left(\mathrm{H}^3(\Gamma^2,\mathcal{A}_{1}^\times/\mathbb{C}_p^\times\otimes \mathbb{Z})\right)
\]
and
\[
 \left(k\cdot\mathrm{Ann}_\mathbb{T}\left(\mathrm{H}^0(\Gamma,\mathbb{Z})\right)\cdot \mathrm{Ann}_\mathbb{T}\left(\mathrm{H}^2(\Gamma,\mathbb{Z})\right)\right)\times 1\subset \mathrm{Ann}_{\mathbb{T}\times \mathbb{T}}\left(\mathrm{H}^3(\Gamma^2,\mathbb{Z}\otimes \mathcal{A}_{1}^\times/\mathbb{C}_p^\times)\right).
\]
\end{lemma}
Combining Lemma \ref{computationofh3gamma2} and Lemma \ref{lem:b} yields:
\begin{corollary}
Let
\[
\mathrm{T}\in k\cdot e\cdot \mathrm{Ann}_\mathbb{T}\left(\mathrm{H}^0(\Gamma,\mathbb{Z})\right)\cdot \mathrm{Ann}_\mathbb{T}\left(\mathrm{H}^2(\Gamma,\mathbb{Z})\right)
\]
be a Hecke operator. Then the class $(\mathrm{T}\times 1)\cdot [\mathcal{D}_1]$ lifts to a class in $\mathrm{H}^2(\Gamma^2,\mathcal{M}_{\mathrm{rq}}^\times)$.
\end{corollary}
\begin{proof}
To simplify notation, we identify the class $[\mathcal{D}_1]$ with the cocycle $\mathcal{D}_1$. Let $\mathrm{T}=e\cdot \mathrm{T}^\prime$ with $\mathrm{T}^\prime\in k\cdot\mathrm{Ann}_\mathbb{T} \left(\mathrm{H}^0(\Gamma,\mathbb{Z})\right)\cdot \mathrm{Ann}_\mathbb{T}\left(\mathrm{H}^2(\Gamma,\mathbb{Z})\right)$.
Recall from Proposition \ref{lemmacomputation} that we have $(\mathrm{T}^\prime\times 1)\cdot \mathcal{D}_1=(1\times \mathrm{T}^\prime)\cdot \mathcal{D}_1$.
This equality and Lemma \ref{lem:b} shows that $\overline{\kappa((\mathrm{T}^\prime \times 1)\cdot \mathcal{D}_1)}=1$.  
In particular, $\kappa\left((\mathrm{T}^\prime\times 1)\cdot \mathcal{D}_1\right)$ belongs to the image of the homomorphism $\mathrm{H}^3(\Gamma^2,\mathbb{C}_p^\times)\to \mathrm{H}^3(\Gamma^2,\mathcal{A}_2^\times)$. The corollary follows, because $\mathrm{H}^3(\Gamma^2,\mathbb{C}_p^\times)$ is finite of exponent $e$.
\end{proof}
We end this discussion by defining a certain ideal of the Hecke algebra $\mathbb{T}$:
\begin{definition}\label{definitionip}
Let $\mathcal{I}_p\subset \mathbb{T}$ be the intersection of the non-trivial ideals $\mathrm{Ann}_\mathbb{T}\left(\mathrm{H}^2(\Gamma,\mathcal{A}_1^\times)\right)$ and $k\cdot e\cdot \mathrm{Ann}_\mathbb{T}\left(\mathrm{H}^0(\Gamma,\mathbb{Z})\right)\cdot \mathrm{Ann}_\mathbb{T}\left(\mathrm{H}^2(\Gamma,\mathbb{Z})\right)$. Then, by construction, for every $\mathrm{T}\in \mathcal{I}_p$ the following hold: 
\begin{enumerate}
\item[\textbullet]the class $(\mathrm{T}\times 1)\cdot [\mathcal{D}_1]$ lifts to a class with values in $\mathcal{M}_{\mathrm{rq}}^\times$ and 
\item[\textbullet]for every RM point $\omega\in \mathcal{H}_p$ the class $\mathrm{T}\cdot [\mathcal{D}_\omega]$ lifts to a class with values in $\mathcal{M}_1^\times$.
\end{enumerate}
\end{definition}

\section{The modularity theorem and classification of analytic theta cocycles}\label{modularitxtheoremchapter}
Classically, coefficients of generating series considered in the \emph{Kudla program} \cite{kudlaspecialcyclesand} take values in the first Chow group of an orthogonal Shimura variety. In the modularity theorem featured in this article, the role of the Chow group is played by the cokernel of the divisor map. That is, the replacement of the Chow group is defined by the exact sequence
\begin{equation}
\mathrm{H}^2(\Gamma^2,\mathcal{M}_{\mathrm{rq}}^\times)\too \mathrm{H}^2(\Gamma^2,\mathrm{Div}^\dagger_\mathrm{rq}(\mathcal{H}_p^2))\too \mathrm{CH}(\Gamma^2)\too 0.
\end{equation} 
The main result of this section is the modularity of the generating series of the Kudla--Millson divisors $\mathcal{D}_n$ with values in $\mathrm{CH}(\Gamma^2)$ (see Theorem \ref{modulatitytheorem}).
Its proof heavily relies on the observation that the Chow group embeds into the group $\mathrm{H}^3(\Gamma^2,\mathcal{A}_2^\times/\C_p^\times)$.
For an abelian group $A$, denote by $A_\mathbb{Q}$ the $\mathbb{Q}$-vector space $A\otimes \mathbb{Q}$. 
The long exact sequence induced from \eqref{sesforrq} yields that there is an injection $\mathrm{CH}(\Gamma^2) \to \mathrm{H}^3(\Gamma^2,\mathcal{A}_2^\times)$.
We have seen in Lemma \ref{computationofh3gamma2} that the group $\mathrm{H}^3(\Gamma^2,\mathbb{C}_p^\times)$ is finite. Thus, the natural projection map $\mathrm{H}^3(\Gamma^2,\mathcal{A}_2^\times)\to \mathrm{H}^3(\Gamma^2,\mathcal{A}_2^\times/\C_p^\times)$ has finite kernel, and we deduce that the composition
\[
\mathrm{CH}(\Gamma^2)_\Q \too \mathrm{H}^3(\Gamma^2,\mathcal{A}_2^\times)_\Q\too  \mathrm{H}^3(\Gamma^2,\mathcal{A}_2^\times/\C_p^\times)_\Q
\]
is injective. This embedding is Hecke equivariant.
The $\GL_2(\Q_p)^2$-module isomorphism (see Equation \eqref{propropropbloed})
\[
\mathcal{A}_2^\times/\C_p^\times \simeq \mathcal{A}_1^\times/\C_p^\times \otimes \Z \ \oplus \ \Z \otimes \mathcal{A}_1^\times/\C_p^\times
\]
implies that $\mathrm{H}^3(\Gamma^2, \mathcal{A}_2^\times/\mathbb{C}_p^\times)_{\mathbb{Q}}$ can be computed using the K\"unneth formula.
For this, recall that $\mathrm{H}^1(\Gamma,\Z)=0$ and that $\mathrm{H}^3(\Gamma,\Z)$ is finite (see the proof of Lemma \ref{computationofh3gamma2}). Moreover,
\begin{equation*}
\mathrm{H}^3(\Gamma,\mathcal{A}_1^\times/\C_p^\times)_\Q\simeq 
   \begin{cases} 
     0 , & \text{if $B$ is split;} \\ 
     \Q , & \text{if $B$ is non-split.}
     \end{cases}
\end{equation*}
Indeed, it fits into the exact sequence induced by the van der Put sequence \eqref{sequencevanderput}
\[
\mathrm{H}^2(\Gamma_0^p,\Z)_\Q\too \mathrm{H}^2(\Gamma_0,\Z)_\Q\oplus \mathrm{H}^2(\Gamma_0^\prime,\Z)_\Q\too \mathrm{H}^3(\Gamma,\mathcal{A}_1^\times/\C_p^\times)_\Q \too \mathrm{H}^3(\Gamma_0^p,\Z)_\Q=0.
\]
If $B$ is split, then the left-most terms are trivial, implying the triviality of the middle term. Otherwise, if $B$ is non-split, then all $\mathrm{H}^2$-terms are $1$-dimensional. Since the left-most map, a pair of corestrictions, is non-zero, which can be seen by the restriction-corestriction sequence, its cokernel must be $1$-dimensional.
Thus, the cohomology cross-products 
\begin{equation*}
\mathrm{H}^3(\Gamma,\mathcal{A}_{1}^\times/\mathbb{C}_p^\times)_\Q\otimes \mathrm{H}^0(\Gamma,\mathbb{Z})_\mathbb{Q} \ \oplus \  \mathrm{H}^1(\Gamma,\mathcal{A}_{1}^\times/\mathbb{C}_p^\times)_\Q\otimes \mathrm{H}^2(\Gamma,\mathbb{Z})_\mathbb{Q}\too \mathrm{H}^3(\Gamma^2,\mathcal{A}_{1}^\times/\mathbb{C}_p^\times\otimes \mathbb{Z})_\mathbb{Q},
\end{equation*}
and 
\begin{equation*}
\mathrm{H}^0(\Gamma,\mathbb{Z})_\Q \otimes \mathrm{H}^3(\Gamma,\mathcal{A}_{1}^\times/\mathbb{C}_p^\times)_\mathbb{Q} \ \oplus \ \mathrm{H}^2(\Gamma,\mathbb{Z})_\Q \otimes \mathrm{H}^1(\Gamma,\mathcal{A}_{1}^\times/\mathbb{C}_p^\times)_\mathbb{Q}\too \mathrm{H}^3(\Gamma^2, \mathbb{Z}\otimes \mathcal{A}_{1}^\times/\mathbb{C}_p^\times) _\mathbb{Q}
\end{equation*}
are Hecke module isomorphisms.
The $\Q$-vector space $\mathrm{H}^2(\Gamma,\mathbb{Z})_\mathbb{Q}$ is finite dimensional, say of dimension $d$. By fixing a basis of this space, and of $\mathrm{H}^3(\Gamma,\mathcal{A}_{1}^\times/\mathbb{C}_p^\times)_\mathbb{Q}$, if $B$ is non-split, the above yields an isomorphism 
\begin{equation}\label{eq:isoformodularity}
\mathrm{H}^3(\Gamma^2,\mathcal{A}_2^\times/\C_p^\times)_\Q\simeq
\begin{cases}
\mathrm{H}^1(\Gamma,\mathcal{A}_1^\times/\C_p^\times)_\Q^{ d} \ \oplus \ \mathrm{H}^1(\Gamma,\mathcal{A}_1^\times/\C_p^\times)_\Q^{ d}, & \text{if $B$ is split},\\[6pt]
\begin{aligned}[t]
&\mathrm{H}^1(\Gamma,\mathcal{A}_1^\times/\C_p^\times)_\Q^{ d} \ \oplus \ \mathrm{H}^1(\Gamma,\mathcal{A}_1^\times/\C_p^\times)_\Q^{ d}\\
& \oplus \ \mathrm{H}^0(\Gamma,\Z)_\Q \ \oplus \ \mathrm{H}^0(\Gamma,\Z)_\Q,
\end{aligned}
& \text{if $B$ is non-split}.
\end{cases}
\end{equation}
In the end, the modularity theorem follows from a classification result for the group $\mathrm{H}^1(\Gamma,\mathcal{A}_{1}^\times/\mathbb{C}_p^\times)$ (Proposition \ref{prop:splittinf} and Proposition \ref{prop:classificationsplit}), which we discuss in the next subsection.
\subsection{Classification of analytic theta cocycles}
We use the notation introduced in Section \ref{liftingobstructiondv}. In particular, we denote by $\Gamma_0,\Gamma_0^\prime,\Gamma_0^p\subset \Gamma$ the arithmetic subgroups that are the stabilizers of the two vertices $v_0$ and $v_1$ and the stabilizer of the edge $e_0$ in $\Gamma$, respectively.
Recall that we have the van der Put exact sequence (Equation \eqref{sequencevanderput})
\[
0\too \mathcal{A}_{1}^\times/\mathbb{C}_p^\times\too \mathcal{F}_0(\mathcal{E}^{\mathrm{or}},\mathbb{Z})\too \mathcal{F}(\mathcal{V},\mathbb{Z})\too 0.
\]
By Shapiro's Lemma, this yields an exact sequence
\[
0\too L \too \mathrm{H}^1(\Gamma,\mathcal{A}_1^\times/\C_p^\times)\too N \too 0,
\]
with 
\[
L\coloneq \mathrm{coker}\left( \Z \too \Z\oplus \Z, \quad z\mapsto ((p+1)z,-(p+1)z))\right) \ \simeq \ \Z \oplus \Z/(p+1)\Z,
\]
and
\[
N\coloneq \mathrm{H}^1(\Gamma_0^p,\Z)^{p\text{-}\mathrm{new}}\coloneq \mathrm{ker}\left(\mathrm{H}^1(\Gamma_0^p,\Z)\too \mathrm{H}^1(\Gamma_0,\Z)\oplus \mathrm{H}^1(\Gamma_0^\prime,\Z)\right).
\] 
The explicit formula for the map defining $L$ follows from the fact that the index of $\Gamma_0^p$ in both $\Gamma_0$ and $\Gamma_0^\prime$ is $p+1$, and that the edges with source $v_1$ belong to the $\Gamma$-orbit of $\overline{e_0}$.
Tensoring with $\mathbb{Q}$ yields a short exact sequence of finite-dimensional $\mathbb{Q}$-vector spaces
\begin{equation}\label{eq:vdpcohom}
0\too L_\Q \xrightarrow{\iota} \mathrm{H}^1(\Gamma,\mathcal{A}_1^\times/\C_p^\times)_\Q\xrightarrow{\lambda} N_\Q\too 0.
\end{equation}
The group of analytic theta cocycles $\mathrm{H}^1(\Gamma,\mathcal{A}_1^\times/\C_p^\times)_\Q$ is endowed with an action of Hecke operators $\mathrm{T}_n$ for all $n\geq 1$.
Moreover, if $(n,p)=1$, then $\mathrm{T}_n$ acts on all cohomology groups $\mathrm{H}^i(\Gamma_0,\Z)$, $\mathrm{H}^i(\Gamma_0^\prime,\Z)$ and $\mathrm{H}^i(\Gamma_0^p,\Z)$ for $i\geq 0$, and the long exact sequence induced from the van der Put exact sequence commutes with $\mathrm{T}_n$.
For the Hecke operator at $p$, let $W_p\in R$ be an element of reduced norm $p$ that normalizes $\Gamma_0^p$.
The element $W_p$ preserves the unoriented edge connecting $v_0$ and $v_1$, and since $\mathrm{Nrd}(W_p)=p$, $W_p$ interchanges $v_0$ and $v_1$.  
Then $\mathrm{T}_p=\Gamma\cdot W_p\cdot \Gamma$ is the Hecke operator at $p$ for the cohomology of $\Gamma$, and $\Gamma\backslash \Gamma \cdot W_p\cdot  \Gamma=\{W_p\}$.
The element $W_p$ also induces the Atkin--Lehner involution $w_p\colon \mathrm{H}^i(\Gamma_0^p,\mathbb{Q})\to \mathrm{H}^i (\Gamma_0^p,\mathbb{Q})$. That is, $w_p$ is the double coset operator $w_p=\Gamma_0^p\cdot W_p\cdot \Gamma_0^p$, and $\Gamma_0^p\backslash \Gamma_0^p \cdot W_p\cdot  \Gamma_0^p=\{W_p\}$.
We prove the following lemma:
\begin{lemma}\label{lem:heckediagramm}
The following diagram commutes:
\begin{center}
\begin{minipage}{\linewidth}
\centering
\begin{tikzcd}
0 \arrow[r] & L_\Q \arrow[r, "\iota"] \arrow[d, "\mathrm{id}"] &  \mathrm{H}^1(\Gamma,\mathcal{A}_1^\times/\C_p^\times)_\Q\arrow[r, "\lambda"] \arrow[d, "\mathrm{T}_p"]  &  N_\Q \arrow[d, "-w_p"] \arrow[r] & 0 \\
0 \arrow[r] & L_\Q \arrow[r, "\iota"]  &  \mathrm{H}^1(\Gamma,\mathcal{A}_1^\times/\C_p^\times)_\Q  \arrow[r, "\lambda"] &  N_\Q  \arrow[r] & 0.
 \end{tikzcd}
\end{minipage}
\end{center} 
\end{lemma}
\begin{proof}
Since $W_p$ interchanges the vertices $v_0$ and $v_1$, we have $W_p^{-1}e_0=\overline{e_0}$.
This fact implies the commutativity of the diagram:
\begin{center}
\begin{minipage}{\linewidth}
\centering
\begin{tikzcd}
\mathrm{H}^i(\Gamma,\mathcal{F}_0(\mathcal{E}^{\mathrm{or}},\Z)) \arrow[r, "\mathrm{T}_p"] \arrow[d, "\simeq"]   & \mathrm{H}^i(\Gamma,\mathcal{F}_0(\mathcal{E}^{\mathrm{or}},\Z)) \arrow[d, "\simeq"] \\
\mathrm{H}^i(\Gamma_0^p,\Z)\arrow[r, "-w_p"] & \mathrm{H}^i(\Gamma_0^p,\Z).
 \end{tikzcd}
\end{minipage}
\end{center} 
Indeed, the Shapiro isomorphism is induced by the map $\varphi\mapsto \varphi(e_0)$, and $\varphi(W_p^{-1}e_0)=-\varphi(e_0)$. The commutativity of the right square follows.
To determine the action on the cohomology of $\Gamma_0$ and $\Gamma_0^\prime$, note that $\Gamma_0^\prime=W_p\Gamma_0W_p^{-1}$. 
Thus, there is a natural isomorphism $\sigma\colon \mathrm{H}^i(\Gamma_0,\Z)\to \mathrm{H}^i(\Gamma_0^\prime,\Z)$ induced by conjugation with $W_p$. With a slight abuse of notation, we denote its inverse by $\sigma$.
Then the following diagram is commutative:
\begin{center}
\begin{minipage}{\linewidth}
\centering
\begin{tikzcd}
\mathrm{H}^i(\Gamma,\mathcal{F}(\mathcal{V},\Z)) \arrow[rr]{rr}{\mathrm{T}_p} \arrow[d]{d}{\simeq}  & & \mathrm{H}^i(\Gamma,\mathcal{F}(\mathcal{V},\Z)) \arrow[d]{d}{\simeq} \\
\mathrm{H}^i(\Gamma_0,\Z)\oplus \mathrm{H}^i(\Gamma_0^\prime,\Z) \arrow[rr]{rr}{t} & &  \mathrm{H}^i(\Gamma_0,\Z)\oplus \mathrm{H}^i(\Gamma_0^\prime,\Z),
 \end{tikzcd}
\end{minipage}
\end{center} 
where $t(x,y)=(\sigma(y),\sigma(x))$.
In degree $0$, this implies the commutative diagram 
\begin{center}
\begin{minipage}{\linewidth}
\centering
\begin{tikzcd}
\mathrm{H}^0(\Gamma_0^p,\Z)_\Q \arrow[r]{r}{\partial} \arrow[d]{d}{-\mathrm{id}}  & \mathrm{H}^0(\Gamma_0,\Z)_\Q\oplus \mathrm{H}^0(\Gamma_0^\prime,\Z)_\Q\arrow[d]{d}{t} \arrow[rr]{rr}{(x,y)\mapsto x+y} & & L_\Q\arrow[d]{d}{\mathrm{id}}  \\
\mathrm{H}^0(\Gamma_0^p,\Z)_\Q \arrow[r]{r}{\partial}  & \mathrm{H}^0(\Gamma_0,\Z)_\Q\oplus \mathrm{H}^0(\Gamma_0^\prime,\Z)_\Q \arrow[rr]{rr}{(x,y)\mapsto x+y} & & L_\Q.
 \end{tikzcd}
\end{minipage}
\end{center} 
The lemma follows. 
\end{proof}
We show that the exact sequence \eqref{eq:vdpcohom} induces a decomposition of the group of analytic theta cocycles into Hecke stable subspaces. The central idea for this is to define a Hecke operator that annihilates $N_\Q$ and acts injectively on the line $L_\Q$. 
We then use the following lemma:
\begin{lemma}\label{lem:decompositiongeneral}
Let 
\[
0 \too \widetilde{L}\xrightarrow{\iota} \Omega\xrightarrow{\lambda} \widetilde{N}\too 0
\]
be a short exact sequence of finite dimensional $\Q$-vector spaces. We assume that these spaces are endowed with an action of a commutative Hecke algebra $\mathbb{T}^\prime$, and that the sequence is compatible with the $\mathbb{T}^\prime$-action.
Furthermore, we assume that there is a Hecke operator $\mathrm{T}_l$ that satisfies the following conditions:
\begin{enumerate}
\item The Hecke operator $\mathrm{T}_l$ acts on $\widetilde{L}$ by multiplication by $l+1$, and
\item $l+1$ is not an eigenvalue of $\mathrm{T}_l$ acting on $\widetilde{N}$.
\end{enumerate}
Then there is a Hecke stable decomposition
\[
\Omega=\Omega^\mathrm{cusp} \ \oplus \  \Omega^\mathrm{Eis},
\]
and the maps $\iota$ and $\lambda$ induce isomorphisms
\[
\iota\colon \widetilde{L} \too \Omega^{\mathrm{Eis}} \qquad \text{and} \quad \lambda\colon \Omega^{\mathrm{cusp}} \too \widetilde{N}.
\]
The decomposition does not depend on the choice of $\mathrm{T}_l$.
\end{lemma}
\begin{proof}
Let $P\in \Q[X]$ be the characteristic polynomial of $\mathrm{T}_l$ acting on $\widetilde{N}$, and put $\mathrm{T}\coloneq P(\mathrm{T}_l)$.
The theorem of Cayley--Hamilton yields that $\mathrm{T}$ annihilates $\widetilde{N}$.
Thus, $\mathrm{T}\cdot \Omega$ is contained in $\iota(\widetilde{L})$, and we obtain inclusions 
\[
\mathrm{T}\cdot \iota(\widetilde{L})\subset  \mathrm{T}\cdot \Omega  \subset \iota(\widetilde{L}).
\]
By condition (1), we have $\mathrm{T}=P(l+1)$ on $\iota(\widetilde{L})$, and by condition (2), $c\coloneq P(l+1)\neq 0$.
Consequently, all inclusions are identities.
We define 
\[
\Omega^{\mathrm{Eis}}\coloneq \iota(\widetilde{L})= \mathrm{im}(\mathrm{T}), \qquad \Omega^{\mathrm{cusp}}\coloneq \mathrm{ker}(\mathrm{T}).
\] 
Note that these subspaces are Hecke stable because of the commutativity of $\mathbb{T}^\prime$.
Observe that $\mathrm{T}^2=c\cdot \mathrm{T}$ on $\Omega$. 
Indeed, if $x\in \Omega$, then $\mathrm{T}\cdot x\in \iota(\widetilde{L})$ by the observation above. The claim follows because $\mathrm{T}=c$ on $\widetilde{L}$.
Thus, $\mathrm{im}(\mathrm{T}^2)=\mathrm{im}(\mathrm{T})$, and, hence
\[
\Omega= \mathrm{im}(\mathrm{T})\oplus \mathrm{ker}(\mathrm{T})=\Omega^{\mathrm{Eis}} \ \oplus \ \Omega^{\mathrm{cusp}}.
\] 
It remains to show that the map $\lambda\colon \Omega^{\mathrm{cusp}}\to \widetilde{N}$ is an isomorphism: it is injective because $\mathrm{ker}(\lambda)=\mathrm{im}(\iota)=\mathrm{im}(\mathrm{T})$ and $\mathrm{im}(\mathrm{T})\cap \mathrm{ker}(\mathrm{T})=0$. Surjectivity follows from the surjectivity of $\lambda$, and the fact that $\lambda(\mathrm{im}(\mathrm{T}))=0$. 
To show the independence on $\mathrm{T}_l$, suppose there is another Hecke operator $\mathrm{T}_{l^\prime}$ satisfying (1) and (2). We are reduced to consider the cuspidal part, as the Eisenstein part is $\iota(\widetilde{L})$ in both cases. Denote by $\mathrm{T}^\prime=P^\prime(\mathrm{T}_{l^\prime})$ the Hecke operator constructed above, but for $l^\prime$. Then $\lambda(\mathrm{ker}(\mathrm{T}^\prime))=\widetilde{N}$. Since $\mathrm{T}$ annihilates $\widetilde{N}$, this implies that $\mathrm{T}\cdot \mathrm{ker}(\mathrm{T}^\prime)\subset \mathrm{ker}(\lambda)$. Now, $\mathrm{T}\cdot \mathrm{ker}(\mathrm{T}^\prime)\subset \mathrm{ker}(\mathrm{T}^\prime)$, since the Hecke algebra is commutative, and $\lambda$ is injective on this subspace. We derive that $\mathrm{ker}(\mathrm{T}^\prime)\subset \mathrm{ker}(\mathrm{T})$. In conclusion, we obtain equality by dimension reasons.
The lemma follows.
\end{proof}
The Hecke module $N_\Q$ can be described in terms of modular forms. To do so, we distinguish between the cases where $B$ is non-split and where $B$ is split.
We fix the following notation for the remainder of this section. 
Let $M\coloneq p\cdot d_B$, in particular $M=p$, if $B$ is split.
We denote by $\mathcal{M}_2(\Gamma_0(M))$ the $\C$-vector space of modular forms of weight $2$ and level $\Gamma_0(M)$, and by $\mathcal{S}_2(\Gamma_0(M))$ the subspace of cusp forms. 
Moreover, we denote by $\mathcal{M}_2(\Gamma_0(M),\Q)\subset \mathcal{M}_2(\Gamma_0(M))$ the $\Q$-vector space of modular forms with rational Fourier coefficients, and by $\mathcal{S}_2(\Gamma_0(M),\Q)\subset \mathcal{M}_2(\Gamma_0(M),\Q)$ the subspace of cusp forms.
For primes $l\mid M$ we write $\mathrm{T}_l$ for the operator $U_l$ on $\mathcal{M}_2(\Gamma_0(M))$.
Both spaces $\mathcal{M}_2(\Gamma_0(M),\Q)$ and $\mathcal{S}_2(\Gamma_0(M),\Q)$ are stable under the action of the Hecke operators $\mathrm{T}_n$ for all $n\geq 1$. 
Let $\mathcal{S}_2(\Gamma_0(M),\Q)^{\mathrm{new}}$ be the $\mathbb{Q}$-subspace of the space of newforms $\mathcal{S}_2(\Gamma_0(M))^{\mathrm{new}}$, that is, the orthogonal complement of the space of oldforms, consisting of cusp forms with rational Fourier coefficients.
\subsubsection{The non-split case}\label{sec:classificationnonsplit}
Assume that $B$ is non-split.
We denote by $\mathbb{T}(N_\Q)\subset \mathrm{End}(N_\Q)$ the $\Q$-algebra generated by the Hecke operators $\mathrm{T}_n$ for $(n,p)=1$ and the Atkin--Lehner involution $w_p$, and by $\mathbb{T}(\mathcal{S}_2(\Gamma_0(M),\Q)^{\mathrm{new}})$ the $\Q$-algebra generated by the Hecke operators $\mathrm{T}_n$ acting on $\mathcal{S}_2(\Gamma_0(M),\Q)^{\mathrm{new}}$. We have the following lemma.
\begin{lemma}\label{lem:heckeiso}
There is an isomorphism 
\[\mathbb{T}(N_\Q)\to \mathbb{T}(\mathcal{S}_2(\Gamma_0(M),\Q)^{\mathrm{new}})
\]
with $\mathrm{T}_{n}\mapsto \mathrm{T}_{n}$ for $(n,p)=1$ and $w_p\mapsto -\mathrm{T}_p$.
\end{lemma}
\begin{proof}
Eichler--Shimura theory implies that there are isomorphisms
\[
\mathrm{ES}\colon \mathrm{H}^1(\Gamma_0^p,\C)\simeq \mathcal{S}_2(\Gamma_0^p)\oplus \overline{\mathcal{S}}_2(\Gamma_0^p),
\]
and 
\[
\mathrm{H}^1(\Gamma_0,\C)\simeq \mathcal{S}_2(\Gamma_0)\oplus \overline{\mathcal{S}}_2(\Gamma_0), \qquad \mathrm{H}^1(\Gamma_0^\prime,\C)\simeq \mathcal{S}_2(\Gamma_0^\prime)\oplus \overline{\mathcal{S}}_2(\Gamma_0^\prime),
\]
which are compatible with the action of double coset operators on source and target. Here, $\mathcal{S}_2(-)$ is the space of cusp forms of weight $2$ in the sense of \cite[Definition 4.7]{darmonrationalpoints}.
The homomorphism in cohomology 
\[
\mathrm{H}^1(\Gamma_0^p,\C)\too \mathrm{H}^1(\Gamma_0,\C)\oplus \mathrm{H}^1(\Gamma_0^\prime,\C)
\]
corresponds on the modular form side to the map $f\mapsto \big(\mathrm{Tr}^{\Gamma_0}_{\Gamma_0^p}(f),\ -\mathrm{Tr}^{\Gamma_0^\prime}_{\Gamma_0^p}(f)\big)$.
Thus, $\mathrm{ES}$ induces an isomorphism 
\[
N_\Q \otimes \C \simeq \mathcal{S}_2(\Gamma_0^p)^{p\text{-}\mathrm{new}} \ \oplus \ \overline{\mathcal{S}}_2(\Gamma_0^p)^{p\text{-}\mathrm{new}},
\]
where $\mathcal{S}_2(\Gamma_0^p)^{p\text{-}\mathrm{new}}$ is the intersection of the kernels of the trace maps.
Clearly, complex conjugation induces an isomorphism between $\mathcal{S}_2(\Gamma_0^p)^{p\text{-}\mathrm{new}}$ and $\overline{\mathcal{S}}_2(\Gamma_0^p)^{p\text{-}\mathrm{new}}$, and this isomorphism is compatible with all Hecke operators $\mathrm{T}_n$ for $(n,p)=1$ and the Atkin--Lehner involution.
Finally, by Jacquet--Langlands there is an isomorphism
\[
\mathcal{S}_2(\Gamma_0^p)^{p\text{-}\mathrm{new}}\simeq \mathcal{S}_2(\Gamma_0(M))^{\mathrm{new}},
\]
which is compatible with the Hecke operators $\mathrm{T}_n$ for $(n,p)=1$ and the Atkin--Lehner involution $w_p$.
\cite[Theorem 3]{AtkinLehner} implies that on the space of newforms of level $\Gamma_0(M)$, we have $w_p=-\mathrm{T}_p$.
Combining the discussion above, we have proven that there is an isomorphism 
\[
\mathbb{T}(N_\Q\otimes \C) \simeq \mathbb{T}(\mathcal{S}_2(\Gamma_0(M))^{\mathrm{new}}),
\]
that satisfies the statement of the lemma.
Since this isomorphism maps generators to generators, and since we have
\[
\mathcal{S}_2(\Gamma_0(M))^{\mathrm{new}}=\mathcal{S}_2(\Gamma_0(M),\Q)^{\mathrm{new}}\otimes \C,
\]
the lemma follows from this isomorphism.
\end{proof}
We now use Lemma \ref{lem:decompositiongeneral} to show that the sequence \eqref{eq:vdpcohom} induces a Hecke stable splitting of the group of theta cocycles:
\begin{proposition}\label{prop:splittinf}
There is a Hecke stable splitting
\[
\mathrm{H}^1(\Gamma,\mathcal{A}_1^\times/\C_p^\times)_\Q=\mathrm{H}^1(\Gamma,\mathcal{A}_1^\times/\C_p^\times)_\Q^{\mathrm{Eis}} \ \oplus \ \mathrm{H}^1(\Gamma,\mathcal{A}_1^\times/\C_p^\times)_\Q^{\mathrm{cusp}}  
\]
and isomorphisms
\[
\mathrm{H}^1(\Gamma,\mathcal{A}_1^\times/\C_p^\times)_\Q^{\mathrm{Eis}} \simeq L_\Q, \qquad \mathrm{H}^1(\Gamma,\mathcal{A}_1^\times/\C_p^\times)_\Q^{\mathrm{cusp}} \simeq N_\Q,
\]
which are compatible with all Hecke operators $\mathrm{T}_n$ for $(n,p)=1$. Moreover, the Hecke operator $\mathrm{T}_p$ corresponds to $-w_p$ on $N_\Q$ and to the identity on $L_\Q$.
\end{proposition}
\begin{proof}
The splitting is obtained by finding a suitable Hecke operator $\mathrm{T}_l$ that satisfies the conditions in Lemma \ref{lem:decompositiongeneral}.
We claim that $\mathrm{T}_l$ for any prime $l\nmid M$ does the job:
Indeed, since $N_\Q\otimes \C$ is isomorphic, as a Hecke module, to two copies of $\mathcal{S}_2(\Gamma_0(M))^{\mathrm{new}}$, the Ramanujan bound implies that $\mathrm{T}_l$ cannot act with eigenvalue $l+1$ on $N_\Q$.
More precisely, if $a_l$ is an eigenvalue of $\mathrm{T}_l$ for a cusp form of level $\Gamma_0(M)$, then $|a_l|\le2\sqrt l<l+1$. Condition (2) follows.
For condition (1), recall that $\mathrm{H}^0(\Gamma_0,\Z)=\Z=\mathrm{H}^0(\Gamma_0^\prime,\Z)$, and that $L$ is a quotient of the direct sums of these groups. Thus, $\mathrm{T}_l$ acts on these cohomology groups by multiplication by the number of right-cosets, which is $l+1$. We obtain condition (1).
Lemma \ref{lem:decompositiongeneral} implies the desired splitting. 
Finally, the compatibility with the operators $\mathrm{T}_n$ for $(n,p)=1$ and $\mathrm{T}_p$ follows from the commutative diagram in Lemma \ref{lem:heckediagramm}.
\end{proof}

\subsubsection{The split case}
Let $B=\mathrm{M}_2(\Q)$ be split. We assume that $\Gamma_0^p=\Gamma_0(p)$ and $\Gamma_0=\SL_2(\Z)$. Thus, $\mathrm{H}^1(\Gamma_0,\Z)=0=\mathrm{H}^1(\Gamma_0^\prime,\Z)$, and, hence, $N=\mathrm{H}^1(\Gamma_0(p),\Z)$. 
The group $N_\Q$ decomposes as
\[
N_\Q=\mathrm{H}^1_{\mathrm{par}}(\Gamma_0(p),\Q) \oplus \Q\cdot \varphi_{\mathrm{DR}},
\]
where 
\begin{enumerate}
\item[\textbullet]$\mathrm{H}^1_{\mathrm{par}}(\Gamma_0(p),\Q)$ is the parabolic cohomology of $\Gamma_0(p)$, that is, the group of classes that vanish after restricting to the stabilizers in $\Gamma_0(p)$ of the cusps, and
\item[\textbullet] $\varphi_{\mathrm{DR}}$ is the class induced by the Dedekind--Rademacher homomorphism (see \cite[Equation (3)]{DPV2}).
\end{enumerate}
Indeed, first of all note that $\varphi_\mathrm{DR}$ is not parabolic, because its restriction to $\begin{psmallmatrix} 1 & 1 \\ 0 & 1\end{psmallmatrix}^\Z\subset \Gamma_0(p)_\infty$ is non-zero. 
The equality follows then by computing dimensions. This is done via Eichler--Shimura theory: there are $\mathbb{C}$-linear isomorphisms
\[
\mathrm{ES}\colon \mathcal{M}_2(\Gamma_0(p)) \oplus \overline{\mathcal{S}}_2(\Gamma_0(p)) \too \mathrm{H}^1(\Gamma_0(p),\mathbb{C})\simeq N_\Q\otimes \C,
\]
and
\[
\mathcal{S}_2(\Gamma_0(p)) \oplus \overline{\mathcal{S}}_2(\Gamma_0(p)) \too \mathrm{H}^1_{\mathrm{par}}(\Gamma_0(p),\mathbb{C})\simeq  \mathrm{H}^1_{\mathrm{par}}(\Gamma_0(p),\Q)\otimes \C.
\]
Thus, $\mathrm{dim}(N_\Q)=\mathrm{dim}(\mathrm{H}^1_{\mathrm{par}}(\Gamma_0(p),\Q))+1$, and we derive the claimed identity. By its definition, $\varphi_\mathrm{DR}$ corresponds to a non-zero multiple of the weight $2$ Eisenstein series
\[
E_2^{(p)}(q)=(p-1)+24\sum_{n\geq 1}\sigma_1^{(p)}(n)\,q^n,\qquad \sigma_1^{(p)}(n)=\sum_{\substack{t\mid n\\(t,p)=1}} t.
\]
See \cite[Equations (2) and (3)]{DPV2}).
The spaces $\mathrm{H}^1_{\mathrm{par}}(\Gamma_0(p),\mathbb{Q})$ and $\Q\cdot \varphi_\mathrm{DR}$ are Hecke stable: $\mathrm{H}^1_{\mathrm{par}}(\Gamma_0(p),\mathbb{Q})$ is stable by Eichler--Shimura, and $\Q\cdot \varphi_\mathrm{DR}$ because $\varphi_{\mathrm{DR}}$ is an eigenvector for all Hecke operators as it corresponds to the Eisenstein series $E_2^{(p)}$.
In \cite[Section 1]{DPV2}, Darmon--Pozzi--Vonk construct the so-called \emph{Dedekind--Rademacher cocycle} $J_\mathrm{DR}\in \mathrm{H}^1(\Gamma,\mathcal{A}_1^\times/\C_p^\times)$. One of the key properties of this class is that $\lambda(J_\mathrm{DR})$ is a non-zero multiple of  $\varphi_\mathrm{DR}$. See \cite[Lemma 2.1 and Equations (27) - (29)]{DPV2}.
Moreover, it is a Hecke eigenclass for all Hecke operators $\mathrm{T}_n$, and its eigenvalues are precisely the ones of the Eisenstein series $E_2^{(p)}$ (see \cite[Lemma 2.1 and Equation (30)]{DPV2}).
We now prove the counterpart of Proposition \ref{prop:splittinf} in the split case:
\begin{proposition}\label{prop:classificationsplit}
There is a Hecke stable splitting
\[
\mathrm{H}^1(\Gamma,\mathcal{A}_1^\times/\C_p^\times)_\Q= \mathrm{H}^1(\Gamma,\mathcal{A}_1^\times/\C_p^\times)_\Q^{\mathrm{cusp}} \oplus \mathrm{H}^1(\Gamma,\mathcal{A}_1^\times/\C_p^\times)_\Q^{\mathrm{Eis}}, 
\]
with
\begin{equation*}
\mathrm{H}^1(\Gamma,\mathcal{A}_1^\times/\C_p^\times)_\Q^\mathrm{Eis}=\iota(L_\Q) \oplus \Q\cdot J_\mathrm{DR}.
\end{equation*}
On the Eisenstein part, every Hecke operator $\mathrm{T}_n$ for $n\geq 1$ transforms with the same eigenvalue as on the Eisenstein series $E_2^{(p)}$.
Moreover, $\lambda$ induces an isomorphism
\[
\lambda\colon \mathrm{H}^1(\Gamma,\mathcal{A}_1^\times/\C_p^\times)_\Q^{\mathrm{cusp}}\too \mathrm{H}^1_\mathrm{par}(\Gamma_0(p),\Q),
\]
which is compatible with all Hecke operators $\mathrm{T}_n$ for $(n,p)=1$, and $\mathrm{T}_p$ corresponds to $-w_p$.
\end{proposition}
\begin{proof}
Put $\mathrm{H}^1(\Gamma,\mathcal{A}_1^\times/\C_p^\times)_\Q^\mathrm{Eis}\coloneq\lambda^{-1}(\Q\cdot\varphi_\mathrm{DR})$. Then we have the following short exact sequence
\[
0\too L_\Q \xrightarrow{\iota} \mathrm{H}^1(\Gamma,\mathcal{A}_1^\times/\C_p^\times)_\Q^\mathrm{Eis} \xrightarrow{\lambda} \Q\cdot \varphi_\mathrm{DR} \too 0.
\] 
We deduce the decomposition
\begin{equation*}
\mathrm{H}^1(\Gamma,\mathcal{A}_1^\times/\C_p^\times)_\Q^\mathrm{Eis}=\iota(L_\Q) \oplus \Q\cdot J_\mathrm{DR},
\end{equation*}
because $\lambda(J_\mathrm{DR})\in \Q\cdot \varphi_{\mathrm{DR}}\setminus \{0\}$.
This splitting is Hecke stable because any Hecke operator $\mathrm{T}_n$ acts on $L_\Q$ and on $\Q\cdot J_\mathrm{DR}$ with the same eigenvalue as for the Eisenstein series  $E_2^{(p)}$.
By the definition of the Eisenstein part, we have the exact sequence:
\[
0\too \mathrm{H}^1(\Gamma,\mathcal{A}_1^\times/\C_p^\times)_\Q^\mathrm{Eis}\too \mathrm{H}^1(\Gamma,\mathcal{A}_1^\times/\C_p^\times)_\Q \too N_\Q/\Q\cdot\varphi_\mathrm{DR}\simeq\mathrm{H}^1_\mathrm{par}(\Gamma_0(p),\Q)\too 0,
\]
and the right-most isomorphism is compatible with the action of the Hecke operators.
The proposition now concludes by applying Lemma \ref{lem:decompositiongeneral}. We claim that any Hecke operator $\mathrm{T}_l$ for $l\neq p$ satisfies the conditions (1) and (2). 
Indeed, condition (2) holds by the Ramanujan bound, and condition (1) holds because $\mathrm{T}_l=l+1$ on $L_\Q$ and on $\Q\cdot J_\mathrm{DR}$ as $l+1$ is the eigenvalue of $\mathrm{T}_l$ on $E_2^{(p)}$.
The desired splitting follows from Lemma \ref{lem:decompositiongeneral}. The compatibility for the isomorphism on the cuspidal part is Lemma \ref{lem:heckediagramm}. The proposition follows.
\end{proof}

\subsection{The modularity theorem}
We now prove the modularity theorem (Theorem D) from the introduction.
To simplify notation, we identify every cocycle $\mathcal{D}_n$ with its associated class and, moreover, denote by $\mathcal{D}_n$ the image in the Chow group. \\
Recall that if $V$ is a $\mathbb{Q}$-vector space and $\Lambda\subset \SL_2(\mathbb{Z})$ is a congruence subgroup, then a generating series $\sum_{n\geq  0} v_n\cdot q^n\in V\left[[q]\right]$ is said to be \emph{modular of weight $k$ and level $\Lambda$} if for every $\mathbb{Q}$-linear map $\alpha\colon V \to \mathbb{Q}$ the series
\[
\sum_{n\geq  0} \alpha(v_n)\cdot q^n\in \mathbb{Q}\left[[q]\right]
\]
is the $q$-expansion of a modular form of weight $k$ and level $\Lambda$. 
We can now state the main theorem of this section:
\begin{theorem}\label{modulatitytheorem}
There is a unique constant term $\mathcal{D}_0\in \mathrm{CH}(\Gamma^2)_\mathbb{Q}$ such that the generating series
\[
\mathcal{D}_0+\sum_{n\geq 1} \  \mathcal{D}_n \cdot q^n\in \mathrm{CH}(\Gamma^2)_\mathbb{Q} \left[[q]\right]
\]
is modular of weight $2$ and level $\Gamma_0(M)$.
\end{theorem}
If $\mathcal{D}_0$ exists, then it is always unique, as there is no non-trivial constant modular form of weight $2$ and level $\Gamma_0(M)$.
Moreover, we have the embedding (see for this the beginning of this section)
\[
\mathrm{CH}(\Gamma^2)_\Q \subset \mathrm{H}^3(\Gamma^2,\mathcal{A}_2^\times/\C_p^\times)_\Q,
\] 
and the desired modularity in $\mathrm{CH}(\Gamma^2)_\Q$ is equivalent to the modularity in $\mathrm{H}^3(\Gamma^2,\mathcal{A}_2^\times/\C_p^\times)_\Q$. 
The right-most term has a K\"unneth decomposition (see Equation \eqref{eq:isoformodularity}), and the theorem follows if we can prove the modularity on all summands.
First of all, consider the summands
\[
\mathrm{H}^1(\Gamma,\mathcal{A}_1^\times/\C_p^\times)_\Q\otimes \mathrm{H}^2(\Gamma,\Z)_\Q \quad \text{and} \quad \mathrm{H}^2(\Gamma,\Z)_\Q\otimes \mathrm{H}^1(\Gamma,\mathcal{A}_1^\times/\C_p^\times)_\Q.
\]
Let $(v_1,\cdots, v_d)$ be the fixed basis of $\mathrm{H}^2(\Gamma,\Z)_\Q$.
By Lemma \ref{lemmacomputation}, $\mathcal{D}_n=(\mathrm{T}_n\times 1)\cdot \mathcal{D}_1=(1\times \mathrm{T}_n) \cdot \mathcal{D}_1$ for all $n\geq 1$.
Therefore, if $J_1=\sum_{k=1}^d J_1^{(k)} \otimes v_k$ is the component of $\mathcal{D}_1$ on $\mathrm{H}^1(\Gamma,\mathcal{A}_1^\times/\C_p^\times)_\Q\otimes \mathrm{H}^2(\Gamma,\Z)_\Q$, then the component of $\mathcal{D}_n$ on this summand equals
\[
J_n=(\mathrm{T}_n\times 1)\cdot J_1=\sum_{k=1}^d (\mathrm{T}_n \cdot J_1^{(k)}) \otimes v_k.
\] 
Analogously, by switching the roles of $\mathrm{H}^1(\Gamma,\mathcal{A}_1^\times/\C_p^\times)_\Q$ and $\mathrm{H}^2(\Gamma,\Z)_\Q$, we see that every component on the right-most summand is of the form
\[
J_n'=(1\times \mathrm{T}_n)\cdot J_1'=\sum_{k=1}^d  v_k \otimes (\mathrm{T}_n \cdot J_1'^{(k)}).
\]
In conclusion, the modularity on these summands follows if we can prove that any generating series of the form 
\begin{equation*}
\sum_{n\geq 1} (\mathrm{T}_n\cdot J_1) \cdot q^n \in \mathrm{H}^1(\Gamma,\mathcal{A}_{1}^\times/\mathbb{C}_p^\times)_\Q \left[[q]\right]
\end{equation*}
admits a constant term such that it is modular.
If $B$ is non-split, then we also need to consider the summands 
\[
\mathrm{H}^3(\Gamma,\mathcal{A}_1^\times/\C_p^\times)_\Q\otimes \mathrm{H}^0(\Gamma,\Z)_\Q \quad \text{and} \quad \mathrm{H}^0(\Gamma,\Z)_\Q\otimes \mathrm{H}^3(\Gamma,\mathcal{A}_1^\times/\C_p^\times)_\Q.
\]
By applying the same arguments as above, though we interchange the roles of $\mathcal{A}_1^\times/\C_p^\times$ and $\Z$, we can reduce the desired modularity to showing that any generating series
\[
\sum_{n\geq 1} (\mathrm{T}_n\cdot u_1) \cdot q^n \in \mathrm{H}^0(\Gamma,\Z)_\Q \left[[q]\right]
\]
admits a constant term such that it is modular.
In the end, the modularity of the series above is a consequence of the Hecke module structure of the involved cohomology groups.
Let $\mathbb{T}(\Gamma_0(M))$ denote the $\Q$-algebra generated by all Hecke operators $\mathrm{T}_n$ acting on the space $\mathcal{M}_2(\Gamma_0(M),\Q)$.
Analogously, let $\mathbb{T}(\mathrm{H}^1(\Gamma,\mathcal{A}_1^\times/\C_p^\times)_\Q)$ and $\mathbb{T}(\mathrm{H}^0(\Gamma,\Z)_\Q)$ denote the $\Q$-algebra generated by all Hecke operators $\mathrm{T}_n$ acting on $\mathrm{H}^1(\Gamma,\mathcal{A}_1^\times/\C_p^\times)_\Q$ and on $\mathrm{H}^0(\Gamma,\Z)_\Q$, respectively. 
We have the following:
\begin{corollary}\label{cor:heckeinnonsplit}
There is a $\Q$-linear map
\[
\mathbb{T}(\Gamma_0(M))\too \mathbb{T}(\mathrm{H}^1(\Gamma,\mathcal{A}_1^\times/\mathbb{C}_p^\times)_\mathbb{Q})
\]
with $\mathrm{T}_n\mapsto \mathrm{T}_n$ for all $n\geq 1$.
Moreover, if $B$ is non-split, then there is a $\Q$-linear map
\[
\mathbb{T}(\Gamma_0(M))\too \mathbb{T}(\mathrm{H}^0(\Gamma,\Z)_\mathbb{Q})
\]
with $\mathrm{T}_n\mapsto \mathrm{T}_n$ for all $n\geq 1$.
\end{corollary}
\begin{proof}
The classification results (Proposition \ref{prop:splittinf} and Proposition \ref{prop:classificationsplit}) imply that it suffices to prove that there are $\Q$-linear maps
\[
\mathbb{T}(\Gamma_0(M))\too \mathbb{T}(\mathrm{H}^1(\Gamma,\mathcal{A}_1^\times/\C_p^\times)_\Q^\mathrm{cusp}), \qquad \mathbb{T}(\Gamma_0(M))\too \mathbb{T}(\mathrm{H}^1(\Gamma,\mathcal{A}_1^\times/\C_p^\times)_\Q^\mathrm{Eis})
\]
with $\mathrm{T}_n\mapsto \mathrm{T}_n$ for all $n\geq 1$.
Moreover, the classification results yield that the cuspidal part is isomorphic to the cuspidal part of $N_\Q$, as Hecke modules. This isomorphism is compatible with $\mathrm{T}_n$ for $(n,p)=1$, and $\mathrm{T}_p$ corresponds to $-w_p$. Finally, the desired map follows from  Lemma \ref{lem:heckeiso} in the non-split case, and in the split case we can copy the proof of this lemma to deduce an isomorphism 
\[
\mathbb{T}(\mathrm{H}^1_\mathrm{par}(\Gamma_0(p),\Q))\simeq \mathbb{T}(\mathcal{S}_2(\Gamma_0(p),\Q))
\]
with $\mathrm{T}_n\mapsto \mathrm{T}_n$ for $(n,p)=1$, and $w_p\mapsto -\mathrm{T}_p$. The latter is based on the observation that $-w_p=\mathrm{T}_p$ on the space of cusp forms. The reason for this identity is that all cusp forms of level $\Gamma_0(p)$ are new because $\mathcal{S}_2(\mathrm{SL}_2(\mathbb{Z}))=0$ (see also \cite[Lemma 7]{AtkinLehner}). \\
The statement for the Eisenstein part in the split case follows from the Hecke stable splitting for the Eisenstein part in Proposition \ref{prop:classificationsplit}.
For the non-split case, observe that $\mathrm{H}^1(\Gamma,\mathcal{A}_1^\times/\C_p^\times)_\Q^{\mathrm{Eis}}$ is isomorphic, as a Hecke module, to the $\Q$-vector space generated by the weight $2$ Eisenstein series 
\[
E_2^{(M)}(q)=\prod_{l\mid M}(l-1)+24\sum_{n\geq 1}\sigma_1^{(M)}(n)\,q^n,\qquad \sigma_1^{(M)}(n)=\sum_{\substack{t\mid n\\(t,M)=1}} t,
\]
where the first product ranges over primes $l$ dividing $M$.
Indeed, for primes $l\nmid M$, we have $\mathrm{T}_l=l+1$, and for primes $l\mid M$, $\mathrm{T}_l=\mathrm{id}$ on $\mathrm{H}^1(\Gamma,\mathcal{A}_1^\times/\C_p^\times)_\Q^{\mathrm{Eis}}$.
A straightforward computation shows that this is precisely the action of $\mathrm{T}_l$ on $E_2^{(M)}$ for all $l$.\\
To conclude, note that any Hecke operator $\mathrm{T}_l$ acts on $\mathrm{H}^0(\Gamma,\Z)_\Q=\Q$ by multiplication by $|\Gamma\backslash \mathcal{O}_l|$. If $l\mid M$, this number is $1$, and if $l\nmid M$, this number is $l+1$. As stated above, these are precisely the eigenvalues of $\mathrm{T}_l$ acting on the Eisenstein series $E_2^{(M)}$, and the desired map follows as in the Eisenstein part case.
\end{proof}
Since the modularity is completely based on the $\Q$-linear map, we only focus on the case $\mathrm{H}^1(\Gamma,\mathcal{A}_{1}^\times/\mathbb{C}_p^\times)_\Q$.
By the corollary above, if $J_1\in \mathrm{H}^1(\Gamma,\mathcal{A}_{1}^\times/\mathbb{C}_p^\times)_\Q$ is a class, then we deduce a $\Q$-linear map 
\[
\alpha_{J_1}\colon\mathbb{T}(\Gamma_0(M)) \too \mathrm{H}^1(\Gamma,\mathcal{A}_{1}^\times/\mathbb{C}_p^\times)_\Q, \qquad \mathrm{T}_n\mapsto \mathrm{T}_n\cdot J_1.
\]
The desired modularity result is based on the following lemma:
\begin{lemma}\label{lem:heckefinish}
Let $\alpha\in \Hom_\mathbb{Q}(\mathbb{T}(\Gamma_0(M)),\mathbb{Q})$. Then there exists a unique rational number $a_0\in \mathbb{Q}$ such that the generating series
\[
a_0+\sum_{n\geq 1} \alpha(\mathrm{T}_n)\cdot q^n\in \mathbb{Q}\left[[q]\right]
\]
is a modular form of weight $2$ and level $\Gamma_0(M)$.
\end{lemma}
\begin{proof}
We can copy the proof of \cite[Proposition 3.24]{Steinmodularforms} to show that the bilinear pairing of $\mathbb{Q}$-vector spaces
\[
\mathcal{M}_2(\Gamma_0(M),\mathbb{Q})\times \mathbb{T}(\Gamma_0(M))\too \mathbb{Q},\qquad (g,\mathrm{T})\mapsto a_1\left(\mathrm{T}\cdot g\right),
\]
where $a_1\left(\mathrm{T}\cdot g\right)$ is the first Fourier coefficient of $\mathrm{T}(g)$, is perfect.
Therefore, the map 
\[
\Upsilon\colon \mathcal{M}_2(\Gamma_0(M),\mathbb{Q})\too \Hom_\mathbb{Q}(\mathbb{T}(\Gamma_0(M)),\mathbb{Q}),
\]
where $\Upsilon(f)$ maps a Hecke operator $\mathrm{T}$ to the first Fourier coefficient $a_1\left(\mathrm{T}\cdot f\right)$, is an isomorphism.
Finally, if we let $g\in \mathcal{M}_2(\Gamma_0(M),\mathbb{Q})$ be the preimage of $\alpha$ under $\Upsilon$, and choose $a_0=a_0(g)$ to be the zero-th Fourier coefficient of $g$, then the lemma follows from the identity 
\[
\alpha(\mathrm{T}_n)=a_1\left(\mathrm{T}_n\cdot g\right)=a_n(g)
\]
for $n\geq 1$.
\end{proof}
We can now prove the desired modularity:
\begin{lemma}
Let $J_1\in \mathrm{H}^1(\Gamma,\mathcal{A}_{1}^\times/\mathbb{C}_p^\times)_\Q$ be an analytic theta cocycle. Then there is a unique constant term $J_0\in\mathrm{H}^1(\Gamma,\mathcal{A}_{1}^\times/\mathbb{C}_p^\times)_\Q$ such that the generating series 
\begin{equation*}
J_0+\sum_{n\geq 1} (\mathrm{T}_n\cdot J_1) \cdot q^n \in \mathrm{H}^1(\Gamma,\mathcal{A}_{1}^\times/\mathbb{C}_p^\times)_\Q \left[[q]\right].
\end{equation*}
is modular of weight $2$ and level $\Gamma_0(M)$.
\end{lemma}
\begin{proof}
Let $\alpha\in \Hom_\mathbb{Q}(\mathrm{H}^1(\Gamma,\mathcal{A}_{1}^\times/\mathbb{C}_p^\times)_\Q,\mathbb{Q})$.
Then $\alpha\circ \alpha_{J_1}\colon \mathbb{T}(\Gamma_0(M)) \to \Q$ is $\Q$-linear, and, thus, by Lemma \ref{lem:heckefinish} we find a unique rational number $a_\alpha(J_1)\in \Q$ such that 
\begin{equation*}
a_\alpha(J_1)+\sum_{n\geq 1} \alpha(\mathrm{T}_n \cdot J_1)\cdot q^n
\end{equation*}
is modular.
This gives rise to a homomorphism
\[
\Hom_\mathbb{Q}(\mathrm{H}^1(\Gamma,\mathcal{A}_{1}^\times/\mathbb{C}_p^\times)_\Q,\Q)\too \Q, \qquad \alpha\mapsto a_\alpha(J_1).
\]
Finally, every such homomorphism is of the form $\alpha\mapsto \alpha(J_0)$ for a certain $J_0\in  \mathrm{H}^1(\Gamma,\mathcal{A}_{1}^\times/\mathbb{C}_p^\times)_\Q$, because $\mathrm{H}^1(\Gamma,\mathcal{A}_{1}^\times/\mathbb{C}_p^\times)_\Q$ is finite dimensional, and the lemma follows.
\end{proof}
This lemma holds also if we replace $\mathrm{H}^1(\Gamma,\mathcal{A}_{1}^\times/\mathbb{C}_p^\times)_\Q$ by $\mathrm{H}^0(\Gamma,\Z)_\Q$ in the non-split case. 
Thus, Theorem \ref{modulatitytheorem} follows.

\section{Intersection with an RM point}\label{section4}
We fix an RM point $\omega\in \mathcal{H}_{p}$ defined over a real quadratic field $K$ throughout the whole section.
\subsection{Intersection of a rational quadratic divisor with an RM point}\label{sec:intersectionsub}
Recall that a rational quadratic divisor on $\mathcal{H}_p^2$ is a locally finite sum of divisors
\[
\Delta_{v,p}=\lbrace (v\centerdot z,z) : z\in \mathcal{H}_{p}\rbrace
\]
for $v\in G$.
We can restrict the divisors $\Delta_{v,p}$ to the line $\mathcal{H}_{p}\times \lbrace \omega\rbrace\subset \mathcal{H}_p^2$. That is, $\Delta_{v,p}$ meets the line $\mathcal{H}_{p}\times \lbrace \omega\rbrace$ in the point $(v\centerdot \omega,\omega)$, and we identify this tuple with the RM point $v\centerdot\omega\in \mathcal{H}_{p}$. 
Thus, restricting to $\mathcal{H}_{p}\times \lbrace \omega\rbrace$ induces a map
\[
\inte\colon\bigoplus_{v\in G} \Z\cdot \Delta_{v,p}\too (\bigoplus_{\tau\in \mathcal{H}_p} \Z \cdot \tau) \otimes \Z, \qquad \Delta_{v,p}\mapsto v\centerdot \omega.
\]
This map is immediately verified to be a $G\times \Gamma_\omega$-module homomorphism. 
We aim to extend this homomorphism to locally finite divisors; let $\Delta=\sum_v n_v\cdot \Delta_{v,p}\in \mathrm{Div}^\dagger_{\mathrm{rq}}(\mathcal{H}_{p}^2)$. Recall that this means that for every affinoid subset $Y\subset \mathcal{H}_p^2$ the set $\{ v : \Delta_{v,p}\cap Y\neq \emptyset, \  n_v\neq 0\}$ is finite.
If we restrict $\Delta$ to $\mathcal{H}_{p}\times \lbrace \omega\rbrace$, then we obtain the formal sum
\begin{equation}\label{eq:restrictin}
\inte(\Delta)\coloneq \sum_{\tau\in  G\centerdot \omega} \left(\sum_{v\in \Sigma_\tau} n_v\right) \cdot  \tau,
\end{equation}
where
\[
\Sigma_\tau\coloneq \lbrace v\in G : v\centerdot \omega=\tau \ \mathrm{and}\ n_v\neq 0\rbrace.
\]
We now show that $\inte(\Delta)$ is a well-defined, locally finite divisor on $\mathcal{H}_p$:
\begin{lemma}\label{lem:finitesums}
The sets $\Sigma_\tau$ for $\tau\in G\centerdot \omega$ are finite, and for any affinoid subset $X\subset \mathcal{H}_p$, the set \[
\bigg\{\tau \in G\centerdot \omega : \tau \in X \text{ and } \sum_{v\in \Sigma_\tau} n_v\neq 0 \bigg\}
\]
is finite. 
\end{lemma}
\begin{proof}
We fix an affinoid neighbourhood $U\subset \mathcal{H}_{p}$ of $\omega$ for the whole proof. 
Let $\tau\in G\centerdot \omega$ be a point and $V\subset \mathcal{H}_{p}$ be an affinoid neighbourhood of $\tau$. The product $Y\coloneq V\times U$ is an affinoid subset of $\mathcal{H}_p^2$, that contains $(\tau,\omega)$. 
One readily verifies the inclusion 
\[
\Sigma_\tau\subset \lbrace v\in G : \Delta_{v,p}\cap Y \neq \emptyset \ \mathrm{ and } \ n_v\neq 0\rbrace.
\]
The right-most set is finite, since $\Delta$ is locally finite, and the finiteness of $\Sigma_\tau$ follows.
The locally finiteness of the formal sum $\inte(\Delta)$ follows similarly.
\end{proof}
Recall that the stabilizer $\Gamma_\omega$ of $\omega$ in $\Gamma$ is up to $2$-torsion infinite cyclic, and that we determined a distinguished generator $\gamma_\omega\in \Gamma_\omega$.
We denote by $\Gamma_\omega^+$ the infinite cyclic group generated by $\gamma_\omega$. 
The intersection map $\inte\colon \mathrm{Div}^\dagger_{\mathrm{rq}}(\mathcal{H}_p^2)\too \mathrm{Div}^\dagger(\mathcal{H}_p)\otimes \Z$ is a $\Gamma\times \Gamma_\omega$-module homomorphism, and, thus, induces a homomorphism in cohomology 
\begin{equation}\label{equationabove}
\inte\colon \mathrm{H}^2\left(\Gamma^2,\mathrm{Div}^\dagger_{\mathrm{rq}}(\mathcal{H}_{p}^2)\right)\xrightarrow{\res} \mathrm{H}^2\left(\Gamma\times \Gamma_\omega^+,\mathrm{Div}^\dagger_{\mathrm{rq}}(\mathcal{H}_{p}^2)\right) \too \mathrm{H}^2\left(\Gamma\times \Gamma_\omega^+, \mathrm{Div}^\dagger(\mathcal{H}_p)\otimes \mathbb{Z}\right).
\end{equation}
\begin{example}\label{formulafordv-1}
We make the above definition explicit for the class $[\mathcal{D}_1]$:
For an RM point $\tau\in \Gamma\centerdot \omega$, fix an element $\beta_\tau\in \Gamma$ with $\tau=\beta_\tau\centerdot \omega$.
Then, $\gamma \in \Gamma$ satisfies $\gamma \centerdot\omega=\tau$ if and only if $\gamma\in \beta_\tau \cdot \Gamma_\omega$.
Therefore, for any $2$-chain $c\in \mathscr{C}_2(\mathcal{H}_\infty^2)$ we have
\begin{align*}
\inte(\mathcal{D}_1)(c)&=\inte\left(\sum_{\gamma\in \Gamma} (c\cap \Delta_{\gamma,\infty})\cdot \Delta_{\gamma,p}\right)=\sum_{\tau\in \Gamma\centerdot \omega} \left(\sum_{\gamma\in \beta_\tau\cdot \Gamma_\omega} c\cap \Delta_{\gamma,\infty} \right)\cdot \tau \\
&=2 \sum_{\tau\in \Gamma \centerdot \omega} \  \left(\sum_{n\in \mathbb{Z}} c\cap \Delta_{\beta_\tau\cdot \gamma_\omega^n,\infty} \right) \ \cdot \tau.
\end{align*}
The last equality holds because $\Gamma_\omega=\Gamma_\omega^+\oplus \lbrace \pm 1\rbrace$ and $c\cap \Delta_{v,\infty}=c\cap \Delta_{-v,\infty}$ for all $v\in G$.
\end{example}
The infinite cyclic group $\Gamma_\omega^+$ is of type $\mathrm{FP}_\infty$; therefore, we may apply the K\"unneth Formula (Proposition \ref{prop:kunneth}), which yields an isomorphism
\[
\begin{array}{c}
\mathrm{H}^2\big(\Gamma,\mathrm{Div}^\dagger(\mathcal{H}_p)\big)\otimes\mathrm{H}^0(\Gamma_\omega^+,\Z)\\[2pt]
\oplus\\[2pt]
\mathrm{H}^1\big(\Gamma,\mathrm{Div}^\dagger(\mathcal{H}_p)\big)\otimes\mathrm{H}^1(\Gamma_\omega^+,\Z)
\end{array}
\ \too\ \mathrm{H}^2\big(\Gamma\times\Gamma_\omega^+,\mathrm{Div}^\dagger(\mathcal{H}_p)\otimes\Z\big).
\]
Indeed, the integer cohomology of $\Gamma_\omega^+$ is free, which implies that the $\mathrm{Tor}$-terms vanish, and $\mathrm{H}^2(\Gamma_\omega^+,\Z)=0$.
The choice of a generator for $\Gamma_\omega^+$ gives rise to an identification 
\[
\mathrm{H}^1(\Gamma_\omega^+,\mathbb{Z})=\mathbb{Z},
\]
and we denote the fundamental class corresponding to $1$ by $\eta_\omega\in \mathrm{H}^1(\Gamma_\omega^+,\mathbb{Z})$. 
We denote by 
\[
\mathrm{pr}\colon \mathrm{H}^2\big(\Gamma\times\Gamma_\omega^+,\mathrm{Div}^\dagger(\mathcal{H}_p)\otimes\Z\big)\too \mathrm{H}^1\big(\Gamma,\mathrm{Div}^\dagger(\mathcal{H}_p)\big)\otimes\mathrm{H}^1(\Gamma_\omega^+,\Z)
\]  
the projection onto the second summand.
We may now formulate the main theorem of this section:
\begin{theorem}\label{mainthm} 
For any Hecke operator $\mathrm{T}\in \mathbb{T}$ we have
\[
\mathrm{pr}\left(\inte\left((\mathrm{T}\times 1)\cdot [\mathcal{D}_{1}]\right)\right)= \left(2\cdot  \mathrm{T}\cdot [\mathcal{D}_\omega]\right)\times \eta_\omega,
\]
where $[\mathcal{D}_\omega]\in \mathrm{H}^1(\Gamma,\mathrm{Div}^\dagger(\mathcal{H}_p))$ is the divisor-valued cohomology class associated to $\omega$.
\end{theorem}
One immediate application of this theorem is that the class $[\mathcal{D}_1]$ is non-trivial. This is not obvious from its definition.
\begin{corollary}
The class $[\mathcal{D}_{1}]$ is non-trivial.
\end{corollary}
\begin{proof}
The corollary follows from the theorem by taking $\mathrm{T}=1$, because the class $2\cdot [\mathcal{D}_\omega]$ is non-trivial by Proposition \ref{prop:class}.
\end{proof}
\begin{remark}
If $B$ is split, then $\mathrm{H}^2(\Gamma,\mathrm{Div}^\dagger(\Gamma\centerdot \omega))=0$ by \cite[Proposition A.3.1]{SpreheThesis}. This vanishing implies that we have the stronger identity
\[
\inte\left((\mathrm{T}\times 1)\cdot [\mathcal{D}_{1}]\right)= \left(2\cdot  \mathrm{T}\cdot [\mathcal{D}_\omega]\right)\times \eta_\omega
\]
without projecting onto the second summand.
Otherwise, if $B$ is non-split, then we have $\mathrm{H}^2(\Gamma,\mathrm{Div}^\dagger(\Gamma\centerdot\omega))\otimes\Q\simeq\Q$
by the results of \cite{GehrmannquaternionicRC}. Thus, the component of $\inte([\mathcal{D}_1])$ in the first summand is in general not trivial. 
\end{remark}
\subsection{Overview of the proof}
Since the proof of Theorem \ref{mainthm} is quite involved, we sketch the main ideas in this subsection.
First of all, we notice that it suffices to consider the case $\mathrm{T}=1$. Indeed, the intersection map is compatible with the $\mathbb{T}\times 1$-action, and we find that 
\[
\inte\left((\mathrm{T}\times 1)\cdot [\mathcal{D}_{1}]\right)=(\mathrm{T}\times 1)\cdot \inte\left([\mathcal{D}_{1}]\right).
\]
Finally, the cohomology cross-product is compatible with Hecke operators, and the claim follows. \\
The central idea is to verify the identity in Theorem \ref{mainthm} for the corresponding cocycles. The cohomology cross-product is induced by the cochain cross-product (see the discussion below Proposition \ref{prop:kunneth}).
That is, if $Q_\bullet$ is a (projective) resolution of $\Z$ over $\Z[\Gamma_\omega^+]$, and $\eta_\omega\in \Hom_{\Gamma_\omega^+}(Q_1,\Z)$ is a $1$-cocycle representing the fundamental class $\eta_\omega$, then a cocycle representative for $[\mathcal{D}_\omega]\times \eta_\omega$ is the cochain cross-product 
\[
\mathcal{D}_\omega\times \eta_\omega\colon \mathscr{C}_1(\mathcal{H}_\infty)\otimes Q_1\too \mathrm{Div}^\dagger(\mathcal{H}_p)\otimes \Z.
\]
The class $\inte([\mathcal{D}_1])$ is induced by the $2$-cocycle
\[
\inte(\mathcal{D}_1)\colon \mathscr{C}_2(\mathcal{H}_\infty^2)\too \mathrm{Div}^\dagger(\mathcal{H}_p)\otimes \Z,
\]
which was computed in Example \ref{formulafordv-1}.
Note that the cocycles $\mathcal{D}_\omega\times \eta_\omega$ and $\inte(\mathcal{D}_1)$ are not comparable because the resolutions differ.
Therefore, we invoke the Eilenberg--Zilber map. This is a $\Gamma\times \Gamma$-equivariant quasi-isomorphism 
\[
\mathrm{EZ}_\bullet\colon C_\bullet(\mathcal{H}_\infty) \otimes C_\bullet(\mathcal{H}_\infty)\too C_\bullet(\mathcal{H}_\infty^2).
\]
More precisely, in Section \ref{sec:subcom} we introduce explicit complexes $\mathscr{C}^x_\bullet(\mathcal{H}_\infty)\subset \mathscr{C}_\bullet(\mathcal{H}_\infty)$ and $\mathscr{C}^y_\bullet(\Delta_{\omega,\infty})\subset C_\bullet(\Delta_{\omega,\infty})$,  where $\Delta_{\omega,\infty}\subset \mathcal{H}_\infty$ is the geodesic connecting the boundary points $\omega^\prime$ and $\omega$. These complexes depend on specific basepoints $x$ and $y$, and are constructed such that the Eilenberg--Zilber map induces a quasi-isomorphism 
\[
\mathrm{EZ}_\bullet\colon \mathscr{C}_\bullet^x(\mathcal{H}_\infty)\otimes \mathscr{C}_\bullet^y(\Delta_{\omega,\infty})\too \mathscr{C}_\bullet(\mathcal{H}_\infty^2).
\] 
The class $\inte([\mathcal{D}_1])$ is then induced by the $2$-cocycle
\[
\inte(\mathcal{D}_1)\colon \left(\mathscr{C}_\bullet^x(\mathcal{H}_\infty)\otimes \mathscr{C}_\bullet^y(\Delta_{\omega,\infty})\right)_2\xrightarrow{\mathrm{EZ}_2} \mathscr{C}_2(\mathcal{H}_\infty^2)\too \mathrm{Div}^\dagger(\mathcal{H}_p)\otimes \Z.
\]
We show at the end of Section \ref{sec:eilenberg} that the projection of $\inte([\mathcal{D}_1])$ to the second summand is induced by the restriction of the cocycle to the summand $\mathscr{C}_1^x(\mathcal{H}_\infty)\otimes\mathscr{C}_1^y(\Delta_{\omega,\infty})$.
Finally, the desired identity follows from the commutativity of the following diagram:
\begin{center}
\begin{minipage}{\linewidth}
\centering
\begin{tikzcd}
\mathscr{C}_1^x(\mathcal{H}_\infty)\otimes \mathscr{C}_1^y(\Delta_{\omega,\infty}) \arrow[rr, " 2\cdot \mathcal{D}_\omega \times \eta_\omega"] \arrow[d, "\mathrm{EZ}_2" ] & & \mathrm{Div}^\dagger(\mathcal{H}_p)\otimes \Z \\
\mathscr{C}_2(\mathcal{H}_\infty^2) \arrow[rr, "\mathcal{D}_1"]&  & \mathrm{Div}^\dagger_{\mathrm{rq}}(\mathcal{H}_p^2). \arrow[u, "\inte"]
 \end{tikzcd}
\end{minipage}
\end{center} 
This is deduced from an explicit computation of intersection numbers (see Section \ref{sec:comparison}).
\subsection{Subcomplexes of singular chain complexes}\label{sec:subcom}
This section is devoted to introducing the complexes $\mathscr{C}^x_\bullet(\mathcal{H}_\infty)$ and $\mathscr{C}^y_\bullet(\Delta_{\omega,\infty})$.
Recall that if $\tau\in \mathcal{H}_p$ is an RM point, then we denote by $\Delta_{\tau,\infty}\subset \mathcal{H}_\infty$ the oriented geodesic connecting the boundary points $\tau^\prime$ and $\tau$. We fix a point
\[
x\in \mathcal{H}_\infty\setminus \bigcup_{\tau\in \mathcal{H}_p^{\mathrm{RM}}}\Delta_{\tau,\infty},
\]
which exists by the proof of \cite[Proposition 2.17]{DGL}. Throughout the whole section, we denote by $\delta_{z_0,z_1}\in C_1(\mathcal{H}_\infty)$ the geodesic segment on $\mathcal{H}_\infty$ connecting two points $z_0,z_1\in \mathcal{H}_\infty$. 
To simplify computations of intersection numbers, we aim to replace $1$-chains in $\mathscr{C}_1(\mathcal{H}_\infty)$ by finite sums of geodesic segments $\delta_{z_0,z_1}$. 
To ensure that these segments are $1$-chains in $\mathscr{C}_1(\mathcal{H}_\infty)$, we require that their endpoints are $\Gamma$-translates of $x$.
Accordingly, we define the complex $\mathscr{C}_\bullet^x(\mathcal{H}_\infty)$ as follows:
\begin{equation*}
\mathscr{C}_n^x(\mathcal{H}_\infty)\coloneq 
   \begin{cases} 
     C_0(\Gamma \centerdot x) , & \mathrm{if} \ n=0, \\ 
     \bigoplus_{z_0,z_1\in \Gamma\centerdot x} \ \mathbb{Z}\cdot \delta_{z_0,z_1} , & \mathrm{if} \  n=1,\\
     \lbrace \sigma\in C_2(\mathcal{H}_\infty): \partial_2(\sigma)\in \mathscr{C}_1^x(\mathcal{H}_\infty)\rbrace, & \mathrm{if} \ n=2 \ \mathrm{and}\\
      C_n(\mathcal{H}_\infty) , & \mathrm{if} \ n\geq 3.
     \end{cases}
\end{equation*}
By definition, the complex $\mathscr{C}^x_\bullet(\mathcal{H}_\infty)$ is a complex of $\mathbb{Z}[\Gamma]$-modules. We observe:
\begin{lemma}\label{reasoning}
The inclusion
\[
\mathscr{C}_\bullet^x(\mathcal{H}_\infty)\subset C_\bullet(\mathcal{H}_\infty)
\]
is a $\mathbb{Z}[\Gamma]$-equivariant quasi-isomorphism, which factors through $\mathscr{C}_\bullet(\mathcal{H}_\infty)$. Thus, the complex $\mathscr{C}_\bullet^x(\mathcal{H}_\infty)$ together with the augmentation map $\varepsilon\colon \mathscr{C}_0^x(\mathcal{H}_\infty)\to \mathbb{Z}$ is a resolution of $\mathbb{Z}$ over $\mathbb{Z}[\Gamma]$.
\end{lemma}
\begin{proof}
The assumption on $x$ implies immediately that the complex $\mathscr{C}_\bullet^x(\mathcal{H}_\infty)$ is a subcomplex of the complex $\mathscr{C}_\bullet(\mathcal{H}_\infty)$, and it remains to prove that the inclusion induces isomorphisms in homology. In degrees $n\geq 2$ this is obvious. Moreover, by the definition of $\mathscr{C}_2^x(\mathcal{H}_\infty)$, the group $\mathrm{H}_1(\mathscr{C}_\bullet^x(\mathcal{H}_\infty))$ embeds into the trivial group $\mathrm{H}_1(C_\bullet(\mathcal{H}_\infty))$ and, therefore, must be trivial as well. Finally, since any two points in $\Gamma\centerdot x$ can be connected by a geodesic segment, that is, a $1$-chain in $\mathscr{C}_1^x(\mathcal{H}_\infty)$, the group $\mathrm{H}_0(\mathscr{C}_\bullet^x(\mathcal{H}_\infty))$ is cyclic and generated by the class induced by $x$. Analogously, $\mathrm{H}_0(C_\bullet(\mathcal{H}_\infty))$ is generated by the class induced by $x$ and, hence, the map induced by the inclusion
\[
\mathrm{H}_0(\mathscr{C}_\bullet^x(\mathcal{H}_\infty))\too \mathrm{H}_0(C_\bullet(\mathcal{H}_\infty)), \qquad k\cdot [x]\mapsto k\cdot [x] 
\]
is an isomorphism. The lemma follows.
\end{proof}
We now define the complex $\mathscr{C}_\bullet^y(\Delta_{\omega,\infty})$.
As in the previous case, we replace $1$-chains by geodesic segments. Additionally, we require the following condition:
\begin{enumerate}
\item[(E)]Let $\delta\in \mathscr{C}_1^x(\mathcal{H}_\infty)$ and $\delta_\omega\in \mathscr{C}_1^y(\Delta_{\omega,\infty})$ be geodesic segments. Then the endpoints of $\delta$ do not lie on $v\centerdot \delta_\omega$, and the endpoints of
$v\centerdot\delta_\omega$ do not lie on $\delta$ for all $v\in G$.
\end{enumerate}
We need to invoke the condition for all $v\in G$ because we aim to construct $\mathscr{C}^y_\bullet(\Delta_{\omega,\infty})$, such that the total complex $\mathscr{C}^x_\bullet(\mathcal{H}_\infty)\otimes \mathscr{C}^y_\bullet(\Delta_{\omega,\infty})$ can be viewed as a subcomplex of $\mathscr{C}_\bullet(\mathcal{H}_\infty^2)$.
For the definition of the desired complex, choose 
\[
y\in \Delta_{\omega,\infty}\setminus \bigcup_{z_0,z_1\in G\centerdot x} \delta_{z_0,z_1}.
\]
Such a point always exists because the condition above excludes a countable number of points of $\Delta_{\omega,\infty}$, whereas $\Delta_{\omega,\infty}$ is uncountable.
The complex $\mathscr{C}_\bullet^y(\Delta_{\omega,\infty})$ is then defined by
\begin{equation*}
\mathscr{C}_n^y(\Delta_{\omega,\infty})\coloneq 
   \begin{cases} 
     C_0(\Gamma^+_\omega \centerdot   y) , & \mathrm{if} \ n=0, \\ 
     \bigoplus_{z_0,z_1\in \Gamma^+_\omega \centerdot y} \mathbb{Z}\cdot  \delta_{z_0,z_1} , & \mathrm{if} \  n=1,\\
     \lbrace \sigma\in C_2(\Delta_{\omega,\infty}): \partial_2(\sigma)\in \mathscr{C}_1^y(\Delta_{\omega,\infty})\rbrace, & \mathrm{if} \ n=2 \ \mathrm{and}\\
      C_n(\Delta_{\omega,\infty}) , & \mathrm{if} \ n\geq 3.
     \end{cases}
\end{equation*}
The complex $\mathscr{C}_\bullet^y(\Delta_{\omega,\infty})$ is a complex of $\mathbb{Z}[\Gamma_\omega^+]$-modules, and the same reasoning as in the proof of Lemma \ref{reasoning} implies that the inclusion 
\[
\mathscr{C}_\bullet^y(\Delta_{\omega,\infty})\subset C_\bullet(\Delta_{\omega,\infty})
\]
is a $\mathbb{Z}[\Gamma_\omega^+]$-equivariant quasi-isomorphism. Furthermore, condition (E) is satisfied: let $\delta\in \mathscr{C}^x_1(\mathcal{H}_\infty)$, $\delta_\omega\in \mathscr{C}_1^y(\Delta_{\omega,\infty})$ and $v\in G$.
The geodesic segment $v\centerdot \delta_\omega$ is part of the geodesic $\Delta_{v\centerdot \omega,\infty}$ and, thus, the endpoints of $\delta$ cannot lie on $v\centerdot \delta_\omega$ by the choice of $x$.
Moreover, the endpoints of $v\centerdot \delta_\omega$ are contained in $G\centerdot y$, therefore, they cannot belong to $\delta$.
We summarize:
\begin{lemma}\label{lem:freeresolutuin}
The complex $\mathscr{C}_\bullet^y(\Delta_{\omega,\infty})$ together with the augmentation map $\varepsilon\colon \mathscr{C}_0^y(\Delta_{\omega,\infty})\to \mathbb{Z}$ is a resolution of $\mathbb{Z}$ over $\mathbb{Z}[\Gamma_\omega^+]$.
\end{lemma} 
To simplify notation, we denote by $\delta_{n,m}\in \mathscr{C}_1^y(\Delta_{\omega,\infty})$ the geodesic segment connecting the points $\gamma_\omega^n\centerdot y $ and $\gamma_\omega^m\centerdot y$ for integers $n,m\in \mathbb{Z}$. In terms of the resolution $\mathscr{C}_\bullet^y(\Delta_{\omega,\infty})$, the fundamental class $\eta_\omega$ is induced by the cocycle
\[
\eta_\omega\colon \mathscr{C}_1^y(\Delta_{\omega,\infty})\too \mathbb{Z}, \qquad \delta_{n_0,n_1} \mapsto n_1-n_0.
 \]
\subsection{The Eilenberg--Zilber map}\label{sec:eilenberg}
In this section, we study the Eilenberg--Zilber map
\[
\mathrm{EZ}_\bullet\colon C_\bullet(\mathcal{H}_\infty) \otimes C_\bullet(\Delta_{\omega,\infty})\too C_\bullet(\mathcal{H}_\infty^2).
\]
See \cite[Section 5]{Eilenbergmaclaneongroups} for its precise construction.
Objectwise, it has the form
\[
\mathrm{EZ}_{i+j}\colon C_i(\mathcal{H}_\infty)\otimes C_j(\Delta_{\omega,\infty})\too C_{i+j}(\mathcal{H}_\infty^2).
\]
It is a $\Gamma\times \Gamma_\omega^+$-equivariant quasi-isomorphism, and in degree $0$ and $1$ it has the simple form 
\[
\mathrm{EZ}_0(\alpha\otimes \beta)=(\alpha,\beta) \qquad \mathrm{and} \qquad \mathrm{EZ}_1(\delta\otimes \beta + \alpha\otimes \delta_\omega)=(\delta,\beta)+(\alpha,\delta_\omega)
\]
for $\alpha\in \mathcal{H}_\infty$, $\beta\in \Delta_{\omega,\infty}$ and $\delta\in C_1(\mathcal{H}_\infty)$, $\delta_\omega\in C_1(\Delta_{\omega,\infty})$. Moreover, if $c_1\in C_2(\mathcal{H}_\infty)$ and $c_2\in C_2(\Delta_{\omega,\infty})$ are $2$-simplices, then 
\[
\mathrm{EZ}_2(\alpha\otimes c_2+c_1\otimes \beta)=(\alpha,c_2)+(c_1,\beta).
\]
To give a formula for $\mathrm{EZ}_2$ on $C_1(\mathcal{H}_\infty)\otimes C_1(\Delta_{\omega,\infty})$, we define the degeneracy maps 
\begin{align*}
&s_0\colon \Delta_2\too \Delta_1, \qquad (t_0,t_1,t_2)\mapsto (t_0+t_1,t_2) \quad \text{and}\\
&s_1\colon \Delta_2\too \Delta_1, \qquad (t_0,t_1,t_2)\mapsto (t_0,t_1+t_2).
\end{align*}
Then:
\[
\mathrm{EZ}_2(\sigma_1\otimes \sigma_2)=(\sigma_1\circ s_1,\sigma_2\circ s_0)-(\sigma_1\circ s_0,\sigma_2\circ s_1)
\]
for $\sigma_1\in C_1(\mathcal{H}_\infty)$ and $\sigma_2\in C_1(\Delta_{\omega,\infty})$. For precise formulas in higher degrees see \cite[Section 5]{Eilenbergmaclaneongroups}. We can prove:
\begin{lemma}\label{lem:eilenberg}
The composition 
\[
\mathscr{C}_\bullet^x(\mathcal{H}_\infty)\otimes \mathscr{C}_\bullet^y(\Delta_{\omega,\infty})\hookrightarrow C_\bullet(\mathcal{H}_\infty)\otimes C_\bullet(\Delta_{\omega,\infty}) \xrightarrow{\mathrm{EZ}_{\bullet}} C_\bullet(\mathcal{H}_\infty^2)
\]
factors through $\mathscr{C}_\bullet(\mathcal{H}_\infty^2)$. In particular, the Eilenberg--Zilber map induces a $\Gamma\times \Gamma^+_\omega$-equivariant quasi-isomorphism 
\[
\mathrm{EZ}_\bullet\colon \mathscr{C}_\bullet^x(\mathcal{H}_\infty)\otimes \mathscr{C}_\bullet^y(\Delta_{\omega,\infty})\too \mathscr{C}_\bullet(\mathcal{H}_\infty^2).
\]
\end{lemma}
\begin{proof}
Since $\mathscr{C}_n(\mathcal{H}_\infty^2)=C_n(\mathcal{H}_\infty^2)$ for all $n\geq 3$, it suffices to consider the degrees $0$, $1$ and $2$.
In degree $0$, this is obvious by the choice of $x$ and $y$. Indeed, if $\mathrm{EZ}_0(\alpha\otimes \beta)=(\alpha,\beta)\in \Delta_{v,\infty}$ for some $v\in G$, then $\alpha=v\centerdot \beta\in \Delta_{v\centerdot \omega,\infty}$ in contradiction to the assumption on points in $\Gamma\centerdot x$.
The statement in degree $1$ relies on the imposed condition (E), which we discussed before Lemma \ref{lem:freeresolutuin}. 
Finally in degree $2$, the statement follows from the compatibility of the Eilenberg--Zilber map with the differentials, the degree $1$ case and the definition of $\mathscr{C}_2(\mathcal{H}_\infty^2)$.
Compare the proof of \cite[Proposition 2.5.1]{SpreheThesis} for a detailed discussion of these computations.
\end{proof} 
The lemma above implies that the class $\inte([\mathcal{D}_1])$ is induced by the $2$-cocycle
\[
\inte(\mathcal{D}_1)\colon \left(\mathscr{C}_\bullet^x(\mathcal{H}_\infty)\otimes \mathscr{C}_\bullet^y(\Delta_{\omega,\infty})\right)_2\xrightarrow{\mathrm{EZ}_2}\mathscr{C}_2(\mathcal{H}_\infty^2)\too \mathrm{Div}^\dagger(\mathcal{H}_p)\otimes \Z.
\]
Recall that $\left(\mathscr{C}_\bullet^x(\mathcal{H}_\infty)\otimes \mathscr{C}_\bullet^y(\Delta_{\omega,\infty})\right)_2$ equals
\[
\mathscr{C}_0^x(\mathcal{H}_\infty)\otimes \mathscr{C}_2^y(\Delta_{\omega,\infty}) \  \oplus \  \mathscr{C}_1^x(\mathcal{H}_\infty)\otimes \mathscr{C}_1^y(\Delta_{\omega,\infty}) \  \oplus \  \mathscr{C}_2^x(\mathcal{H}_\infty)\otimes \mathscr{C}_0^y(\Delta_{\omega,\infty}).
\]
We first observe that
\begin{equation}\label{eq:firstobserve}
\inte(\mathcal{D}_1)\left(\mathscr{C}_0^x(\mathcal{H}_\infty)\otimes \mathscr{C}_2^y(\Delta_{\omega,\infty})\right)=0.
\end{equation}
Indeed, let $z\otimes c\in \mathscr{C}_0^x(\mathcal{H}_\infty)\otimes \mathscr{C}_2^y(\Delta_{\omega,\infty})$. 
Then $\mathrm{EZ}_2(z\otimes c)$ is contained in $\{z\}\times\Delta_{\omega,\infty}$, 
and this set intersects $\Delta_{v,\infty}$ if and only if $z\in v\centerdot\Delta_{\omega,\infty}=\Delta_{v\centerdot \omega,\infty}$ for $v\in \Gamma$. 
By the choice of the basepoint $x$, no point of $\Gamma\centerdot x$ lies on such a geodesic. Hence,
\[
\mathrm{EZ}_2(z\otimes c)\in C_2\left(\mathcal{H}_\infty^2\setminus \Delta_{v,\infty}\right)\quad \text{for all } v\in \Gamma,
\]
and, therefore,
\[
\mathcal{D}_1(z\otimes c)=\sum_{v\in \Gamma} \left(\mathrm{EZ}_2(z\otimes c)\cap \Delta_{v,\infty}\right)\cdot \Delta_{v,p} =0.
\]
The corresponding identity for the restriction of $\inte(\mathcal{D}_1)$ to 
$\mathscr{C}_2^x(\mathcal{H}_\infty)\otimes \mathscr{C}_0^y(\Delta_{\omega,\infty})$ does not hold. 
Indeed, we can always find points 
$x_0,x_1,x_2\in \Gamma\centerdot x$ such that the geodesic triangle $c$ with vertices 
$x_0,x_1,x_2$ contains $y$. Then $c\in \mathscr{C}_2^x(\mathcal{H}_\infty)$, 
but $\mathrm{EZ}_2(c\otimes y)$ intersects $\Delta_{1,\infty}$. But, the restriction of $\inte(\mathcal{D}_1)$ to this summand is the cross-product of a $2$-cocycle and a $0$-cocycle.  To see this, let $c\in \mathscr{C}_2^x(\mathcal{H}_\infty)$ and let 
$z_m=\gamma_\omega^m \centerdot y\in \mathscr{C}_0^y(\Delta_{\omega,\infty})$. 
Since $\inte(\mathcal{D}_1)$ is $\Gamma\times\Gamma_\omega^+$-equivariant and 
$\Gamma_\omega^+$ acts trivially on $\Z$, we obtain
\[
\inte(\mathcal{D}_1)(c\otimes z_m)=(1,\gamma_\omega^m) \cdot \inte(\mathcal{D}_1)(c\otimes z_0)=\inte(\mathcal{D}_1)(c\otimes z_0).
\]
In particular, this expression is independent of $m\in \Z$. 
We define
\[
g\colon \mathscr{C}_2^x(\mathcal{H}_\infty)\too \mathrm{Div}^\dagger(\mathcal{H}_p), \qquad c\mapsto \inte(\mathcal{D}_1)(c\otimes z_0),
\]
and
\[
g_\omega\colon \mathscr{C}_0^y(\Delta_{\omega,\infty})\too \Z, \qquad \sum_{i=1}^r k_i \cdot z_{m_i} \mapsto \sum_{i=1}^r k_i.
\]
Then $g$ is a $2$-cocycle, $g_\omega$ is a $0$-cocycle, and 
$\inte(\mathcal{D}_1)=g\times g_\omega$ on $\mathscr{C}_2^x(\mathcal{H}_\infty)\otimes \mathscr{C}_0^y(\Delta_{\omega,\infty})$.
Thus, the class $\mathrm{pr}(\inte([\mathcal{D}_1]))$ is induced by the $2$-cocycle
\[
\inte(\mathcal{D}_1)\colon \mathscr{C}_1^x(\mathcal{H}_\infty)\otimes \mathscr{C}_1^y(\Delta_{\omega,\infty}) \too \mathrm{Div}^\dagger(\mathcal{H}_p)\otimes \Z.
\]
The proof of Theorem \ref{mainthm} relies on computing $\inte(\mathcal{D}_1)$ on this summand. In the end, the desired equality of cocycles follows from an identity of signed intersection numbers, which we prove in the next section.

\subsection{Comparison of intersection numbers}\label{sec:comparison}
Let $\delta\in \mathscr{C}_1^x(\mathcal{H}_\infty)$ be a geodesic segment. Example \ref{formulafordv-1} implies that 
\[
\inte(\mathcal{D}_1)(\delta\otimes \delta_{0,1})=2\cdot \sum_{\tau\in \Gamma\centerdot \omega} \left(\sum_{n\in \Z} \mathrm{EZ}_2\left((\beta_\tau^{-1}\centerdot \delta)\otimes \delta_{n,n+1}\right)\cap \Delta_{1,\infty}\right)\cdot \tau,
\]
and all a priori infinite sums are finite sums by Lemma \ref{lem:finitesums}. 
Note that to obtain this equality, we used the identity of signed intersection numbers
\begin{align*}
\mathrm{EZ}_2(\delta\otimes \delta_{0,1})\cap \Delta_{\beta_\tau\gamma_\omega^n,\infty}&=(\beta_\tau^{-1},\gamma_\omega^n)\centerdot \mathrm{EZ}_2(\delta\otimes \delta_{0,1})\cap ((\beta_\tau^{-1},\gamma_\omega^n)\cdot \Delta_{\beta_\tau\gamma_\omega^n,\infty})\\
&=\mathrm{EZ}_2(\beta_\tau^{-1}\centerdot \delta\otimes \delta_{n,n+1})\cap \Delta_{1,\infty}.
\end{align*}
Motivated by the above, we define
\[
I_\omega(\delta)\coloneq \sum_{n\in \Z} \mathrm{EZ}_2(\delta\otimes \delta_{n,n+1})\cap \Delta_{1,\infty}\in \Z
\]
for $\delta\in \mathscr{C}_1^x(\mathcal{H}_\infty)$. That is, we have:
\begin{equation}\label{eq:identtityforprop}
\inte(\mathcal{D}_1)(\delta\otimes \delta_{0,1})=2\cdot \sum_{\tau\in \Gamma\centerdot \omega} I_\omega(\beta_\tau^{-1}\centerdot \delta)\cdot \tau.
\end{equation}
This subsection is devoted to establishing the following comparison of signed intersection numbers:
\begin{proposition}\label{prop:comparisomprop}
We have
\[
I_\omega(\delta)=-\delta\cap \Delta_{\omega,\infty},
\]
where the right-most term is the signed intersection number of $\delta$ and $\Delta_{\omega,\infty}$ (see the definition in Section \ref{sec:3-dim}).
\end{proposition}
This formula yields the following corollary, which in the end implies Theorem \ref{mainthm}:
\begin{corollary}\label{cor:comparison}
We have
\[
\inte(\mathcal{D}_1)(\delta\otimes \delta_{0,1})= 2\cdot (\mathcal{D}_\omega \times \eta_\omega) (\delta\otimes \delta_{0,1})
\]
for all $\delta\in \mathscr{C}_1^x(\mathcal{H}_\infty)$.
\end{corollary}
\begin{proof}
The definition of the cochain cross-product (see the discussion after Proposition \ref{prop:kunneth}) implies:
\begin{align*}
(\mathcal{D}_\omega \times \eta_\omega) (\delta\otimes \delta_{0,1})&=(-1)^1 \cdot \mathcal{D}_\omega(\delta)\cdot \eta_\omega(\delta_{0,1})=- \sum_{\tau\in \Gamma\centerdot \omega} (\delta\cap \Delta_{\tau,\infty})\cdot \tau.
\end{align*}
Finally, the compatibility of the signed intersection with the $\Gamma$-action on $\mathcal{H}_\infty$ yields
\[
\delta\cap \Delta_{\tau,\infty}=(\beta_\tau^{-1}\centerdot \delta)\cap \Delta_{\beta_\tau^{-1}\centerdot \tau,\infty}=(\beta_\tau^{-1}\centerdot \delta)\cap \Delta_{\omega,\infty}\overset{\ref{prop:comparisomprop}}{=} -I_\omega(\beta_\tau^{-1}\centerdot \delta),
\]
and the corollary follows from Identity \eqref{eq:identtityforprop}.
\end{proof}
We now turn to the proof of Proposition \ref{prop:comparisomprop}.
If the geodesic segments $\delta$ and $\delta_{n,n+1}$ do not intersect, then $\mathrm{EZ}_2(\delta\otimes \delta_{n,n+1})\in C_2(\mathcal{H}_\infty^2\setminus \Delta_{1,\infty})$, and, thus, $\mathrm{EZ}_2(\delta\otimes \delta_{n,n+1})\cap \Delta_{1,\infty}=0$. 
In particular, if $\delta$ and $\Delta_{\omega,\infty}$ do not intersect, then $\delta$ and $\delta_{n,n+1}$ do not intersect for all $n\in \Z$, and, hence,
\[
I_\omega(\delta)=\sum_{n\in \Z}\mathrm{EZ}_2(\delta\otimes \delta_{n,n+1})\cap \Delta_{1,\infty} =0=-\delta\cap \Delta_{\omega,\infty}.
\]
Proposition \ref{prop:comparisomprop} follows in this situation.  
It remains to consider the case that $\delta$ and $\Delta_{\omega,\infty}$ intersect. 
The automorph $\gamma_\omega$ translates points on the geodesic $\Delta_{\omega,\infty}$ along its orientation and preserves the length of geodesic segments. Thus, there is precisely one segment $\delta_{n_0,n_0+1}$ which meets $\delta$. This implies:
\[
I_\omega(\delta)=\mathrm{EZ}_2(\delta\otimes \delta_{n_0,n_0+1})\cap \Delta_{1,\infty}.
\]
We will from now on write $\delta_\omega\coloneq \delta_{n_0,n_0+1}$ for this segment.
Note that changing the orientation of $\delta$, that is, replacing $\delta$ by the geodesic segment $\overline{\delta}$ obtained by interchanging source and target, reverses the sign on both sides of the equation in Proposition \ref{prop:comparisomprop}. Therefore, we may assume that $\delta\cap \Delta_{\omega,\infty}=-1$, and Proposition \ref{prop:comparisomprop} holds if $\mathrm{EZ}_2(\delta\otimes \delta_\omega)\cap \Delta_{1,\infty}=1$.
Let $\vartheta\in \Delta_{\omega,\infty}$ be the unique intersection point of $\delta$ and $\Delta_{\omega,\infty}$, $P=(\vartheta,\vartheta)\in \Delta_{1,\infty}$ and $W\coloneq \mathrm{T}_\vartheta(\mathcal{H}_\infty)$ be the tangent space at $\vartheta$.
Recall from the proof of Lemma \ref{lem:idenfificationwithunit} that the intersection number of a $2$-chain with $\Delta_{1,\infty}$ is computed by considering the exponential map 
\[
\exp_P=\exp_\vartheta\times \exp_\vartheta\colon W^2\setminus \Delta(W)\too \mathcal{H}_\infty^2\setminus \Delta_{1,\infty},
\]
and the map $d\colon W^2\setminus \Delta(W)\to W\setminus \{0\}$ which maps $(w_1,w_2)$ to $w_1-w_2$.
More precisely, these maps induce isomorphisms
\begin{center}
\begin{minipage}{\linewidth}
\centering
\begin{tikzcd}
\Z=\mathrm{H}_1(W\setminus \{0\}) & \mathrm{H}_1(W^2\setminus \Delta(W)) \arrow[l, "d_\star"']\arrow[rr, "(\exp_P)_\star"]& & \mathrm{H}_1(\mathcal{H}_\infty^2\setminus \Delta_{1,\infty}).
\end{tikzcd}
\end{minipage}
\end{center}
We first give a preimage of the chain $\mathrm{EZ}_2(\delta\otimes \delta_{\omega})$ by means of the preimages of $\delta$ and $\delta_\omega$ under the exponential map $\exp_\vartheta$; let $\lambda_1\in C_1(W)$ be the preimage of the geodesic segment $\delta$ and let $\lambda_2\in C_1(W)$ be the preimage of $\delta_{\omega}$ under the exponential map $(\exp_\vartheta)_\star\colon C_1(W)\to C_1(\mathcal{H}_\infty)$, respectively.
Slightly abusing notation, we denote by
\[
\mathrm{EZ}_\bullet\colon C_\bullet(W)\otimes C_\bullet(W)\too C_\bullet(W^2)
\]
the Eilenberg--Zilber map for the product space $W^2$.
The following diagram is readily verified to be commutative:
\begin{center}
\begin{minipage}{\linewidth}
\centering
\begin{tikzcd}
 & C_\bullet(W^2) \arrow[rd, "(\exp_P)_\star"] &  \\
C_\bullet(W)\otimes C_\bullet(W)\arrow[ur, "\mathrm{EZ}"] \arrow[dr, "(\exp_\vartheta)_\star \otimes (\exp_\vartheta)_\star"] & & C_\bullet(\mathcal{H}_\infty^2). \\
 &  C_\bullet(\mathcal{H}_\infty)\otimes C_\bullet(\mathcal{H}_\infty) \arrow[ur, "\mathrm{EZ}"] & 
 \end{tikzcd}
\end{minipage}
\end{center} 
Therefore, the $2$-chain $\mathrm{EZ}_2(\lambda_1\otimes \lambda_2)\in C_2(W^2)$ is a desired preimage of $\mathrm{EZ}_2(\delta\otimes \delta_\omega)$.
Next, we apply the map $d_\star$ to this $2$-chain and obtain 
\begin{equation}\label{eq:dstar}
d_\star\left( \mathrm{EZ}_2(\lambda_1\otimes \lambda_2)\right)=[\lambda_1\circ s_1 - \lambda_2\circ s_0] -[\lambda_1\circ s_0 - \lambda_2\circ s_1],
\end{equation}
where $s_0$ and $s_1$ are the degeneracy maps defined in Section \ref{sec:eilenberg}.
Since $\delta$ and $\delta_\omega$ are two different geodesic segments that meet transversally at $\vartheta$, there are two linearly independent vectors $v_1,v_2\in W$, and two strictly increasing continuous functions $f_1,f_2\colon [0,1]\to \mathbb{R}$ such that 
\begin{enumerate}
\item[\textbullet] $\lambda_i(t)=f_i(t)\cdot v_i$, for $i=1,2$ and 
\item[\textbullet] $f_1$ and $f_2$ admit a zero in $(0,1)$.
\end{enumerate}
Since $\delta\cap\Delta_{\omega,\infty}=-1$ and since $\delta_\omega$ and $\Delta_{\omega,\infty}$ share the same orientation, the basis $(v_1,v_2)$ is positively oriented.
Therefore, the isomorphism of real vector spaces $W\to \R^2$ that maps the basis $(v_1,v_2)$ to the standard basis $(e_1,e_2)$ induces the identity  map
\[
\mathrm{id}\colon \Z=\mathrm{H}_1(W\setminus \{0\})\too \mathrm{H}_1(\mathbb{R}^2\setminus \lbrace 0\rbrace)=\Z
\]
in homology.
Thus, by applying this map to the chain in \eqref{eq:dstar}, the desired intersection number $\mathrm{EZ}_2(\delta\otimes \delta_\omega)\cap \Delta_{1,\infty}$ equals the class
\[
\left[\partial_2\left(\begin{psmallmatrix}f_1\circ s_1 \\  -f_2\circ s_0\end{psmallmatrix} -\begin{psmallmatrix} f_1\circ s_0\\ -f_2\circ s_1\end{psmallmatrix}\right)\right]\in \mathrm{H}_1(\R^2\setminus \{0\})=\Z.
\]
The following picture illustrates this $1$-cycle:
\begin{center}
\begin{tikzpicture}[
    scale=0.8,
    edge/.style={thick,postaction={decorate},
      decoration={markings,mark=at position 0.6 with {\arrow{Stealth[length=3mm,width=2.2mm]}}}},
    axis/.style={-{Stealth},gray!70,thin},
    lbl/.style={font=\scriptsize,gray!80!black},
    rot/.style={-{Stealth[length=2mm]},semithick}
  ]
  \def\xa{-1.5}\def\xb{1.5}\def\yt{1.5}\def\yb{-1.5}
  \useasboundingbox (-3.2,-2.8) rectangle (3.2,2.8);

  \draw[axis] (-2.2,0) -- (2.2,0) node[right,font=\small,black] {$x_1$};
  \draw[axis] (0,-2.2) -- (0,2.2) node[above,font=\small,black] {$x_2$};

  \foreach \x in {\xa,\xb} \draw[gray!70] (\x,0.06) -- (\x,-0.06);
  \foreach \y in {\yt,\yb} \draw[gray!70] (0.06,\y) -- (-0.06,\y);
  \node[below,lbl] at (\xa,\yb) {$f_1(0)$};
  \node[below,lbl] at (\xb,\yb) {$f_1(1)$};
  \node[left,lbl]  at (\xa,\yt) {$-f_2(0)$};
  \node[left,lbl]  at (\xa,\yb) {$-f_2(1)$};

  \draw[edge] (\xa,\yt) -- (\xb,\yt);
  \draw[edge] (\xb,\yt) -- (\xb,\yb);
  \draw[edge] (\xb,\yb) -- (\xa,\yb);
  \draw[edge] (\xa,\yb) -- (\xa,\yt);

\end{tikzpicture}
\end{center}
By our convention, the class of this cycle equals $1$. 
This finishes the proof of the remaining case, and Proposition \ref{prop:comparisomprop} follows.

\subsection{Proof of the main theorem}
The main theorem (Theorem \ref{mainthm}) follows from the lemma below, which we prove in the remainder of this section:
\begin{lemma}\label{lemma:identitxoftwo}
We have the identity of $2$-cocycles
\[
\inte(\mathcal{D}_1)=2\cdot \mathcal{D}_\omega\times \eta_\omega
\]
on $\mathscr{C}_1^x(\mathcal{H}_\infty)\otimes \mathscr{C}_1^y(\Delta_{\omega,\infty})$.
\end{lemma}
\begin{proof}
Both sides are $\Gamma\times\Gamma_\omega^+$-equivariant $2$-cocycles on
$\mathscr{C}_\bullet^x(\mathcal{H}_\infty)\otimes\mathscr{C}_\bullet^y(\Delta_{\omega,\infty})$, and vanish on the summand $\mathscr{C}_0^x(\mathcal{H}_\infty)\otimes\mathscr{C}_2^y(\Delta_{\omega,\infty})$. Indeed, the left-most cocycle vanishes on this summand by Equality \eqref{eq:firstobserve}, and the right-most cocycle vanishes since it is supported on $\mathscr{C}_1^x(\mathcal{H}_\infty)\otimes\mathscr{C}_1^y(\Delta_{\omega,\infty})$.
Let $\delta\in\mathscr{C}_1^x(\mathcal{H}_\infty)$ be a geodesic segment and let
$\sigma\in\mathscr{C}_1^y(\Delta_{\omega,\infty})$ be a $1$-cycle. The triviality of $\mathrm{H}_1(\mathscr{C}_\bullet^y(\Delta_{\omega,\infty}))$ implies that there is a $2$-chain $c\in\mathscr{C}_2^y(\Delta_{\omega,\infty})$ with $\partial_2(c)=\sigma$, and hence
\[
\delta\otimes\sigma=\partial_1(\delta)\otimes c-\partial_3(\delta\otimes c).
\]
Both cocycles vanish on $\partial_3(\delta\otimes c)$ and, by the observation above, on
$\partial_1(\delta)\otimes c$. Thus, both cocycles vanish on $\delta\otimes\sigma$ for
every $1$-cycle $\sigma$.
We apply this observation to the following $1$-cycles: we can take $\sigma_1=\delta_{n_0,n_1}+\delta_{n_1,n_0}$, and $\sigma_2=\delta_{n_0,n_1}-\sum_{k=n_0}^{n_1-1}\delta_{k,k+1}$ for $n_0<n_1$.
The vanishing on $\delta\otimes \sigma_1$ implies that we can assume without loss of generality that $n_0<n_1$, and the vanishing on $\delta\otimes \sigma_2$ implies that it suffices to prove the identity for $\delta\otimes\delta_{k,k+1}$ for $k\in\Z$.
Since $(1,\gamma_\omega^k)\cdot(\delta\otimes\delta_{0,1})=\delta\otimes\delta_{k,k+1}$ and
$\Gamma_\omega^+$ acts trivially on $\Z$, the values of both cocycles at
$\delta\otimes\delta_{k,k+1}$ and at $\delta\otimes\delta_{0,1}$ coincide. Finally, the identity for $\delta\otimes\delta_{0,1}$ is Corollary \ref{cor:comparison}.
\end{proof}

\section{Evaluations and a conjecture of Darmon--Vonk}\label{section5}
The aim of the last section of this article is to prove the antisymmetry conjecture from the introduction. 
\subsection{The Darmon--Vonk case}\label{RMCalaDV}
Recall that the divisor map induces an exact sequence in cohomology
\[
 \mathrm{H}^1(\Gamma,\mathcal{A}_{1}^\times)\too \mathrm{H}^1(\Gamma,\mathcal{M}_{1}^\times)\xrightarrow{\mathrm{div}} \mathrm{H}^1(\Gamma,\mathrm{Div}^\dagger(\mathcal{H}_p)).
\]
A \emph{rigid meromorphic cocycle} (in the sense of \cite{DarmonVonksingularmoduli}) is a class in $\mathrm{H}^1(\Gamma,\mathcal{M}_{1}^\times)$ whose associated divisor belongs to $\mathcal{RM}(\Gamma)$. Thus,
\begin{equation}
\mathrm{H}^1_{\mathrm{RC}}(\Gamma,\mathcal{M}_{1}^\times)\coloneq \mathrm{div}^{-1}\left(\mathcal{RM}(\Gamma)\right)
\end{equation}
is the group of rigid meromorphic cocycles. Proposition \ref{prop:class} implies that
\[
\mathcal{RM}(\Gamma)=\bigoplus_{\omega\in \Gamma\backslash \mathcal{H}_p^{\mathrm{RM}}} \mathbb{Z}\cdot [\mathcal{D}_\omega].
\]
Therefore, if $\mathcal{J}\in \mathrm{H}^1_{\mathrm{RC}}(\Gamma,\mathcal{M}_{1}^\times)$ is a rigid meromorphic cocycle, then there are finitely many RM points $\omega_1,\cdots,\omega_m\in \Gamma\backslash \mathcal{H}_p^{\mathrm{RM}}$ such that 
\[
\mathrm{div}(\mathcal{J})=\sum_{i=1}^m n_i \cdot [\mathcal{D}_{\omega_i}]
\]
with $n_i\neq 0$. We say that $\mathcal{J}$ is \emph{regular at an RM point} if the RM point does not belong to any $G$-orbit of the RM points $\omega_1,\cdots,\omega_m$. \\
Below, we explain how one can meaningfully  evaluate a rigid meromorphic cocycle at an RM point $\tau\in \mathcal{H}_p$. The complement $X_\tau\coloneq (G\centerdot\tau)^c\subset \mathcal{H}_p$ of the $G$-orbit of $\tau$ is $G$-invariant. Thus, the group $\mathrm{Div}^\dagger(X_\tau)$ of locally finite divisors with support disjoint from $G\centerdot \tau$ is a $G$-module. We denote the preimage of $\mathrm{Div}^\dagger(X_\tau)$ under the divisor map by $\mathcal{M}_{\tau}^\times$. Observe:
\begin{lemma}\label{regulardefinitionmono}
The natural map 
\[
\mathrm{H}^1(\Gamma,\mathcal{M}_{\tau}^\times)\too \mathrm{H}^1(\Gamma,\mathcal{M}_{1}^\times),
\]
induced by the inclusion $\mathcal{M}_{\tau}^\times \subset \mathcal{M}_{1}^\times$, is injective.
\end{lemma}
\begin{proof}
The lemma is based on the fact that $\mathrm{H}^0(\Gamma,\mathrm{Div}^\dagger(\mathcal{H}_p))=0$. Indeed, let $\Delta=\sum n_z \cdot z$ be a $\Gamma$-invariant locally finite divisor with $n_z\neq 0$ for some $z\in \mathcal{H}_p$. Then the restriction of $\Delta$ to the open subset $U_z=\mathrm{red}^{-1}(\mathrm{red}(z))\subset \mathcal{H}_p$ is a finite sum. The stabilizer of $U_z$ in $\Gamma$, which is either conjugated to $\Gamma_0$ or $\Gamma_0^p$, contains a matrix $\gamma$ such that no power of $\gamma$ stabilizes $z$. Then $n_{\gamma^n z}=n_z\neq 0$ for all $n\in \Z$ contradicting that the restriction of $\Delta$ to $U_z$ is a finite sum. We apply this vanishing: the long exact sequence attached to the short exact sequence of $G$-modules
\[
0\too \mathcal{A}_{1}^\times\too \mathcal{M}_{\tau}^\times \too \mathrm{Div}^\dagger(X_\tau)\too 0
\]
together with the snake lemma implies that it suffices to prove that the homomorphism
\[
\mathrm{H}^1(\Gamma,\mathrm{Div}^\dagger(X_\tau))\too \mathrm{H}^1(\Gamma,\mathrm{Div}^\dagger(\mathcal{H}_p))
\]
is injective. This statement follows if we can show that
\[
\left(\mathrm{Div}^\dagger(\mathcal{H}_p)/\mathrm{Div}^\dagger(X_\tau)\right)^\Gamma=0.
\]
Let $\Delta\in \mathrm{Div}^\dagger(\mathcal{H}_p)$ be a locally finite divisor such that for any $\gamma\in \Gamma$ the divisor $\Delta-\gamma\centerdot \Delta$ belongs to $\mathrm{Div}^\dagger(X_\tau)$. We decompose $\Delta=\Delta_\tau+\Delta_0$ with $\Delta_0\in \mathrm{Div}^\dagger(G\centerdot \tau)$ and $\Delta_\tau\in \mathrm{Div}^\dagger(X_\tau)$. Then for every $\gamma\in \Gamma$ we must have $\Delta_0-\gamma\centerdot \Delta_0\in \mathrm{Div}^\dagger(X_\tau)$, which implies that $\Delta_0-\gamma\centerdot \Delta_0=0$. In particular, $\Delta_{0}\in \mathrm{Div}^\dagger(\mathcal{H}_p)^\Gamma=0$, and the claim follows. 
\end{proof}
Note that if $\mathcal{J}$ is regular at $\tau$, then there is always a lift to $\mathrm{H}^1(\Gamma,\mathcal{M}_\tau^\times)$. Indeed, the divisor of $\mathcal{J}$ lies in the image of $\mathrm{H}^1(\Gamma,\mathrm{Div}^\dagger(X_\tau))$. Both divisor sequences have the same kernel
$\mathcal{A}_1^\times$, thus, $\mathrm{div}(\mathcal{J})$ lifts to a class in $\mathrm{H}^1(\Gamma,\mathcal{M}_\tau^\times)$. Finally, two classes with the same divisor differ by an analytic cocycle in $\mathrm{H}^1(\Gamma,\mathcal{A}_1^\times)$, and the claim follows.\\
The evaluation map $\mathrm{ev}_\tau\colon \mathcal{M}_{\tau}^\times\to \mathbb{C}_p^\times$, which assigns to every rigid meromorphic function $f$ the $p$-adic number $f(\tau)\in \mathbb{C}_p^\times$, is a well-defined $\Gamma_\tau$-module homomorphism and, thus, induces a homomorphism in cohomology 
\begin{equation}\label{evaluationattauincoho}
\mathrm{ev}_\tau\colon \mathrm{H}^1(\Gamma,\mathcal{M}_{\tau}^\times)\xrightarrow{\mathrm{res}}\mathrm{H}^1(\Gamma_\tau^+,\mathcal{M}_{\tau}^\times)\too \mathrm{H}^1(\Gamma_\tau^+,\mathbb{C}_p^\times).
\end{equation}
The choice of the generator of $\Gamma_\tau^+$ determines an identification
\[
\mathrm{H}_1(\Gamma_\tau^+,\mathbb{Z})=\mathbb{Z}
\]
and we denote by $c_\tau\in \mathrm{H}_1(\Gamma_\tau^+,\mathbb{Z})$ the fundamental class corresponding to $1$.
\begin{definition}\label{defvalueii}
Let $\mathcal{J}\in \mathrm{H}^1_{\mathrm{RC}}(\Gamma,\mathcal{M}_{1}^\times)$ be a rigid meromorphic cocycle that is regular at $\tau$. By abuse of notation, identify $\mathcal{J}$ with its unique preimage in $\mathrm{H}^1(\Gamma,\mathcal{M}_{\tau}^\times)$. The \emph{value of $\mathcal{J}$ at $\tau$} is defined to be the image of $\mathcal{J}$ under the map 
\[
\mathrm{H}^1(\Gamma,\mathcal{M}_{\tau}^\times)\xrightarrow{\mathrm{ev}_\tau} \mathrm{H}^1(\Gamma^+_\tau,\mathbb{C}_p^\times)  \xrightarrow{-\cap c_\tau} \mathbb{C}_p^\times.
\]
We denote this $p$-adic number by $\mathcal{J}[\tau]$.
\end{definition}
Clearly, the image of every analytic cocycle, that is, an element of $\mathrm{H}^1(\Gamma,\mathcal{A}_{1}^\times)$, in $\mathrm{H}^1(\Gamma,\mathcal{M}_1^\times)$ is a rigid meromorphic cocycle and also regular at $\tau$. 
An example of an analytic cocycle is the so-called \emph{trivial cocycle}: ~we fix an isomorphism $\iota_p\colon B_p\to \mathrm{M}_2(\mathbb{Q}_p)$ that satisfies $\iota_p(R\otimes \mathbb{Z}_p)=\mathrm{M}_2(\mathbb{Z}_p)$. This isomorphism gives rise to an identification of $\mathcal{H}_p$ with $\mathbb{P}^1(\mathbb{C}_p)\setminus \mathbb{P}^1(\mathbb{Q}_p)$. More precisely, the map
\[
\mathbb{P}^1(\mathbb{C}_p)\setminus \mathbb{P}^1(\mathbb{Q}_p)\too  \mathcal{H}_p, \qquad \omega\mapsto \left[\begin{psmallmatrix} \omega & -\omega^2\\ 1 & -\omega\end{psmallmatrix}\right]
\]
is an isomorphism of rigid analytic spaces. We consider the coordinate function $z(x)=\frac{x_0}{x_1}$ on $\mathbb{P}^1$. In terms of the bar resolution $\Z[\Gamma^{\bullet+1}]$, the trivial cocycle is induced by the $1$-cocycle 
\[
J_\mathrm{triv}\colon \Z[\Gamma^2]\too \mathcal{A}_1^\times, \qquad \left(1,\begin{psmallmatrix} a & b \\ c & d\end{psmallmatrix}\right)\mapsto -cz+a.
\] 
By abuse of notation, we denote by $J_{\mathrm{triv}}$ the cocycle and its class. The value of $J_\mathrm{triv}$ at $\tau$ is easily computed: 
\begin{lemma}\label{corvalueatanalyticisunit}
We have $J_\mathrm{triv}[\tau]=\varepsilon_\tau^{-1}$, where $\varepsilon_\tau\in\Q(\tau)$ is the unit associated to $\tau$ (see Section \ref{sec:3-dim}).
\end{lemma}
\begin{proof}
Recall that the unit $\varepsilon_\tau$ is determined by the condition
\[
\gamma_\tau\cdot u=\varepsilon_\tau\cdot u \quad \mathrm{with} \quad u=\left(\begin{array}{rr} \tau & -\tau^2\\ 1 & -\tau\end{array}\right).
\]
A straightforward computation shows that $\gamma_\tau\cdot u=\left(J_{\mathrm{triv}}(\gamma_\tau^{-1})(\tau)\right)\cdot u$. Thus, $\varepsilon_\tau^{-1}=J_{\mathrm{triv}}(\gamma_\tau)(\tau)$. The lemma follows.
\end{proof} 
Besides the observation that the values of $J_\mathrm{triv}$ are explicit algebraic numbers, the lemma above also implies that $J_\mathrm{triv}$ is a non-torsion class in $\mathrm{H}^1(\Gamma,\mathcal{A}_1^\times)$. In fact, the evaluation at $\tau$ is a group homomorphism $\mathrm{H}^1(\Gamma,\mathcal{A}_1^\times)\to\C_p^\times$, and $\varepsilon_\tau$ is non-torsion.
This implies the following: 
\begin{lemma}\label{lem:finiteindex}
The group $\mathrm{H}^1(\Gamma,\mathcal{A}_1^\times)$ is a finitely generated abelian group, whose torsion-free part is of rank $1$.
In particular, $J_\mathrm{triv}$ generates a subgroup of finite index in $\mathrm{H}^1(\Gamma,\mathcal{A}_1^\times)$. 
\end{lemma}
\begin{proof}
The projection map $\mathcal{A}_1^\times\to \mathcal{A}_1^\times/\C_p^\times$ induces the exact sequence in cohomology:
\[
\mathrm{H}^1(\Gamma,\C_p^\times)\too \mathrm{H}^1(\Gamma,\mathcal{A}_1^\times) \too \mathrm{H}^1(\Gamma,\mathcal{A}_1^\times/\C_p^\times) \xrightarrow{\delta} \mathrm{H}^2(\Gamma,\C_p^\times).
\]
The left-most group is finite by the Universal Coefficient Theorem, because $\mathrm{H}^1(\Gamma,\Z)=0$. 
Furthermore, the long exact sequence induced from the van der Put exact sequence implies that the group of theta cocycles is finitely generated.
In conclusion, this yields that $\mathrm{H}^1(\Gamma,\mathcal{A}_1^\times)$ is also finitely generated.
Since $J_\mathrm{triv}$ is a non-torsion class in this group, the rank of the torsion-free part is at least $1$.
Below we show that the rank is at most $1$.
If $f\in\mathcal{A}_1^\times$, then the function $z\mapsto|f(z)|\in p^\Q$ is constant on every affinoid subset $U_v=\mathrm{red}^{-1}(v)$. See for this the first line below the formula for a vertex affinoid in the proof of \cite[Proposition 1.1]{vanderputharmoni}. 
We denote by $\mathrm{ord}_p (f|_{U_v})\in \Q$ the rational number $\mathrm{ord}_p(f(z))$ for $z\in U_v$.
We define the map
\[
\mathrm{ord}_{\mathcal{V}}\colon \mathcal{A}_1^\times\too\mathcal{F}(\mathcal{V},\Q), \qquad  f\mapsto(v\mapsto\mathrm{ord}_p (f|_{U_v})),
\]
which is $\GL_2(\Q_p)$-equivariant, and maps constant functions $\mathcal{H}_p\to \C_p$ to constant functions $\mathcal{V}\to \Q$. Moreover, it fits into the following commutative diagram:
\begin{center}
\begin{minipage}{\linewidth}
\centering
\begin{tikzcd}
0\arrow[r] & \C_p^\times\arrow[r] \arrow[d, "\mathrm{ord}_p"]  &\mathcal{A}_1^\times \arrow[r] \arrow[d, "\mathrm{ord}_{\mathcal{V}}"] & \mathcal{A}_1^\times/\C_p^\times\arrow[r] \arrow[d, "\mathrm{vdP}"] &  0\\
0 \arrow[r] &\Q \arrow[r]& \mathcal{F}(\mathcal{V},\Q) \arrow[r, "d"]& \mathcal{F}_0(\mathcal{E}^{\mathrm{or}},\Q) \arrow[r] & 0.
 \end{tikzcd}
\end{minipage}
\end{center} 
The lower row is the short exact sequence that induces the Mayer--Vietoris exact sequence. That is, for $f\in \mathcal{F}(\mathcal{V},\Q)$ the function $d(f)$ is defined by
\[
d(f)\colon e\mapsto f(s(e))-f(t(e)).
\]
See the part below the diagram in the proof of \cite[Proposition 1.1]{vanderputharmoni} for the commutativity of the right-most square.
We derive the commutative diagram in cohomology
\begin{center}
\begin{minipage}{\linewidth}
\centering
\begin{tikzcd}
\mathrm{H}^1(\Gamma,\mathcal{A}_1^\times) \arrow[rr] \arrow[d] & & \mathrm{H}^1(\Gamma,\mathcal{A}_1^\times/\C_p^\times) \arrow[r, "\delta"] \arrow[d, "\lambda"] & \mathrm{H}^2(\Gamma,\C_p^\times)\arrow[d, "\mathrm{ord}_p"] \\
\mathrm{H}^1(\Gamma_0,\Q) \oplus \mathrm{H}^1(\Gamma_0^\prime,\Q)  \arrow[rr, "d"]& & \mathrm{H}^1(\Gamma_0^p,\Q) \arrow[r, "\delta^\prime"] & \mathrm{H}^2(\Gamma,\Q).
 \end{tikzcd}
\end{minipage}
\end{center} 
Therefore, if $J$ is a theta cocycle that lifts to an analytic cocycle, then $\delta^\prime\circ \lambda(J)=0$. Thus, $\lambda(J)\in \mathrm{im}(\lambda)\cap \mathrm{im}(d)$.
We claim that $ \mathrm{im}(\lambda)_\Q\cap \mathrm{im}(d)_\Q=0$. Indeed, if $B$ is split, then the left-most terms of the lower row vanish, which implies that $\mathrm{im}(d)_\Q=0$.
Otherwise, if $B$ is non-split, then the map $d$ corresponds under the Eichler--Shimura identifications to the map 
\[
\mathcal{S}_2(\Gamma_0)\oplus \mathcal{S}_2(\Gamma_0^\prime)\too \mathcal{S}_2(\Gamma_0^p).
\]
The image of this map is the space of oldforms. Thus, the claim follows because $\mathrm{im}(\lambda)_\Q=N_\Q$ corresponds to the space of newforms, which is the orthogonal complement to the space of oldforms.
In both cases, we deduce that $J\in \mathrm{ker}(\delta)_\Q$ implies that $J\in \iota(L_\Q)\simeq \Q$.
In other words, the projection map induces the exact sequence
\[
0 \too \mathrm{H}^1(\Gamma,\mathcal{A}_1^\times)_\Q\too L_\Q \simeq \Q,
\]
which implies that the rank is at most $1$. The lemma follows.
\end{proof}
We define the group
\[
E_\tau\coloneq\{x\in\C_p^\times : x^m\in\varepsilon_\tau^{\Z}\text{ for some } m\geq 1\}.
\]
Lemma \ref{corvalueatanalyticisunit} and Lemma \ref{lem:finiteindex} imply that 
\begin{equation}\label{eq:trivialvalues}
\mathrm{H}^1(\Gamma,\mathcal{A}_1^\times)[\tau]\subset E_\tau.
\end{equation}
From now on, let $\omega\in \mathcal{H}_p$ be an RM point which is not contained in the $G$-orbit of $\tau$, and let $\mathrm{T}\in \mathcal{I}_p$ be a Hecke operator. The class $\mathrm{T}\cdot [\mathcal{D}_\omega]$ lifts to a class $\mathcal{J}_\omega^\mathrm{T}\in \mathrm{H}^1_{\mathrm{RC}}(\Gamma,\mathcal{M}_1^\times)$ which is unique up to multiplication with an analytic cocycle. The rigid meromorphic cocycle $\mathcal{J}_\omega^\mathrm{T}$ is regular at $\tau$, thus we can evaluate $\mathcal{J}_\omega^\mathrm{T}$ at $\tau$. We define the $p$-adic number
\begin{equation}
J_p^\mathrm{T}(\tau,\omega)\coloneq \mathcal{J}_\omega^\mathrm{T}[\tau]\in \mathbb{C}_p^\times,
\end{equation}
which is well-defined up to multiplication with an element in $E_\tau$. Numerical experiments show that the assignment $J_p^\mathrm{T}(\tau,\omega)$ enjoys strong parallels with the difference of two classical singular moduli. In particular, Darmon--Vonk conjecture that it is \emph{antisymmetric} in its argument (see \cite[Remark 3.19]{DarmonVonksingularmoduli}). 
\begin{conjecture}\label{mainconjecturetoprove}
We have 
\[
J_p^\mathrm{T}(\omega,\tau)=J_p^\mathrm{T}(\tau,\omega)^{-1}
\]
up to multiplication by an element in $E_\tau \cdot E_\omega$.
\end{conjecture}
We will give a proof of this conjecture in the following subsections by expressing the $p$-adic number $J_p^\mathrm{T}(\omega,\tau)$ as the value of a rigid meromorphic cocycle for $\Gamma^2$ at the pair $(\tau,\omega)$.
\subsection{The $4$-dimensional case}
Recall that there is a natural exact sequence of Hecke modules
\[
\mathrm{H}^2(\Gamma^2,\mathcal{A}_{2}^\times)\too \mathrm{H}^2(\Gamma^2,\mathcal{M}_{\mathrm{rq}}^\times)\xrightarrow{\mathrm{div}} \mathrm{H}^2(\Gamma^2,\mathrm{Div}_{\mathrm{rq}}^\dagger(\mathcal{H}_p^2)).
\]
A \emph{rigid meromorphic cocycle for $\Gamma^2$} is a class in $\mathrm{H}^2(\Gamma^2,\mathcal{M}_{\mathrm{rq}}^\times)$ whose associated divisor is a Kudla--Millson divisor. Thus, 
\[
\mathrm{H}^2_{\mathrm{RC}}(\Gamma^2,\mathcal{M}_{\mathrm{rq}}^\times)\coloneq \mathrm{div}^{-1}\left(\mathcal{KM}(\Gamma^2)\right)
\]
is the group of rigid meromorphic cocycles for $\Gamma^2$.\\
Below, we define the value of a rigid meromorphic cocycle for $\Gamma^2$ at a pair of RM points. Fix two RM points $\omega,\tau\in \mathcal{H}_p$ which do not belong to the same $G$-orbit. That is, $(\tau,\omega)\notin \Delta_{v,p}$ for all elements $v\in G$. In particular, every rigid meromorphic function $f\in \mathcal{M}_{\mathrm{rq}}^\times$ is regular at the pair $(\tau,\omega)$ and we can consider the well-defined map $\mathrm{ev}_{(\tau,\omega)}\colon \mathcal{M}_{\mathrm{rq}}^\times\to \mathbb{C}_p^\times$ which maps every rigid meromorphic function $f$ to the $p$-adic number $f(\tau,\omega)$. Clearly, the evaluation map is a $\Gamma_\tau\times \Gamma_\omega$-module homomorphism and, thus, induces a homomorphism in cohomology 
\[
\mathrm{ev}_{(\tau,\omega)}\colon \mathrm{H}^2(\Gamma^2,\mathcal{M}_{\mathrm{rq}}^\times)\xrightarrow{\res} \mathrm{H}^2(\Gamma_\tau^+\times \Gamma_\omega^+,\mathcal{M}_{\mathrm{rq}}^\times)\too \mathrm{H}^2(\Gamma_\tau^+\times \Gamma_\omega^+,\mathbb{C}_p^\times).
\]
The K\"unneth Formula for group homology yields that the \emph{homology cross-product} 
\[
\times\colon \mathrm{H}_1(\Gamma_\tau^+,\mathbb{Z})\otimes \mathrm{H}_1(\Gamma_\omega^+,\mathbb{Z})\too \mathrm{H}_2(\Gamma_\tau^+\times \Gamma_\omega^+,\mathbb{Z})
\]
is an isomorphism, and we identify 
\[
\mathrm{H}_2(\Gamma_\tau^+\times \Gamma_\omega^+,\mathbb{Z})=\mathbb{Z}
\]
by mapping the generator $c_{\tau,\omega}\coloneq c_\tau\times c_\omega$ to $1$. Finally, we can define:
\begin{definition}
Let $\mathcal{J}\in \mathrm{H}^2(\Gamma^2,\mathcal{M}_{\mathrm{rq}}^\times)$ be a class. The \emph{value of $\mathcal{J}$ at the pair $(\tau,\omega)$} is defined to be the image of $\mathcal{J}$ under the homomorphism 
\[
\mathrm{H}^2(\Gamma^2,\mathcal{M}_{\mathrm{rq}}^\times) \xrightarrow{\mathrm{ev}_{(\tau,\omega)}}  \mathrm{H}^2(\Gamma_\tau^+\times\Gamma_\omega^+,\mathbb{C}_p^\times)\xrightarrow{\cap c_{\tau,\omega}}  \mathbb{C}_p^\times.
\]
We denote this $p$-adic number by $\mathcal{J}[\tau,\omega]$.
\end{definition}
The next lemma shows that the defined $p$-adic number is always a root of unity if the cohomology class $\mathcal{J}$ takes values in the group $\mathcal{A}_{2}^\times$:
\begin{lemma}\label{lemmasigma4}
The image of the restriction map
\[
\res\colon \mathrm{H}^2(\Gamma^2,\mathcal{A}_{2}^\times) \too \mathrm{H}^2(\Gamma_\tau^+\times \Gamma_\omega^+,\mathcal{A}_{2}^\times)
\]
is torsion. Thus, for every class $\mathcal{J}\in \mathrm{H}^2(\Gamma^2,\mathcal{A}_{2}^\times)$ we have
\[
\mathcal{J}[\tau,\omega]\in (\mathbb{C}_p^\times)_{\mathrm{tor}}.
\]
\end{lemma}
\begin{proof}
The isomorphism \eqref{propropropbloed} implies that we have a short exact sequence of $\Gamma^2$-modules
\begin{equation*}
1\too \mathbb{C}_p^\times \too \mathcal{A}_{1}^\times  \oplus \mathcal{A}_{1}^\times \too \mathcal{A}_{2}^\times\too 1.
\end{equation*}
Thus, we derive an exact sequence
\begin{equation*}
\mathrm{H}^2(\Gamma^2,\mathcal{A}_{1}^\times\otimes \mathbb{Z})\ \oplus  \ \mathrm{H}^2(\Gamma^2,\mathbb{Z}\otimes \mathcal{A}_{1}^\times)\too \mathrm{H}^2(\Gamma^2,\mathcal{A}_{2}^\times)\too \mathrm{H}^3(\Gamma^2,\mathbb{C}_p^\times).
\end{equation*}
The group $\mathrm{H}^3(\Gamma^2,\mathbb{C}_p^\times)$ is finite by Lemma \ref{computationofh3gamma2} and, therefore, it is enough to prove that the images of the restriction maps
\begin{equation*}
\res\colon \mathrm{H}^2(\Gamma^2,\mathcal{A}_{1}^\times\otimes \mathbb{Z})\too \mathrm{H}^2(\Gamma_\tau^+\times \Gamma_\omega^+,\mathcal{A}_{1}^\times\otimes \mathbb{Z})
\end{equation*}
and
\[
\res\colon \mathrm{H}^2(\Gamma^2, \mathbb{Z}\otimes \mathcal{A}_{1}^\times)\too \mathrm{H}^2(\Gamma_\tau^+\times \Gamma_\omega^+,\mathbb{Z}\otimes \mathcal{A}_{1}^\times)
\]
are torsion. To see this, note that the K\"unneth Formula for group cohomology implies that the cohomology cross-product
\[
\mathrm{H}^2(\Gamma,\mathcal{A}_{1}^\times)\otimes \mathrm{H}^0(\Gamma,\mathbb{Z}) \ \oplus \ \mathrm{H}^0(\Gamma,\mathcal{A}_{1}^\times)\otimes \mathrm{H}^2(\Gamma,\mathbb{Z}) \xrightarrow{\times} \mathrm{H}^2(\Gamma^2,\mathcal{A}_{1}^\times\otimes \mathbb{Z})
\]
is a monomorphism with finite cokernel. Here, we used that $\mathrm{H}^1(\Gamma,\Z)=0$. The cohomology cross-product is compatible with restriction maps, and the diagram
\begin{center}
\begin{minipage}{\linewidth}
\centering
\begin{tikzcd}
\mathrm{H}^2(\Gamma,\mathcal{A}_{1}^\times)\otimes \mathrm{H}^0(\Gamma,\mathbb{Z}) \ \oplus \ \mathrm{H}^0(\Gamma,\mathcal{A}_{1}^\times)\otimes \mathrm{H}^2(\Gamma,\mathbb{Z}) \arrow[r]{r}{\times} \arrow[d]{d}{\res} & \mathrm{H}^2(\Gamma^2,\mathcal{A}_{1}^\times\otimes \mathbb{Z})\arrow[d]{d}{\res} \\
\mathrm{H}^2(\Gamma_\tau^+,\mathcal{A}_{1}^\times)\otimes \mathrm{H}^0(\Gamma_\omega^+,\mathbb{Z}) \ \oplus \ \mathrm{H}^0(\Gamma_\tau^+,\mathcal{A}_{1}^\times)\otimes \mathrm{H}^2(\Gamma_\omega^+,\mathbb{Z}) \ \arrow[r]{r}{\times} & \mathrm{H}^2(\Gamma_\tau^+\times \Gamma_\omega^+,\mathcal{A}_{1}^\times\otimes \mathbb{Z})
 \end{tikzcd}
\end{minipage}
\end{center} 
commutes. The left-most terms of the lower row vanish because $\Gamma_\omega^+$ and $\Gamma_\tau^+$ are both of cohomological dimension $1$, and the claim follows for the first restriction map. After changing the roles of $\mathcal{A}_{1}^\times$ and $\mathbb{Z}$, we derive the claim for the second restriction map, too, and the lemma follows.
\end{proof}
Lemma \ref{lemmasigma4} enables us to propose a candidate for the $p$-adic number $J_p^\mathrm{T}(\tau,\omega)$ in terms of rigid meromorphic cocycles for $\Gamma^2$. Let $\mathrm{T}\in \mathcal{I}_p$ be a Hecke operator. By Definition \ref{definitionip}, there is a class $\mathcal{J}_{1}^{\mathrm{T}}\in \mathrm{H}^2(\Gamma^2,\mathcal{M}_{\mathrm{rq}}^\times)$ that satisfies
\[
\mathrm{div}(\mathcal{J}_{1}^{\mathrm{T}})=(\mathrm{T}\times 1)\cdot [\mathcal{D}_1]\in \mathcal{KM}(\Gamma^2).
\]
This class is unique up to multiplication with a class in $\mathrm{H}^2(\Gamma^2,\mathcal{A}_{2}^\times)$. In particular, $\mathcal{J}_1^{\mathrm{T}}$ is a rigid meromorphic cocycle, and we define
\begin{equation}
\widehat{J}_p^\mathrm{T}(\tau,\omega)\coloneq \mathcal{J}_{1}^{\mathrm{T}}[\tau,\omega]\in \mathbb{C}_p^\times.
\end{equation}
This definition is well-defined up to multiplication with a root of unity in $\mathbb{C}_p$. The function $\widehat{J}_p^\mathrm{T}$ has the advantage that its antisymmetric behaviour is easy to prove:
\begin{proposition}\label{propforantisymmetry}
Let $\mathrm{T} \in \mathcal{I}_p$ be a Hecke operator. Then for all pairs of RM points $(\tau,\omega)$, not contained in the same $G$-orbit, we have the identity
\[
\widehat{J}_p^\mathrm{T}(\omega,\tau)=\widehat{J}_p^\mathrm{T}(\tau,\omega)^{-1}
\]
up to a root of unity.
\end{proposition}
\begin{proof}
Recall that we denote by $\varsigma_p\colon \mathcal{H}_p^2\to \mathcal{H}_p^2$ the isomorphism of rigid analytic spaces and by $t\colon G^2\to G^2$ the isomorphism of groups which switches the coordinates. The morphism $\varsigma_p$ induces the homomorphism of groups
\[
\varsigma_p\colon \mathcal{M}_{\mathrm{rq}}^\times \too \mathcal{M}_{\mathrm{rq}}^\times, \qquad f\mapsto f\circ \varsigma_p,
\]
which is compatible with $t$. Thus, the pair $(\varsigma_p,t)$ induces a map in cohomology $\sigma^\prime\colon \mathrm{H}^2(\Gamma^2,\mathcal{M}_\mathrm{rq}^\times)\to \mathrm{H}^2(\Gamma^2,\mathcal{M}_\mathrm{rq}^\times)$ that fits into the commutative diagram
\begin{center}
\begin{minipage}{\linewidth}
\centering
\begin{tikzcd}
\mathrm{H}^2(\Gamma^2,\mathcal{M}_{\mathrm{rq}}^\times) \arrow[r]{r}{\mathrm{div}} \arrow[d]{d}{\sigma^\prime} & \mathrm{H}^2(\Gamma^2,\mathrm{Div}^\dagger_{\mathrm{rq}}(\mathcal{H}_p^2))\arrow[d]{d}{\sigma}\\
\mathrm{H}^2(\Gamma^2,\mathcal{M}_{\mathrm{rq}}^\times) \arrow[r]{r}{\mathrm{div}} & \mathrm{H}^2(\Gamma^2,\mathrm{Div}^\dagger_{\mathrm{rq}}(\mathcal{H}_p^2)).
 \end{tikzcd}
\end{minipage}
\end{center}
By Corollary \ref{symmetryofdn}, $\sigma$ fixes every Kudla--Millson divisor. Hence, $\sigma$ fixes $(\mathrm{T}\times 1)\cdot [\mathcal{D}_1]$, and, thus, the commutative diagram above implies the equality 
\[
\mathrm{div}\left(\sigma^\prime\left(\mathcal{J}_1^\mathrm{T}\right)\right)=\mathrm{div}\left(\mathcal{J}_1^\mathrm{T}\right).
\]
Therefore:
\begin{equation}\label{keyidentityinlastsecton}
\sigma^\prime\left(\mathcal{J}_1^\mathrm{T}\right)=\mathcal{J}_1^\mathrm{T} \ \mathrm{mod} \ \mathrm{H}^2(\Gamma^2,\mathcal{A}_2^\times).
\end{equation}
The homomorphism $\sigma^\prime$ is compatible with restriction maps and, moreover, we have the obvious identity
\[
\mathrm{ev}_{(\tau,\omega)}=\mathrm{ev}_{(\omega,\tau)}\circ \varsigma_p\colon \mathcal{M}_{\mathrm{rq}}^\times \too \mathbb{C}_p^\times.
\]
Consequently, we have the following commutative diagram:
\begin{center}
\begin{minipage}{\linewidth}
\centering
\begin{tikzcd}
\mathrm{H}^2(\Gamma^2,\mathcal{M}_{\mathrm{rq}}^\times) \arrow[r, "\mathrm{res}"] \arrow[d, "\sigma^\prime"] & \mathrm{H}^2(\Gamma_\tau^+\times \Gamma_\omega^+,\mathcal{M}_{\mathrm{rq}}^\times) \arrow[r, "\mathrm{ev}_{(\tau,\omega)}"] \arrow[d, "\sigma^\prime"] & \mathrm{H}^2(\Gamma_\tau^+\times \Gamma_\omega^+,\mathbb{C}_p^\times) \arrow[d, "t^\star"] \\
\mathrm{H}^2(\Gamma^2,\mathcal{M}_{\mathrm{rq}}^\times) \arrow[r, "\res"] & \mathrm{H}^2(\Gamma_\omega^+\times \Gamma_\tau^+,\mathcal{M}_{\mathrm{rq}}^\times) \arrow[r, "\mathrm{ev}_{(\omega,\tau)}"] & \mathrm{H}^2(\Gamma_\omega^+\times \Gamma_\tau^+,\mathbb{C}_p^\times).
 \end{tikzcd}
\end{minipage}
\end{center}
Finally, the isomorphism in homology 
\[
t_\star=(t,\mathrm{id})_\star\colon \mathrm{H}_2(\Gamma_\omega^+\times \Gamma_\tau^+,\mathbb{Z})\too \mathrm{H}_2(\Gamma_\tau^+\times \Gamma_\omega^+,\mathbb{Z})
\]
maps the generator $c_{\omega,\tau}$ to $-c_{\tau,\omega}$. The proposition is now an immediate consequence of this generator swap, identity \eqref{keyidentityinlastsecton} and the commutative diagram above.
\end{proof}

\subsection{Comparison of the functions $J_p$ and $\widehat{J}_p$}
The article concludes by proving the following theorem:
\begin{theorem}\label{proptoproveconjecture}
We have 
\[
J_p^\mathrm{T}(\tau,\omega)^{-2}= \widehat{J}_p^{\mathrm{T}}(\tau,\omega)
\]
up to multiplication with an element in $E_\tau$.
\end{theorem}
Before we establish this theorem, we show that it implies the desired Darmon--Vonk antisymmetry conjecture:
\begin{corollary}
Conjecture \ref{mainconjecturetoprove} holds.
\end{corollary}
\begin{proof}
By applying Theorem \ref{proptoproveconjecture} to $(\tau,\omega)$ and to $(\omega,\tau)$ and
Proposition \ref{propforantisymmetry}, we obtain
\[
J_p^\mathrm{T}(\omega,\tau)^{-2}=\widehat{J}_p^\mathrm{T}(\omega,\tau)=\widehat{J}_p^\mathrm{T}(\tau,\omega)^{-1}
=J_p^\mathrm{T}(\tau,\omega)^{2}
\]
up to an element of $E_\tau\cdot E_\omega$. The corollary follows, because if $x^2\in E_\tau\cdot E_\omega$, then $x\in E_\tau\cdot E_\omega$.
\end{proof}
The central idea for the theorem is to construct a function $\eva\colon \mathcal{M}_{\mathrm{rq}}^\times \to \mathcal{M}_{\tau}^\times$,
that fits into a commutative diagram
\begin{equation}\label{commutativediagramm}
\begin{tikzcd}
 \mathcal{M}_{\mathrm{rq}}^\times\arrow[r]{r}{\mathrm{div}}\arrow[d]{d}{\eva}  & \mathrm{Div}^\dagger_{\mathrm{rq}}(\mathcal{H}_p^2) \arrow[d]{d}{\inte} \\
\mathcal{M}_{\tau}^\times \arrow[r]{r}{\mathrm{div}} & \mathrm{Div}^\dagger(X_\tau),
 \end{tikzcd}
\end{equation}
and which satisfies 
\begin{equation}\label{eq:twovariableeva}
\mathrm{ev}_{(\tau,\omega)}=\mathrm{ev}_\tau\circ \eva\colon \mathcal{M}_{\mathrm{rq}}^\times \too \mathbb{C}_p^\times.
\end{equation}
We establish this function in the lemma below:
\begin{lemma}\label{lemwelldefinedeva}
Let $f\in \mathcal{M}_{\mathrm{rq}}^\times\cap \mathcal{A}_{2}$ be a rigid analytic function with rational quadratic divisor. Then the image of $f$ under the homomorphism of rings
\[
\eva\colon \mathcal{A}_{2}\too \mathcal{A}_{1}, \qquad g\mapsto g(-,\omega)
\]
is non-zero.
Therefore, the homomorphism 
\[
\eva\colon \mathcal{M}_{\mathrm{rq}}^\times \too \mathcal{M}_{1}^\times, \qquad f\mapsto f(-,\omega)
\]
is a well-defined homomorphism of groups. 
It fits into the commutative diagram \eqref{commutativediagramm}, and, therefore, its image is contained in $\mathcal{M}_\tau^\times$.
If we identify $\mathcal{M}_{\tau}^\times$ with the $G\times \Gamma_\omega$-module $\mathcal{M}_{\tau}^\times\otimes \mathbb{Z}$, then the map $\eva$ is a $G\times \Gamma_\omega$-module homomorphism, and satisfies the condition in \eqref{eq:twovariableeva}.
\end{lemma}
\begin{proof}
The line
\[
\mathcal{H}_p\times \omega=\lbrace (z,\omega): z\in \mathcal{H}_p\rbrace\subset \mathcal{H}_p^2
\]
is a prime divisor on $\mathcal{H}_p^2$ which is not rational quadratic. If $f(-,\omega)=0$, then $\mathcal{H}_p\times \omega$ is contained in the set of zeros of $f$. Thus, the coefficient of $\mathcal{H}_p\times \omega$ in $\mathrm{div}(f)$ would be non-zero, contradicting the assumption that $\mathrm{div}(f)$ is rational quadratic.   \\
The evaluation map extends to $\mathcal{M}_\mathrm{rq}$: any function $f\in \mathcal{M}_{\mathrm{rq}}^\times$ can be written as $f=g/h$ for $g,h\in \mathcal{M}_{\mathrm{rq}}^\times \cap \mathcal{A}_2$ because the divisor map $\mathrm{div}\colon \mathcal{M}_2^\times \to \mathrm{Div}^\dagger(\mathcal{H}_p^2)$ is surjective. 
In remains to prove the commutativity of the diagram \eqref{commutativediagramm}.
Let $f\in \mathcal{M}_{\mathrm{rq}}^\times$ be a rigid meromorphic function with rational quadratic divisor $\mathcal{D}\coloneq\mathrm{div}(f)$. The proof of \cite[Proposition 3.2]{DGL} implies that $\mathcal{D}$ can be written as a sum $\mathcal{D}=\sum_{m\geq 0}\mathcal{D}_m$ where the formal sums $\mathcal{D}_m=\sum n_{v_i}\cdot \Delta_{v_i,p}$ are genuine divisors. That is, they are finite formal sums. Moreover, they construct certain simple rigid analytic functions $f_v\in \mathcal{M}_{\mathrm{rq}}^\times$ with $\mathrm{div}(f_v)=\Delta_{v,p}$, and prove that the infinite product
\[
f_\mathcal{D}\coloneq\prod_{m\geq 0} f_m \quad \mathrm{with} \quad f_m\coloneq \prod_{i} f_{v_i}^{n_{v_i}}\in \mathcal{M}_{\mathrm{rq}}^\times
\]
converges to a rigid meromorphic function on $\mathcal{H}_p^2$, whose associated divisor is $\mathcal{D}$. Since $\eva$ maps invertible rigid analytic functions on $\mathcal{H}_p^2$ to invertible rigid analytic functions on $\mathcal{H}_p$, it is enough to consider $f=f_{\mathcal{D}}$. In this case, a simple limit argument implies that it suffices to consider $f=f_v$. For these functions it is easy to verify that 
\[
\mathrm{div}\circ \eva(f_v)=v\centerdot \omega=\inte \circ \ \mathrm{div}(f_v).
\]
The commutativity follows, and the identity \eqref{eq:twovariableeva} is obvious by its definition.
\end{proof}
We denote by 
\[
\eva\colon \mathrm{H}^2(\Gamma^2,\mathcal{M}_\mathrm{rq}^\times)\xrightarrow{\res}  \mathrm{H}^2(\Gamma\times \Gamma_\omega^+,\mathcal{M}_\mathrm{rq}^\times) \too \mathrm{H}^2(\Gamma\times \Gamma_\omega^+,\mathcal{M}_\tau^\times \otimes \Z)
\]
the induced map in cohomology.
The $\Gamma_\tau$-module homomorphism $\mathrm{ev}_\tau\colon \mathcal{M}_{\tau}^\times\to \mathbb{C}_p^\times$ induces a homomorphism in cohomology
\[
\mathrm{H}^2(\Gamma\times \Gamma_\omega^+,\mathcal{M}_{\tau}^\times\otimes \Z) \xrightarrow{\res}  \mathrm{H}^2(\Gamma_\tau^+\times\Gamma_\omega^+,\mathcal{M}_{\tau}^\times\otimes \Z) \xrightarrow{\mathrm{ev}_\tau}  \mathrm{H}^2(\Gamma_\tau^+\times\Gamma_\omega^+,\mathbb{C}_p^\times),
\]
and by abuse of notation we denote this map by $\mathrm{ev}_\tau$.
The condition in \eqref{eq:twovariableeva} implies that the maps in cohomology 
\[
\mathrm{ev}_{(\tau,\omega)}= \ \mathrm{ev}_\tau\circ \mathrm{ev}_{(-,\omega)} \colon \mathrm{H}^2(\Gamma^2,\mathcal{M}_{\mathrm{rq}}^\times) \too \mathrm{H}^2(\Gamma_\tau^+\times\Gamma_\omega^+,\mathbb{C}_p^\times)
\]
coincide.
In particular, if $\mathcal{J}\in \mathrm{H}^2(\Gamma^2,\mathcal{M}_{\mathrm{rq}}^\times)$, then
\begin{equation}\label{eq:valuescoincide}
\mathcal{J}[\tau,\omega]=\left(\mathrm{ev}_\tau\circ \eva (\mathcal{J})\right)\cap c_{\tau,\omega}.
\end{equation}
The class $\eva(\mathcal{J})$ is an element of $\mathrm{H}^2(\Gamma\times \Gamma_\omega^+,\mathcal{M}_\tau^\times \otimes \Z)$. Elements of this group decompose:
Since $\mathrm{H}^i(\Gamma_\omega^+,\Z)=0$ for all $i\geq 2$, the K\"unneth Formula implies that the cohomology cross-product
\[
\mathrm{H}^1(\Gamma,\mathcal{M}_\tau^\times)\otimes \mathrm{H}^1(\Gamma_\omega^+,\Z) \ \oplus \ \mathrm{H}^2(\Gamma,\mathcal{M}_\tau^\times)\otimes \mathrm{H}^0(\Gamma_\omega^+,\Z) \too \mathrm{H}^2(\Gamma\times \Gamma_\omega^+,\mathcal{M}_{\tau}^\times\otimes \Z)
\]
is an isomorphism. In other words, every class in $\mathrm{H}^2(\Gamma\times \Gamma_\omega^+,\mathcal{M}_{\tau}^\times\otimes \Z)$ is uniquely of the form
\[
J\times \eta_\omega + R \times \mu_\omega  \quad \text{for } J\in \mathrm{H}^1(\Gamma,\mathcal{M}_\tau^\times), \ R \in  \mathrm{H}^2(\Gamma,\mathcal{M}_\tau^\times),
\]
where $\mu_\omega\in \mathrm{H}^0(\Gamma_\omega^+,\Z)=\Z$ is a generator.
This decomposition can be used to compute the values of these classes at $\tau$:
\begin{lemma}\label{finallemmachater4}
Let $J\in \mathrm{H}^1(\Gamma,\mathcal{M}_{\tau}^\times)$ and $R\in \mathrm{H}^2(\Gamma,\mathcal{M}_\tau^\times)$. Then
\[
\mathrm{ev}_\tau(J\times \eta_\omega) \cap c_{\tau,\omega}=J[\tau]^{-1} \quad \text{and} \quad \mathrm{ev}_\tau(R\times \mu_\omega) \cap c_{\tau,\omega}=1,
\]
where $J[\tau]$ is the $p$-adic number defined in Definition \ref{defvalueii}.
\end{lemma}
\begin{proof}
The statement for the class $R$ follows because $\mathrm{H}^2(\Gamma_\tau^+,\mathcal{M}_\tau^\times)=0$, and the evaluation map is defined by restricting from $\Gamma\times \Gamma_\omega^+$ to $\Gamma_\tau^+\times \Gamma_\omega^+$.\\
The cohomology cross-product is compatible with restriction maps. Thus,
\[
\res(J\times \eta_\omega)=\res(J)\times \eta_\omega 
\]
as classes in $\mathrm{H}^2(\Gamma_\tau^+\times \Gamma_\omega^+,\mathcal{M}_{\tau}^\times\otimes \Z)$. 
The compatibility of the cohomology cross-product with module homomorphisms implies that
\[
\mathrm{ev}_\tau(\res(J)\times \eta_\omega)=\mathrm{ev}_\tau(\res(J)) \times \eta_\omega\in \mathrm{H}^2(\Gamma_\tau^+\times \Gamma_\omega^+,\mathbb{C}_p^\times).
\]
The first identity is now an immediate consequence of the fact that for any class $\mathcal{I}\in \mathrm{H}^1(\Gamma_\tau^+,\mathbb{C}_p^\times)$ we have
\[
(\mathcal{I}\times \eta_\omega)\cap c_{\tau,\omega}=(\mathcal{I}\cap c_\tau)^{-1}.
\]
The exponent $-1$ follows from the formula for the cochain cross-product: we have 
\[
(u_1\times u_2)(c_1\otimes c_2)=(-1)^{\mathrm{deg}(u_1)\cdot \mathrm{deg}(u_2)}\cdot u_1(c_1)\otimes u_2(c_2).
\]
In our case, $\mathrm{deg}(u_1)=1=\mathrm{deg}(u_2)$, giving the exponent in the formula.
\end{proof}
We now apply the above to $\eva(\mathcal{J}_1^\mathrm{T})$.
The commutative diagram \eqref{commutativediagramm} yields the diagram in cohomology
\begin{center}
\begin{minipage}{\linewidth}
\centering
\begin{tikzcd}
 \mathrm{H}^2(\Gamma^2,\mathcal{M}_{\mathrm{rq}}^\times)\arrow[rr]{rr}{\mathrm{div}}\arrow[d]{d}{\eva}&  & \mathrm{H}^2(\Gamma^2,\mathrm{Div}^\dagger_{\mathrm{rq}}(\mathcal{H}_p^2)) \arrow[d]{d}{\inte} \\
\mathrm{H}^2(\Gamma\times \Gamma_\omega^+,\mathcal{M}_{\tau}^\times\otimes \Z) \arrow[rr]{rr}{\mathrm{div}} & & \mathrm{H}^2(\Gamma\times \Gamma_\omega^+,\mathrm{Div}^\dagger\left(X_\tau\right)\otimes \Z).
 \end{tikzcd}
\end{minipage}
\end{center} 
Thus:
\[
\mathrm{div}(\eva(\mathcal{J}_1^\mathrm{T}))=\inte((\mathrm{T}\times 1)\cdot [\mathcal{D}_1]).
\]
Recall that we proved the identity of cocycles
\[
\inte(\mathcal{D}_1)=2\cdot \mathcal{D}_\omega\times \eta_\omega + g\times g_\omega,
\]
where $g$ is a $2$-cocycle of $\Gamma$ and $g_\omega$ a $0$-cocycle of $\Gamma_\omega^+$, in Lemma \ref{lemma:identitxoftwo} and at the end of Section \ref{sec:eilenberg}.
All divisor-valued cocycles of this equality take values in the group $\mathrm{Div}^\dagger(X_\tau)\otimes \Z$ because both are supported on $G\centerdot \omega \subset X_\tau$. 
Hence, we also obtain the identity in Theorem \ref{mainthm} with values in $\mathrm{Div}^\dagger(X_\tau)\otimes \Z$:
\[
\inte((\mathrm{T}\times 1)\cdot [\mathcal{D}_1])\overset{\ref{mainthm}}{=} 2\cdot \mathrm{T}\cdot [\mathcal{D}_\omega]\times \eta_\omega + D \times \mu_\omega=\mathrm{div}\left((\mathcal{J}_\omega^\mathrm{T})^{2}\times \eta_\omega\right)+ D\times \mu_\omega,
\]
for $D=\mathrm{T}\cdot[g]\in \mathrm{H}^2(\Gamma,\mathrm{Div}^\dagger(X_\tau))$.
In conclusion, the above yields:
\[
\mathrm{div}(\eva(\mathcal{J}_1^\mathrm{T}))=\mathrm{div}\left((\mathcal{J}_\omega^\mathrm{T})^{2}\times \eta_\omega\right)+ D\times \mu_\omega
\]
Let $\eva\left(\mathcal{J}_{1}^{\mathrm{T}}\right)=J\times \eta_\omega + R\times \mu_\omega$ be the K\"unneth decomposition.
Then:
\[
\mathrm{div}(J)\times \eta_\omega+\mathrm{div}(R)\times \mu_\omega= \mathrm{div}\left((\mathcal{J}_\omega^\mathrm{T})^{2}\times \eta_\omega\right)+ D\times \mu_\omega.
\]
Since the K\"unneth decomposition is a direct sum, and tensoring with $\mathrm{H}^i(\Gamma_\omega^+,\Z)$ is injective for $i=0,1$, this implies
\[
J=(\mathcal{J}_\omega^\mathrm{T})^{2} \ \text{ mod } \mathrm{H}^1(\Gamma,\mathcal{A}_1^\times).
\]
This equality, Equality \eqref{eq:valuescoincide} and Lemma \ref{finallemmachater4} yields:
\begin{corollary}\label{maincorofrommainpropchapter4}
We have
\[
\widehat{J}_p^\mathrm{T}(\tau,\omega)=J_p^\mathrm{T}(\tau,\omega)^{-2}
\]
up to multiplication with a $p$-adic number in $\mathrm{ev}_\tau\left(\mathrm{H}^2(\Gamma\times \Gamma_\omega^+,\mathcal{A}_{1}^\times \otimes \Z)\right)\cap c_{\tau,\omega}\cdot E_\tau$.
\end{corollary}
\begin{proof}
Recall that $\widehat{J}_p^\mathrm{T}(\tau,\omega)=\mathcal{J}_1^{\mathrm{T}}[\tau,\omega]$ up to torsion, and that $J_p^\mathrm{T}(\tau,\omega)=\mathcal{J}_\omega^\mathrm{T}[\tau]$ up to $E_\tau$.
By the equality above the corollary, we have  
\[
\eva\left(\mathcal{J}_{1}^{\mathrm{T}}\right)=\left(\mathcal{J}_\omega^\mathrm{T}\right)^{2}\times \eta_\omega+ R\times \mu_\omega \  \text{ mod } \ \mathrm{H}^2(\Gamma\times \Gamma_\omega^+,\mathcal{A}_{1}^\times \otimes \Z). 
\]
We evaluate both sides at $\tau$:
\begin{align*}
\mathcal{J}_1^\mathrm{T}[\tau,\omega]&\overset{\eqref{eq:valuescoincide}}{=} \mathrm{ev}_\tau\circ \eva(\mathcal{J}_1^\mathrm{T})\cap c_{\tau,\omega}\\
& \ =\mathrm{ev}_\tau\left( \left(\mathcal{J}_\omega^\mathrm{T}\right)^{2}\times \eta_\omega\right)\cap c_{\tau,\omega} \ \cdot \ \mathrm{ev}_\tau(R\times \mu_\omega)\cap c_{\tau,\omega}\\
&\overset{\ref{finallemmachater4}}{=} \mathcal{J}_\omega^\mathrm{T}[\tau]^{-2} \cdot 1.
\end{align*}
This is the desired identity, which holds up to $\mathrm{ev}_\tau\left(\mathrm{H}^2(\Gamma\times \Gamma_\omega^+,\mathcal{A}_{1}^\times \otimes \Z)\right)\cap c_{\tau,\omega}$, because the first identity only holds up to $\mathrm{H}^2(\Gamma\times \Gamma_\omega^+,\mathcal{A}_{1}^\times \otimes \Z)$.
\end{proof}
Finally, the desired Theorem \ref{proptoproveconjecture} follows by computing the ambiguity in the corollary above: 
\begin{corollary}\label{corollaryanalytic22}
We have
\[
\mathrm{ev}_\tau\left(\mathrm{H}^2(\Gamma\times \Gamma_\omega^+,\mathcal{A}_{1}^\times\otimes \Z)\right)\cap c_{\tau,\omega}\subset  E_\tau.
\]
In particular, we have
\[
\widehat{J}_p^\mathrm{T}(\tau,\omega)=J_p^\mathrm{T}(\tau,\omega)^{-2} \text { mod } E_\tau,
\]
i.e. ~Theorem \ref{proptoproveconjecture} follows.
\end{corollary}
\begin{proof}
Let $\mathcal{J}=J\times \eta_\omega+  R\times \mu_\omega$ with classes $J\in \mathrm{H}^1(\Gamma,\mathcal{A}_{1}^\times)$, $R\in \mathrm{H}^2(\Gamma,\mathcal{A}_1^\times)$.
Lemma \ref{finallemmachater4} implies:
\[
\mathrm{ev}_\tau(\mathcal{J})\cap c_{\tau,\omega}=\mathrm{ev}_\tau(J\times \eta_\omega)\cap c_{\tau,\omega}=J[\tau]^{-1}.
\]
The $p$-adic number $J[\tau]$ is contained in the set $E_\tau$ by Equation \eqref{eq:trivialvalues}, and the corollary follows.
\end{proof}

\bibliographystyle{abbrv}
\bibliography{bibfile}

\end{document}